\documentclass{amsart}

\usepackage{color}
\usepackage{soul}
\usepackage[pagewise]{lineno}

\usepackage{hyperref}
\hypersetup{
    colorlinks=true, 
    linktoc=all,     
    linkcolor=blue,  
}

\newcommand{\Q}{\mathbb{Q}}
\newcommand{\C}{\mathbb{C}}

\newcommand{\Z}{\mathbb{Z}}

\newcommand{\N}{\mathbb{N}}

\newcommand{\dom}{\operatorname{dom}}
\newcommand{\ran}{\operatorname{ran}}

\newcommand{\norm}[1]{\left\| #1 \right\|}
\newcommand{\Proj}{\operatorname{Proj}}
\newcommand{\unit}{\mathbf{1}}
\renewcommand{\star}{\ast}

\newcommand{\B}{\mathcal{B}}

\newcommand{\boldA}{\mathbf{A}}
\newcommand{\boldB}{\mathbf{B}}
\newcommand{\boldC}{\mathbf{C}}
\newcommand{\boldF}{\mathbf{F}}
\newcommand{\boldG}{\mathbf{G}}
\newcommand{\zerovec}{\mathbf{0}}

\newcommand{\cstar}{$\mathrm{C}^*$}
\newcommand{\BratD}{\mathcal{D}}
\newcommand{\calC}{\mathcal{C}}
\newcommand{\calE}{\mathcal{E}}

\newcommand{\calG}{\mathcal{G}}
\newcommand{\calH}{\mathcal{H}}

\newcommand{\calS}{\mathcal{S}}
\newcommand{\Id}{\operatorname{Id}}
\newcommand{\Ad}{\operatorname{Ad}}
\newcommand{\embed}{\mathcal{E}}
\newcommand{\uK}[1]{K_0^{\operatorname{sc}}(#1)}
\newcommand{\uKnoarg}{K_0^{\operatorname{sc}}}
\newcommand{\Cat}{\mbox{\bf Cat}}
\newcommand{\Tot}{\mbox{\bf Tot}}

\theoremstyle{plain}
\newtheorem{theorem}{Theorem}[section]
\newtheorem{lemma}[theorem]{Lemma}
\newtheorem{corollary}[theorem]{Corollary}
\newtheorem{proposition}[theorem]{Proposition}

\newtheorem{question}[theorem]{Question}

\newtheorem{maintheorem}{Main Theorem}

\theoremstyle{definition}
\newtheorem{definition}[theorem]{Definition}

\numberwithin{equation}{section}

\begin{document}
\title{Computable $K$-theory for \cstar-algebras II:  AF algebras}
\author[C. J. Eagle]{Christopher J. Eagle${}^1$} 
\address[C. J. Eagle]
{Department of Mathematics and Statistics, University of Victoria. PO BOX 1700 STN CSC, Victoria, British Columbia, Canada. V8W 2Y2}%
\email{eaglec@uvic.ca}
\urladdr{http://www.math.uvic.ca/~eaglec}
\thanks{${}^1$ Supported by NSERC Discovery Grant RGPIN-2021-02459}

\author[I. Goldbring]{Isaac Goldbring${}^2$}
\address[I. Goldbring]{University of California, Irvine}
\email{igoldbri@uci.edu}
\urladdr{https://www.math.uci.edu/~isaac}
\thanks{${}^2$  Supported by NSF grant DMS-2054477.}

\author[T. H. McNicholl]{Timothy H. McNicholl}
\address[T. H. McNicholl]{Iowa State University}
\email{mcnichol@iastate.edu}
\urladdr{https://faculty.sites.iastate.edu/mcnichol/}


\begin{abstract}
We continue the study of the effective content of $K$-theory for \cstar-algebras, with a focus on AF algebras.  We show that from a c.e. presentation of an AF algebra it is possible to compute a representation of the algebra as an inductive limit of finite-dimensional algebras.  Using this, and an analogous result for dimension groups, we show that the computable $K_0$ functor provides a computable equivalence of categories between c.e. presentations of AF algebras and c.e. presentations of unital (scaled) dimension groups, giving an effective version of Elliott's classification theorem.  We use our results to determine the complexity of the index set and isomorphism problems for various classes of AF algebras.
\end{abstract}
\maketitle

\tableofcontents

\section{Introduction}\label{sec:Intro} 

In our previous work \cite{UHFPaper} with R. Miller, we initiated the effective study of $K$-theory for \cstar-algebras.  More precisely, we established a computable functor which associated, to each computably enumerable (c.e.) presentation $\boldA^\#$ of a \cstar-algebra $\boldA$, a c.e. presentation $K_0(\boldA^\#)$ of its $K_0$ group $K_0(\boldA)$.  In addition, this functor maps 
computable $\star$-homomorphisms to computable group homomorphisms.  These statements are 
made precise in Section \ref{subsubsec:BackCompKThy} below.  
A similar such computable functor was also defined for $K_1$ as well.  As a proof of concept, the special case of the $K_0$ functor restricted to \emph{uniformly hyperfinite (UHF)} algebras was completely analyzed.

In this paper, we broaden our scope and consider our computable $K_0$ functor restricted to the class of \emph{approximately finite-dimensional (AF)} algebras, which are the \cstar-algebras isomorphic to those obtained as inductive limits of finite-dimensional \cstar-algebras.  Every AF algebra $\boldA$ has $K_1(\boldA) = 0$, so the $K$-theory for AF algebras is entirely focused on $K_0$.  Since AF algebras are stably finite, their $K_0$ group has a natural ordering which renders it a (partially) ordered abelian group.  
In addition, for a c.e. presentation $\boldA^\#$, the positive cone of this group, 
which we denote $K_0(\boldA)^+$, is a c.e. set of the presentation $K_0(\boldA^\#)$; 
accordingly, we refer to $(K_0(\boldA^\#), K_0(\boldA)^+)$ as the \emph{ordered presented group 
of $\boldA^\#$}.
The ordered abelian groups that arise as ordered $K_0$ groups of AF algebras are known as \emph{dimension groups}.
By the Effros-Handelman-Shen Theorem, these are precisely the ordered abelian groups isomorphic to an inductive limit of \emph{simplicial groups}, that is, groups of the form $\mathbb{Z}^n$ with their canonical ordering.  By distinguishing the image $K_0(\unit_\boldA)$ of the unit of $\boldA$ in $K_0(\boldA)$, one views $K_0(\boldA)$ as a \emph{unital (or scaled) group}, that is, as an ordered abelian group with a distinguished order unit.  Elliott's classification theorem \cite{Elliott.1976} states that AF algebras are characterized, up to isomorphism, by their unital ordered $K_0$ groups.  In fact, the $K_0$ functor yields an equivalence of categories between the category of AF algebras, with approximate unitary equivalence classes of $*$-algebra homomorphisms as morphisms, and the category of unital dimension groups, with order-unit preserving positive group homomorphisms as morphisms.

While $K_0$ groups provide an algebraic representation of AF algebras, Bratteli diagrams provide a combinatorial method of representing AF algebras, namely as 
infinite multi-graphs arranged in a sequence of levels.  
Labeled Bratteli diagrams assign numbers to the vertexes, and every AF algebra is completely described 
by a labeled Bratteli diagram.  

Our first overarching goal is to classify the c.e. presentations of AF algebras.  
To this end, we introduce the concept of an \emph{AF certificate} (defined in Section \ref{subsubsec:BackAF}
below).  These essentially describe an AF algebra as a specific inductive limit of finite dimensional 
algebras.  Our first main result is the following.

\begin{maintheorem}\label{mthm:CompAFCert} 
If a unital AF algebra has a c.e. presentation, then it also has a computable AF certificate. 
\end{maintheorem}

In classical mathematics, when an AF algebra is considered, a representation of it 
as an inductive limit of finite-dimensional algebras is assumed.  However, a c.e. 
presentation of an AF algebra, which merely provides information about the 
norm on a certain dense set, does not in and of itself readily provide such a description.  
Rather, from the information provided by a c.e. presentation of an AF algebra $\boldA$, one must algorithmically discern the finite-dimensional $\star$-subalgebras of $\boldA$ (which are not the same
as the finitely generated subalgebras).  Locating these $\star$-subalgebras turns out to be no 
mean feat, and as an intermediate step we go to great lengths to first prove an effective 
version of Glimm's Lemma (Theorem \ref{thm:EffGlimm} below).  We note that our proof of Main Theorem \ref{mthm:CompAFCert} is uniform in that it provides an algorithm for producing AF certificates
from c.e. presentations of AF algebras.

Our second main result, which will also contribute to our first overarching goal, 
is an effective version of the Effros-Handelman-Shen Theorem referred to above.

\begin{maintheorem}\label{mthm:ClsCePresDimGrp} 
Suppose $\calG^\#$ is a c.e. presentation of a dimension group.
\begin{enumerate}
    \item There is a c.e. presentation $\boldA^\#$ of an AF algebra so that 
the presented ordered $K_0$ group of $\boldA^\#$
is computably isomorphic to $\calG^\#$.\label{mthm:ClsCePresDimGrp::NotUntl}

    \item If $u$ is an order unit of $\calG$, 
    then there is a c.e. presentation $\boldA^\#$ of a
    unital AF algebra so that the presented unital ordered $K_0$ group of $\boldA^\#$
is computably isomorphic to $(\calG, u)^\#$.\label{mthm:ClsCePresDimGrp::Untl}
\end{enumerate}
\end{maintheorem}

Again, the proof is uniform.  With these two results in hand, we then 
achieve our goal of classifying the c.e. presentations of unital AF algebras.

\begin{maintheorem}[Classification of computably presentable unital AF algebras]\label{mthm:ClsCompPrsntAF}
Suppose $\boldA$ is a unital AF algebra. Then, the following are equivalent.
\begin{enumerate}
    \item $\boldA$ has a c.e. presentation. \label{mthm:ClsCompPrsntAF::CePres}

    \item There is a computable inductive sequence 
    of finite dimensional algebras whose inductive limit is $\star$-isomorphic
    to $\boldA$ and so that each 
    bonding map is unital.\label{mthm:ClsCompPrsntAF::CompIndSeqFinDimAlg}

    \item $\boldA$ is computably presentatable. \label{mthm:ClsCompPrsntAF::CompPres}

    \item The unital ordered $K_0$ group of $\boldA$ is computably presentatable. \label{mthm:ClsCompPrsntAF::CompPresUntlOrdAbl}

    \item $\boldA$ has a computably presentable labeled Bratteli diagram.\label{mthm:ClsCompPrsntAF::CompBrat}
\end{enumerate}
\end{maintheorem}

Once again, the proof is uniform.  Our next overarching goal is to strengthen this 
classification result by showing that computably isomorphic objects 
map to computably isomorphic objects.  This is the content of the following.

\begin{maintheorem}\label{mthm:CompIso}
Suppose $\boldA^\#$ and $\boldB^\#$ are c.e. presentations of unital AF algebras.  
Then, the following are equivalent.
\begin{enumerate}
    \item $\boldA^\#$ is computably $\star$-isomorphic to $\boldB^\#$.\label{mthm:CompIso::Pres}

    \item The presented unital ordered $K_0$ groups of $\boldA^\#$ and $\boldB^\#$ are
    computably isomorphic.\label{mthm:CompIso::Grp}

    \item Every Bratteli diagram of a computable AF certificate of 
    $\boldA^\#$ is computably equivalent to 
    every Brattelii diagram of a computable AF certificate of $\boldB^\#$. \label{mthm:CompIso::Brat}
\end{enumerate}
\end{maintheorem}

Not surprisingly, the proof is uniform.  Our results so far suggest that 
it should be possible to demonstrate an effective version of Elliott's 
equivalence of categories, and in fact it is:

\begin{maintheorem}\label{mthm:CompEqvCat}
The following categories are computably equivalent:
\begin{enumerate}
    \item The category of c.e. presentations of unital AF algebras.

    \item The category of c.e. presentations of unital dimension groups.
\end{enumerate}
\end{maintheorem}

We formally define the terminology in this theorem in Section \ref{sec:CompEqvCat} below; for now
let us say that computable equivalence means the required functors and natural isomorphisms 
can be computed.  The proof is very much based on the uniformity of the proofs of the prior 
theorems.

Our main theorems have several interesting consequences.  For one, 
every c.e. presentation of a unital AF algebra is computable (Corollary \ref{cor:CompPresAF}). 
Also, in contrast to UHF algebras, there is an AF algebra that is not computably categorical
(Corollary \ref{cor:NotCompCatAF}).  We can use this observation to prove 
similar statements for dimension groups and Bratteli diagrams. 
Finally, we use our results to determine the complexity of the index set and isomorphism 
problems for AF and UHF algebras.

This paper brings together material from several areas of mathematics: computability theory, 
category theory, group theory, and the theory of \cstar-algebras. 
In order to make the presentation fairly self-contained and navigable, we have chosen 
the following organization.  We first cover background material from these areas; this consists 
of information that is already in the literature but which we attempt to organize and summarize
so that the reader need not constantly consult an array of other sources. 
We divide this coverage into two components: the background from classical mathematics
(Section \ref{sec:BackClscl}) and from effective mathematics (Section \ref{sec:BackEff}).  
The latter presents a picture of the 
computability theory of the topics covered in the former with a parallel organization. 
We then attend to the development of preliminary matters from classical and effective mathematics 
(Sections \ref{sec:PrlmClscl} and \ref{sec:PrlmEff}). 
The Effective Glimm Lemma is proven in Section \ref{sec:EffGlimm}, and our main theorems are proven in Sections \ref{sec:CompAFCert} through \ref{sec:CompEqvCat}.
Section \ref{sec:IndxIso} contains our results on index sets

\section{Background from classical mathematics}\label{sec:BackClscl}

In this section, we summarize relevant background material from four 
subjects which will be intertwined throughout the rest of this paper: inductive limits in 
categories, ordered abelian groups, \cstar-algebras, and Bratteli diagrams.
We begin with inductive limits as this topic supports much of our work with regards to 
the second and third of these topics. 

\subsection{Inductive limits in categories}\label{subsec:BackIndLimCat}

Inductive limits are ubiquitous in mathematics by virtue of their ability to 
construct large objects from chains of smaller ones.  Here, we discuss how 
category theory provides a 
framework for unifying these constructions so that one need not repeat 
definitions and arguments that are essentially identical.
Accordingly, we now assume $\calC$ denotes a category with a zero object.

Suppose that for each $n \in \N$, $A_n$ is an object of $\calC$ 
and $\phi_n$ is a morphism from $A_n$ to $A_{n+1}$. 
The sequence $(A_n, \phi_n)_{n \in \N}$ is called an
\emph{inductive sequence of $\calC$}.
Through composition, these morphisms yield morphisms 
$\phi_{n, n'}$ from $A_n$ to $A_{n'}$ whenever $n' > n$.
Additionally, we define $\phi_{n,n}$ to be the identity morphism of $A_n$, and when $\ell < n$, let $\phi_{n, \ell}$ denote the zero morphism.

Now suppose $A$ is an object of $\calC$, and assume that for each $n \in \N$, 
$\nu_n$ is a morphism from $A_n$
to $A$ so that $\nu_n = \nu_{n + 1} \circ \phi_n$.  
We then refer to $(A, (\nu_n)_{n \in \N})$ as an
\emph{inductive upper limit} of $(A_n, \phi_n)_{n \in \N}$.
If $(B, (\mu_n)_{n \in \N})$ is also an inductive upper 
limit of 
$(A_n, \phi_n)_{n \in \N}$, then a \emph{reduction} of $(A, (\nu_n)_{n \in \N})$ to
$(B, (\mu_n)_{n \in \N})$ is a morphism $\lambda$ from $A$ to $B$ so that 
$\lambda \circ \nu_n = \mu_n$ for all $n \in \N$.
We say that $(A, (\nu_n)_{n \in \N})$ is an \emph{inductive limit} of 
$(A_n, \phi_n)_{n \in \N}$ if for every inductive upper limit $(B, (\mu_n)_{n \in \N})$ of 
$(A_n, \phi_n)_{n \in \N}$, there is a unique reduction from $(A, (\nu_n)_{n \in \N})$ 
to $(B, (\mu_n)_{n \in \N})$.  The following is a fairly direct consequence of the definitions.

\begin{proposition}\label{prop:IndLimUni}
Suppose $(A, (\nu_n)_{n \in \N})$ and $(B, (\mu_n)_{n \in \N})$
are inductive limits of $(A_n, \phi_n)_{n \in \N}$.  
Then the reduction from $(A, (\nu_n)_{n \in \N})$ to $(B, (\mu_n)_{n \in \N})$ is an isomorphism.
\end{proposition}

As a result of the previous proposition, one may then speak of \emph{the} inductive limit of an inductive sequence (when it exists).

\subsection{Ordered abelian groups}\label{subsec:OrdAbl}
We now move on to discussing the second of our four topics, ordered abelian groups.
Our goal here is to summarize the pertinent concepts and results from the theory 
of ordered abelian groups.  We refer the reader to standard sources such as 
\cite{Goodearl.1986} for more comprehensive treatments.

\subsubsection{Basic definitions and principles}

Fix an abelian group $G$ which we write additively and whose identity element 
we denote $\zerovec_G$. 
Traditionally, the ordered abelian groups studied by computability theorists
are linearly ordered.  However, the investigation of AF algebras necessitates 
the consideration of partial orders that are consistent with the group operation. 
The most direct way to establish such an order on $G$ is via a \emph{positive cone}
which is defined to be a subset $P$ of $G$ that satisfies the following:
\begin{enumerate}
	\item $P$ is closed under addition. 
	
	\item $P \cap (-P) = \{\zerovec_G\}$.
	
	\item $P + (-P) = G$.
\end{enumerate}
Suppose $P$ is a positive cone of $G$.
This positive cone defines a partial order $\leq$ on $G$ so that
$x \leq y$ if and only if $y - x \in P$. 
Accordingly, we call the pair $(G,P)$ an \emph{ordered abelian group}.

We now set $\calG = (G, P)$; we may sometimes write $\calG^+$ instead of $P$. 
Suppose $\calG' = (G', P')$ is an ordered abelian group.  
We say that a map $\phi$ is a \emph{homomorphism} from $\calG$ to $\calG'$ if $\phi$ 
is a homomorphism from $G$ to $G'$ (as groups) so that $\phi[P] \subseteq P'$; we may also say that $\phi$ is a 
\emph{positive} morphism.  The ordered abelian groups form a category in which the morphisms are these
homomorphisms.
We refer to $\calG$ as a \emph{subgroup} of $\calG'$ if $G$ is a subgroup of 
$G'$ and $P \subseteq P'$.  
If $\calG$ is a subgroup of $\calG'$, 
and if every element of $\calG'$ between $g$ and $g'$ belongs to $\calG$ whenever $g$ and $g'$ belong 
to $\calG$,
then we say that $\calG$ is \emph{convex}.

A positive element $u$ of $\calG$ is called an \emph{order unit (or a scale) of $\calG$} if 
for every $g \in \calG$ there exists a positive integer $n$ so that 
$-n u \leq g \leq n u$. 
If $u$ is an order unit of $\calG$, then $(\calG, u)$ is referred to as 
a \emph{unital ordered abelian group}.
When $(\calG, u)$ and $(\calH, v)$ are unital ordered abelian groups, we say that 
a map $\phi$ is a \emph{homomorphism} from $(\calG,u)$ to $(\calH,v)$ if 
it is a homomorphism from $\calG$ to $\calH$ for which $\phi(u) = u'$.
The unital ordered abelian groups form a category in which the morphisms are these 
homomorphisms.

The following principle will be very useful in our proof of Main Theorem \ref{mthm:ClsCePresDimGrp}.
We include a proof for the sake of completeness.

\begin{lemma}\label{lm:UntlCnvx}
Suppose $u$ is a positive element of an ordered abelian group $\calG$. 
Then $u$ is an order unit of the convex subgroup of $\calG$ that is generated by $u$.
\end{lemma}

\begin{proof}
Suppose $\calG = (G, P)$.  Let $\calH$ denote the convex subgroup of $\calG$ that is generated by $u$.  
Set
\[
H' = \bigcup_{n \in \N} [-n u, n u].
\]
Let $\calH' = (H', H' \cap P)$.  It follows that 
$\calH'$ is a subgroup of $\calG$.  

We claim that $\calH'$ is convex.  To see this, let $g,g' \in H'$.  
Assume $h \in G$ and $g \leq h \leq g'$.  
There exist $n,n' \in \N$ so that $-n u \leq g \leq n u$ and 
so that $-n' u \leq g' \leq n' u$.  Thus, $h \leq n' u \leq (n + n')u$
and $h \geq -n u \geq -(n + n') u$.
Thus, $\calH'$ is convex. Hence, $\calH$ is a subgroup of $\calH'$ and the lemma is proven.
\end{proof}

\subsubsection{Simplicial groups}

Fix a positive integer $n$.  There is a natural positive cone on the group
$\Z^n$, namely the set of all $n$-tuples of nonnegative integers.
We denote this positive cone by $\Z^n_{\geq 0}$, and we identify $\Z^n$ 
with $(\Z^n, \Z^n_{\geq 0})$.  These ordered groups are referred to as 
\emph{simplicial} and they are the building blocks for many of the ordered groups we 
wish to consider.

We also establish a bit of notation.  We let $e^n_1, \ldots, e^n_n$ denote the standard basis elements of $\Z^n$ and when $h \in \Z^n$ and $g = e^n_j$, we set $\pi_g(h) = h(j)$.

\subsubsection{Inductive limits}

It is fairly well-known 
that every inductive sequence of abelian groups has an inductive limit; see the proof of 
Lemma \ref{lm:IndLimAbl} below.
Similarly, every inductive sequence in the category of ordered abelian groups has an inductive limit; see, for example, 
\cite[Proposition 6.2.6]{Rordam.Larsen.Laustsen.2000}.  A similar argument establishes that every inductive 
sequence in the category of unital ordered abelian groups has an inductive limit.

\subsubsection{Dimension groups}
We have now established sufficient material for a formal discussion of dimension groups, which are 
treated extensively in \cite{Durand.Perrin.2022} and in \cite{Goodearl.1986}.
Fix an ordered abelian group $\calG$.  
A \emph{certificate of dimensionality} of $\calG$ is a sequence $(n_s, \phi_s, \nu_s)_{s \in \N}$, so 
that 
$(\calG, (\nu_s)_{s \in \N})$ is an inductive limit of $(\Z^{n_s}, \phi_s)_{s \in \N}$ in the category of ordered abelian groups.
A \emph{dimension group} is an ordered abelian group that has a certificate of dimensionality.

The Shen property, which we define momentarily, is a key property of dimension groups.
The definition is based on the concept of a factoring of a homomoprhism, 
which we define as follows.  
Suppose $\theta : \Z^n \rightarrow G$, $\phi : \Z^n \rightarrow \Z^p$, and $\theta' : \Z^p \rightarrow G$
are group homomorphisms and let $\alpha \in \ker(\theta)$.  
We say that $(\phi, \theta')$ \emph{factors $\theta$ at $\alpha$} if 
$\alpha \in \ker(\phi)$ and $\theta = \theta'\circ\phi$. 

Suppose $\calG$ is an ordered abelian group. 
We then say that $\calG$ has the \emph{Shen Property} if 
for every homomorphism $\theta : \Z^n \rightarrow \calG$ 
and every $\alpha \in \ker(\theta)$, there is a pair of ordered group
homomorphisms $\phi : \Z^n \rightarrow \Z^p$ and $\theta' : \Z^p \rightarrow G$ that factors $\theta$ at $\alpha$.  

We will need the following fact (see \cite[Lemma 3.4.4]{Durand.Perrin.2022}):

\begin{proposition}\label{prop:ShenPrp}
Every dimension group has the Shen property.
\end{proposition}

\subsection{\cstar-algebras}
We now proceed to the third of our four background topics: \cstar-algebras.
We assume the reader is familiar with the basic definitions and results 
of the theory of \cstar-algebras which can be found in \cite{Davidson.1996}.
We begin by reviewing some additional general background information before
proceeding to the more specific topics of $K$-theory, matrix algebras, finite-dimensional algebras, 
and AF algebras.  We conclude this section with a collection of perturbation principles
which will be very useful in our proof of the Effective Glimm Lemma alluded to in the
introduction.

\subsubsection{General background}

If $U$ is a subset of a metric space $X$, and if $\epsilon > 0$, we let 
$B(U; \epsilon) = \{p \in X\ :\ d(p, U) < \epsilon\}$.  

Let $\boldA$ be a \cstar-algebra.  
If $\boldA$ is unital, we let $\unit_\boldA$ denote the unit of $\boldA$.
Let $\Proj(\boldA)$ denote the set of all projections of $\boldA$.
We note that, by the \cstar-identity, every element of $\Proj(\boldA)$ has norm $0$ or $1$.  
We say that $v \in \boldA$ is a \emph{partial isometry} if $v^*v \in \Proj(\boldA)$.   
Partial isometries will play a role in the perturbation principles we discusss below
and hence in the proof of the Effective Glimm Lemma.

One of the consequences of the \cstar-identity is that every $\star$-homomorphism of \cstar-algebras 
is $1$-Lipschitz.  This principle simplifies many arguments concerning the 
computability of these maps.

When $u \in \boldA$ is unitary, we let  $\Ad_u:\boldA\to \boldA$ denote the map given by $\Ad_u(a) = uau^*$ for all $a \in \boldA$.
These maps will prove advantageous when classifying $\star$-homomorphisms of 
finite-dimensional algebras.

An $n \times n$ array $(a_{r,s})_{r,s}$ of self-adjoint elements of $\boldA$ is an
\emph{$n \times n$ system of matrix units of $\boldA$} if $a_{r,s}a_{r',s'} = \delta_{s,r'} a_{r,s'}$
for all $r,s,r',s' \in \{1, \ldots, n\}$. 
If $\boldA$ is unital, then we call such a system \emph{unital} if $\sum_j a_{j,j} = \unit_\boldA$. 
The construction of these systems plays a key role in the computability theory of finite-dimensional algebras and AF algebras.

Suppose $\boldB$ is a \cstar-algebra.  
The direct sum $\boldA \bigoplus \boldB$ is also a \cstar-algebra; in particular, 
the norm is defined by $\norm{(a,b)} = \max\{\norm{a}, \norm{b}\}$.  
This algebra is also denoted by $\boldA \times \boldB$. 

Assume $\phi$ and $\psi$ are $\star$-homomorphisms 
from $\boldA$ to $\boldB$.  We say that $\phi$ and $\psi$ are 
\emph{approximately unitarily equivalent} if there is a sequence $(u_n)_{n \in \N}$ of 
unitaries of $B$ so that $(\Ad_{u_n} \circ \phi)_{n \in \N}$ converges pointwise to $\psi$.
We will rely on this equivalence relation when we discuss AF algebras as a category in the proof 
of Main Theorem \ref{mthm:CompEqvCat}.

\subsubsection{$K$-theory}\label{subsubsec:BackKthy}

We now provide a quick synopsis of the fundamental principles of $K$-theory which are 
expounded much more completely in \cite{Rordam.Larsen.Laustsen.2000}.  
The reader may also consult our previous paper \cite{UHFPaper} 
for an explanation of just the definitions and results we rely on here.
To start, $K$-theory defines a functor $K_0$ from the category of \cstar-algebras 
to the category of abelian groups.  If $\boldA$ is a separable \cstar-algebra, 
then $K_0(\boldA)$ is countable. 
If $\boldA$ is stably finite\footnote{Due to its complexity, we demur from precisely defining `stably finite'; the curious reader may find a definition in standard sources such as \cite{Rordam.Larsen.Laustsen.2000}.  Suffice it to say, AF algebras are stably finite.}, then $K$-theory also associates to $\boldA$ 
a canonical positive cone of $K_0(\boldA)$, which is denoted $K_0(\boldA)^+$.  
If $\boldA$, $\boldB$ are \cstar-algebras and
$\phi$ is a $\star$-homomorphism from $\boldA$ to $\boldB$, 
then $K_0(\phi)$ is a homomorphism from 
$K_0(\boldA)$ to $K_0(\boldB)$; if, in addition, $\boldA$ and $\boldB$ are stably finite, then $K_0(\varphi)$ is an ordered group homomorphism from $(K_0(\boldA),K_0(\boldA)^+)$ to $(K_0(\boldB),K_0(\boldB)^+)$.  
If $\boldA$ is stably finite and unital, then $(K_0(\boldA), K_0(\boldA)^+)$ has an order unit;  
in fact, $K$-theory associates $\boldA$ with a particular order unit of $(K_0(\boldA), K_0(\boldA)^+)$,
which we denote by $K_0(\unit_\boldA)$ (since it is in fact obtained from the unit of $\boldA$). 
Furthermore, if $\boldA$ and $\boldB$ are stably finite and unital, and if $\phi$ is a unital $\star$-homormorphism 
from $\boldA$ to $\boldB$, then $K_0(\phi)(K_0(\unit_\boldA)) = K_0(\unit_\boldB)$.
Hence, \it we have a functor from the category of stably finite unital \cstar-algebras (in which the 
morphisms are assumed to be unital) to the 
category of unital ordered abelian groups\rm; we denote this functor by $\uKnoarg$ 
($\operatorname{sc}$ stands for \emph{scaled} as the triple consisting of an ordered abelian group together with an order unit is often called a \emph{scaled group}).

\subsubsection{Matrix algebras}\label{subsubsec:BackMtrx}
We now discuss matrix algebras, which are the building blocks of AF algebras.
Fix a positive integer $n$.  
Let $M_n(\C)$ denote the \cstar-algebra of $n \times n$ complex matrices, which is 
a \cstar-algebra under the operator norm and in which the involution is the adjoint operation.

Suppose that $\ell$ is also a positive integer.  It is well-known that there is a nonzero $\star$-homomorphism
from $M_n(\C)$ to $M_\ell(\C)$ if and only if $\ell \geq n$ (see, for example, \cite{Davidson.1996}).  
Suppose $\ell \geq n$.  For $k$ a positive integer for which
$kn \leq \ell$ define, for each $A \in M_n(\C)$, $\embed_{n,\ell,k}(A)$ 
to be the $\ell \times \ell$ matrix obtained by repeating $A$ $k$ times along the 
diagonal.  It is well-known that if $\psi : M_n(\C) \rightarrow M_\ell(\C)$ is a nonzero 
$*$-homomorphism, then there is a unique positive integer $k$ and a unitary $U \in M_\ell(\C)$
so that $\psi(A) = \Ad_U(\embed_{n,\ell,k}(A))$ (again, see \cite{Davidson.1996}).
The number $k$ is called the 
\emph{multiplicity} of $\psi$.  

We now discuss the $K$-theory of these algebras.  To begin, there is a standard isomorphism $\zeta_n$ from 
$\uK{M_n(\C)}$
to $(\Z, n)$ (that is, the ordered group $\Z^n$ together with the order unit $n$).  
If $\phi$ is a homomorphism from $\uK{M_n(\C)}$ to $\uK{M_\ell(\C)}$, 
then there is a canonical homomorphism 
$\eta_{n, \ell}(\phi)$ from $(\Z,n)$ to $(\Z, \ell)$ so that 
$\eta_{n, \ell}(\phi) \circ \zeta_n = \zeta_\ell \circ \phi$.  

\subsubsection{Finite-dimensional \cstar-algebras}\label{subsubsec:BackFinDimAlg}

We state the following well-known result as a theorem for future reference; a
proof can be found in standard resources such as \cite{Davidson.1996}.

\begin{theorem}[Classification of finite-dimensional algebras]\label{thm:ClsFinDimAlg}
A \cstar-algebra $\boldA$ is finite-dimensional if and only if 
there is a finite multiset $F$ of positive integers so that 
$\boldA$ is $\star$-isomorphic to $\bigoplus_{n \in F} M_n(\C)$.  
\end{theorem}

Suppose $F$ is a finite multiset of positive integers and set $\boldB = \bigoplus_{n \in F} M_n(\C)$.  
When $\pi$ is the projection map associated with a summand 
of $\boldB$, we associate $\pi$ with an injection map $\iota_\pi :\ran(\pi) \rightarrow \boldB$ 
so that $\pi \circ \iota_\pi = \Id_{\ran(\pi)}$. 
In addition, we ensure that $\pi \circ \iota_{\pi'} = \zerovec$ if $\pi \neq \pi'$.

We now discuss the $K$-theory of finite-dimensional algebras.  
Let $k = $ the cardinality of $F$.  
Then, by the properties of the $K_0$ functor (see, e.g., \cite[Section 3.2]{Rordam.Larsen.Laustsen.2000}), it is possible to construct from the isomorphisms and order units  described in 
Section \ref{subsubsec:BackMtrx} an order unit 
$v_F$ for $\Z^k$ and
a canonical isomorphism $\zeta_F$ from $\uK{\boldB}$ to $(\Z^k, v_F)$.
Moreover, this construction provides a 
canonical bijective mapping $\xi_F$ from the projection maps of the 
summands of $\boldB$ to the generators of $\Z^k$.

Suppose $F'$ is a finite multiset of positive integers and set $\boldC = \bigoplus_{n \in F'} M_n(\C)$. 
Let $k' = $ the cardinality of $F'$.  
Suppose $\phi$ is a $\star$-homomorphism from $\boldB$ to $\boldC$.  
The properties of the $K_0$ functor allow one to construct, from the maps 
$\eta_{m,n}$ described in Section \ref{subsubsec:BackMtrx}, 
a canonical homomorphism $\eta_{F,F'}(\phi)$ from the ordered group $\Z^k$ to 
$\Z^{k'}$
so that $\zeta_{F'} \circ \phi  = \eta_{F,F'}(\phi)\circ\zeta_F$.
In addition, the construction of $\eta_{F,F'}$ ensures that whenever 
$\pi$ is the projection map for a summand of $\boldB$ and $\pi'$ is a projection 
map for a summand of $\boldC$, 
$\pi_{\xi_{F'}(\pi')}(\eta_{F,F'}(\phi)(\xi_F(\pi))) = $
the multiplicity of $\pi' \circ \phi \circ \iota_\pi$.
It follows form the discussion in Section \ref{subsubsec:BackMtrx}
that $\eta_{F, F'}(\phi)(v_F) \leq v_{F'}$ and so 
$K_0(\phi)(K_0(\unit_\boldB)) \leq K_0(\unit_\boldC)$.  
Conversely, if $\gamma : (K_0(\boldB), K_0(\boldB)^+) \rightarrow (K_0(\boldC), K_0(\boldC)^+)$
is a homomorphism and if $\gamma(K_0(\unit_\boldB)) \leq K_0(\unit_\boldC)$, then 
there is a $\star$-homomorphism $\delta$ from $\boldA$ to $\boldB$ so 
that $K_0(\delta) = \gamma$. 

\subsubsection{AF algebras}\label{subsubsec:BackAF}
We finally arrive at a formal treatment of AF algebras.
\begin{definition}\label{def:AFCert}
Suppose $\boldA$ is a \cstar-algebra.  
 An \emph{AF certificate} of $\boldA$ is a sequence 
$(F_j, \psi_j)_{j \in \N}$ that satisfies the following:
\begin{enumerate}
    \item Each $F_j$ is a finite multiset of positive integers.
 
    \item For each $j \in \N$, $\psi_j$ is a unital $\star$-embedding of $\bigoplus_{n \in F_j} M_n(\C)$
    into $\boldA$.

    \item For all $j$, $\ran(\psi_j) \subseteq \ran(\psi_{j+1})$.

    \item $\boldA = \overline{\bigcup_{j \in \N} \ran(\psi_j)}$.
\end{enumerate}
\end{definition}

A \cstar-algebra is an \emph{AF algebra} if it has an AF certificate.  
We remark that this definition is not, but nevertheless is equivalent to, 
the standard definition of an AF algebra as an inductive limit of 
finite-dimensional algebras.  Our choice of this definition is motivated by forthcoming 
computabilty considerations.

We can now state Elliott's Theorem formally, a proof of which can be found in 
\cite{Rordam.Larsen.Laustsen.2000}.

\begin{theorem}[Elliott's Theorem]\label{thm:Elliott}
Suppose $\boldA$ and $\boldB$ are AF algebras.  If 
$\uK{\boldA}$ and 
$\uK{\boldB}$ are isomorphic, 
then $\boldA$ and $\boldB$ are  $\star$-isomorphic.  
Moreover, if $\alpha$ is an isomorphism from $\uK{\boldA}$
to $\uK{\boldB}$, 
then there is a unital $\star$-isomorphism $\phi : \boldA \rightarrow \boldB$
so that $K_0(\phi) = \alpha$. 
\end{theorem}

\subsubsection{Some perturbation principles}

The results in this section will support our proof of the Effective Glimm Lemma.
They generally have the flavor of saying that if one or more elements of a \cstar-algebra
$\boldA$ are ``sufficently close'' to some closed $\star$-subalgebra $\boldB$, then 
there is a small unitary $u$ so that when the map $\Ad_u$ is applied to these elements, one
 obtains elements of $\boldB$ that have certain desired relations to each other. 
Our aim is to state these principles in such a way as to make precise 
how close is ``sufficiently close'' and how to obtain the required unitaries. 
This level of precision might not have much utility for purely classical considerations,
but for future computability developments it is essential.  
Our source for these results is \cite{Strung.2021}.

We begin with the following, which is essentially \cite[Lemma 8.4.1]{Strung.2021}.

\begin{lemma}\label{lm:841Strung}
Let $\boldA$ be a unital \cstar-algebra.  Assume $b \in \boldA$ is self-adjoint and 
$p \in B(b; 1/2)\cap \Proj(\boldA)$.  Then, there exists 
$q \in \Proj(\boldA) \cap \mbox{\cstar}(\{b\})$ so that 
$\norm{q - b} \leq 2\norm{p - b}$.  Moreover, 
$p$ and $q$ are equivalent. 
\end{lemma}

Define $\delta_0 : (0, 1) \times (\N \setminus \{0\}) \rightarrow (0, \infty)$ by 
recursion as follows:
\begin{eqnarray*}
\delta_0(\epsilon,1) & := & \min\{\epsilon/2, 1/2\}\\
\delta_0(\epsilon, n+1) & := & \min\{4^{-1}\epsilon(n+1)^{-1}, \delta_0(12^{-1}\epsilon(n+1)^{-2}, n), 1\}\\
\end{eqnarray*}

The following is very similar to \cite[Lemma 8.4.2]{Strung.2021}, and our proof 
is based on the proof therein.

\begin{lemma}\label{lm:842Strung}
Let $\boldA$ be a unital \cstar-algebra, and assume $\boldB$ is a closed $\star$-subalgebra of $\boldA$. 
Suppose $p_1, \ldots, p_n$ are mutually orthogonal projections of $\boldA$ so that 
$\max_j d(p_j, \boldB) < \delta_0(\epsilon, n)$.  
Then, $\boldB$ contains mutually orthogonal projections $q_1, \ldots q_n$ so that \\
$\max_j \norm{p_j - q_j} < \epsilon$.  Furthermore, if 
$\sum_j p_j = \unit_\boldA$, then $\sum_j q_j = \unit_\boldA$.
\end{lemma}

\begin{proof}
We proceed by induction.  Suppose $n = 1$.  
Choose $b \in \boldB$ so that $\norm{p_1 - b} < \delta_0(\epsilon, 1)$. 
Since $p_1$ is a projection, $\norm{p_1 - \frac{1}{2}(b^* + b)} \leq \norm{p_1 - b}$. 
So, we can simply apply Lemma \ref{lm:841Strung} 
and if $\unit_\boldA \in \boldB$, we just set $p_1 = \unit_\boldA$.  

Suppose $n > 1$, and set $\delta = \delta_0(\epsilon, n)$.  
For each $j \in \{1, \ldots, n\}$, choose $b_j \in \boldB$ so that $\norm{p_j - b_j} < \delta$.  
By definition, $\delta \geq \delta_0(\epsilon/(12n^2), n - 1)$.  
Thus, by way of induction, there exist mutually orthogonal projections $q_1, \ldots, q_{n-1} \in \boldB$ 
so that $\max_{j < n} \norm{p_j - q_j} < \epsilon/(12n^2)$. 
Set $p = \sum_{j = 1}^{n-1} p_j$, and set $q = \sum_{j = 1}^{n-1} q_j$.  
We first estimate $\norm{p_n - (\unit_\boldA - q) b_n (\unit_\boldA - q)}$ as follows. 
We first note that
\begin{eqnarray*}
\norm{p_n - (\unit_\boldA - q) b_n (\unit_\boldA -q)} & = & \norm{(\unit_\boldA - p)p_n(\unit_\boldA - p) - (\unit_\boldA - q) b_n (\unit_\boldA -q)} \\
\end{eqnarray*}
Set 
\begin{eqnarray*}
E_1 & = & \norm{(\unit_\boldA - p)p_n(\unit_\boldA - p) - (\unit_\boldA - q) p_n (\unit_\boldA -p)}\\
E_2 & = & \norm{(\unit_\boldA - q) p_n (\unit_\boldA -p) - (\unit_\boldA - q) b_n (\unit_\boldA - p)}\\
E_3 & = & \norm{(\unit_\boldA - q) b_n (\unit_\boldA - p) - (\unit_\boldA - q) b_n (\unit_\boldA - q)}
\end{eqnarray*}
By the Triangle Inequality, 
\[
\norm{(\unit_\boldA - p)p_n(\unit_\boldA - p) - (\unit_\boldA - q) b_n (\unit_\boldA -q)} \leq E_1 + E_2 + E_3.
\]
However, by factoring out $p_n (\unit_\boldA - p)$, we see that $E_1 \leq \norm{p - q} < \epsilon / (12 n)$.
Similarly, by factoring out $\unit_\boldA - q$ and $\unit_\boldA - p$, 
we see that $E_2 \leq \norm{p_n - b_n} < \delta$.  
Finally, by factoring out $\norm{(\unit_\boldA - q)b_n}$, we 
obtain that 
\[
E_3 \leq \norm{b_n}\norm{p - q} \leq (\norm{p_n - b_n} + \norm{p_n}) \norm{p - q} < 2\norm{p - q}.
\]
Hence, we now have that 
\begin{eqnarray*}
\norm{p_n - (\unit_\boldA - q) b_n (\unit_\boldA -q)} & \leq & 3\norm{p - q} + \delta \\
& < &(\epsilon / (4n)) + \delta \\
& < & \epsilon /(2n).
\end{eqnarray*}
In addition, since $p_n, p, q$ are projections, we now have that 
\[
\norm{p_n - (\unit_\boldA - q) \frac{b_n + b_n^*}{2} (\unit_\boldA -q)} < \epsilon / (2n).
\]
Therefore, by Lemma \ref{lm:841Strung}, there is a projection 
$q_n \in (\unit_\boldA - q)\boldB(\unit_\boldA - q) \cap B(p_n; \epsilon / n)$.  
It follows that $q_n q_j = q_j q_n = \zerovec$ when $j < n$.  

Suppose 
$\sum_j p_j = \unit_\boldA$.  It then follows 
that $\norm{\unit_\boldA - \sum_j q_j} < \epsilon < 1$. 
Since $q_1, \ldots, q_n$ are mutually orthogonal, $\sum_j q_j$ is a projection.
Therefore $\unit_\boldA - \sum_j q_j$ is a projection. 
Hence, 
$\unit_\boldA = \sum_j q_j$.
\end{proof}

Define $\delta_1 : (0,1) \times (\N \setminus \{0\}) \rightarrow (0,1)$ by recursion as follows: 
\begin{eqnarray*}
\delta_1(\epsilon, 1) & := & \frac{1}{2}\\
\delta_1(\epsilon, n + 1) & := & \min\{3^{-1}, \epsilon (48 n)^{-1}, \delta_1(\epsilon (48n)^{-1}, n)\}.
\end{eqnarray*}
The following is similar to \cite[Lemma 8.4.4]{Strung.2021}, the only difference being our ``$\delta$ function" 
is defined somewhat differently for the sake of future developments.

\begin{lemma}\label{lm:844Strung}
Let $\boldA$ be a unital \cstar-algebra, and suppose $p_1, \ldots, p_n$ are projections of $\boldA$
so that $\max_{i\neq j} \norm{p_ip_j} < \delta_1(\epsilon,n)$. 
Then, there exist mutually orthogonal projections $q_1, \ldots, q_n$ of $\boldA$ so that 
$\max_j \norm{p_j - q_j} < \epsilon$.
\end{lemma}

\begin{proof}
We proceed by induction.  The case $n = 1$ is vacuous. 
So, suppose $n$ is a positive integer, and let $p_1, \ldots, p_{n+1}$ be projections of $\boldA$
so that $\max_{i\neq j} \norm{p_i p_j} < \delta_1(\epsilon, n + 1)$.  
By way of induction, there exist mutually orthogonal projections $q_1, \ldots, q_n$ of $\boldA$ 
so that $\max_{j \leq n} \norm{q_j - p_j} < \epsilon/(48n)$.  
Set $q = \sum_j q_j$.  Thus, $q \in \Proj(\boldA)$.  
We estimate $\norm{p_{n + 1} - (\unit_\boldA - q)p_{n + 1} (\unit_\boldA - q)}$ as follows. 
\begin{eqnarray*}
\norm{p_{n + 1} - (\unit_\boldA - q)p_{n + 1} (\unit_\boldA - q)} & = & 
\norm{qp_{n + 1} - qp_{n+1}q + p_{n + 1} q}\\
& \leq & 3 \norm{q p_{n + 1} } \\
& < & 3 \left(\sum_{j = 1}^n \norm{p_j p_{n +1}} + \epsilon/(48 n) \right)\\
& < & 3 (\epsilon/ 48 + \epsilon/48)\\
& < & \epsilon/8.
\end{eqnarray*}
By Lemma \ref{lm:841Strung}, there is a projection 
$q_{n + 1} \in (\unit_\boldA - q) \boldA (\unit_\boldA - q)$ so that 
$\norm{q_{n + 1} - p_{n + 1}} < \epsilon$.  
It follows from direct computation that $q_{n + 1} q_j = q_j q_{n + 1} = 0$ when $j \leq n$. 
\end{proof}

For all $\epsilon \in (0,1)$ and every positive integer $n$, set 
\[
\delta_2(\epsilon,n) = \min\{5^{-1}, \epsilon (8 - 5 \epsilon), \delta_1(\epsilon,n)\}.
\]
The following is essentially \cite[Lemma 8.4.7]{Strung.2021}, and there is very little difference 
in the proof which we therefore choose to omit.

\begin{lemma}\label{lm:847Strung}
Suppose $\boldA$ is a unital \cstar-algebra and suppose $\boldB$ is a closed $\star$-subalgebra of $\boldA$. 
Assume $(e_{i,j})_{i,j}$ is an $n \times n$ system of matrix units of $\boldA$ so that 
$\max_{i,j} d(e_{i,j}, \boldB) < \delta_2(\epsilon, n)$.  
Then, there is an $n \times n$ system of matrix units $(f_{i,j})_{i,j}$ of $\boldB$ 
so that $\max_{i,j} \norm{e_{i,j} - f_{i,j}} < \epsilon$.  
In addition, if $\unit_\boldA \in \boldB$, and if $\sum_i e_{i,i} = \unit_\boldA$, then $\sum_i f_{i,i} = \unit_\boldA$.
\end{lemma}

\subsection{Bratteli diagrams}
We now arrive at our final background topic: Bratteli diagrams.
Bratteli diagrams were introduced in \cite{Bratteli.1972}, and 
they provide a combinatorial way to visualize the structure of 
AF algebras.  They also provide a visualization of the structure of 
dimension groups \cite{Durand.Perrin.2022}. 
Our goal here is to provide a definition of this concept 
and to explain the association of Bratteli diagrams with AF algebras by 
means of AF certificates.  We also discuss the relation between these diagrams
and dimension groups.

\subsubsection{Fundamental definitions}

\begin{definition}\label{def:Brat}
A \emph{Bratteli diagram} consists of a set $V$ of \emph{vertices}, a \emph{level function}
$L : V \rightarrow \N$, and 
an \emph{edge function} $E : V \times V \rightarrow \N$ that satisfy the following:
\begin{enumerate}
    \item For all $u,v \in V$, if $E(u,v) > 0$, then $L(v) = L(u) + 1$.

    \item For each $n \in \N$, $L^{-1}[\{n\}]$ is finite.  
\end{enumerate}
\end{definition}

When $\BratD$ is a Bratteli diagram, let $V_\BratD$ denote its vertex set, 
$E_\BratD$ its edge function, and $L_\BratD$ its level function.
We refer to $L_\BratD^{-1}[\{n\}]$ as the \emph{$n$-th level} of $\BratD$. 
We conceive of a Bratteli diagram as representing a multigraph where 
the edge function counts the number of edges between two vertexes; 
there is no need to formally represent the edges themselves.  
Bratteli diagrams provide most of the information required to represent 
an AF algebra.  To give a complete representation, one considers labelings, which 
we define as follows. In the case of AF algebras, these will assist in representing the dimensions 
of the summands, and in the case of unital dimension groups they will describe the order units.

\begin{definition}\label{def:LblBrat}

\

\begin{enumerate}
    \item A \emph{labeling} of a Bratteli diagram $\BratD$ is a function $\Lambda : V_\BratD \rightarrow \N$. 

    \item A \emph{labeled Bratteli diagram} consists of a Bratteli diagram $\BratD$ and 
    a labelling of $\BratD$.
\end{enumerate}
\end{definition}

We regard a labeled Bratteli diagram as a kind of Bratteli diagram.  
When $\BratD$ is a labeled Bratteli diagram, let $\Lambda_\BratD$ denote its 
labeling.

There is an obvious notion of isomorphism for Bratteli diagrams (labeled and unlabeled) which
is often too restrictive.  The definition of a more useful equivalence relation 
requires the following. 

\begin{definition}\label{def:BratPth}
Let $\BratD$ be a Bratteli diagram.  For all $u,v \in V_\BratD$, let 
\[
P_\BratD(u,v) = 
\left\{\begin{array}{cc}
0 & L_\BratD(u) \geq L_\BratD(v)\\
E(u,v) & L_\BratD(v) = L_\BratD(u) + 1\\
\sum_{w \in L_\BratD^{-1}[\{L(u) + 1\}]} E(u,w) P_\BratD(w,v) & L_\BratD(v) > L_\BratD(u) + 1\\
\end{array}
\right.
\]
\end{definition}

Informally speaking, $P_\BratD(u,v)$ counts the number of paths from $u$ to $v$ in $\BratD$.
By means of this path-counting function we introduce the concept of a telescoping.

\begin{definition}\label{def:Tel}
Suppose $\BratD$ and $\BratD'$ are Bratteli diagrams. 
We say that $\BratD$ is a \emph{telescoping} of $\BratD'$ if there is an increasing sequence 
$(n_k)_{k \in \N}$ of integers so that $n_0 = 0$, 
$L_\BratD^{-1}[\{k\}] = L_{\BratD'}^{-1}[\{n_k\}]$, and $E_\BratD(u,v) = P_{\BratD'}(u,v)$. 
If $\BratD$ and $\BratD'$ are labelled, then we also require that 
$\Lambda_\BratD(v) = \Lambda_{\BratD'}(v)$.
\end{definition}

We say that two Bratteli diagrams $\BratD$ and $\BratD'$ are \emph{equivalent} 
if there is a sequence $\BratD_0 = \BratD$, $\ldots$, $\BratD_n = \BratD'$
so that for each $j < n$, $\BratD_j$ and $\BratD_{j + 1}$ are isomorphic or one 
is a telescoping of the other.

\subsubsection{Bratteli diagrams of AF algebras}

We now give a formal definition of the relationship between Bratteli diagrams 
and AF algebras by means of AF certificates.  We start with the Bratteli diagram of an inductive sequence of finite-dimensional algebras.

\begin{definition}\label{def:BratIndSeq}
Suppose $(\bigoplus_{n \in F_k} M_n(\C), \psi_k)_{k \in \N}$ is an inductive sequence of 
finite-dimensional \cstar-algebras.
We define the \emph{standard labeled Bratteli diagram of} $(\bigoplus_{n \in F_k} M_n(\C), \psi_k)_{k \in \N}$
as follows:
\begin{enumerate}
    \item The vertices are the pairs of the form $(k, \pi)$, where $\pi$ is the projection map of a summand of
$\bigoplus_{n \in F_k} M_n(\C)$.

    \item The edge function is defined by defining  
    $E((k,\pi), (k+1, \pi'))$ to be the multiplicity of 
$\pi' \circ \psi_k \circ \iota_\pi$.

    \item The level function is defined by setting $L(k, \pi) = k$. 

    \item The labeling function is defined by defining $\Lambda(k, \pi)$ to be  
    $\sqrt{\dim(\dom(\pi))}$.
\end{enumerate}
We say that a labeled Bratteli diagram $\BratD$ is a \emph{Bratteli diagram of $(\bigoplus_{n \in F_k} M_n(\C), \psi_k)_{k \in \N}$} if it is isomorphic to the standard Bratteli diagram of 
this inductive sequence. 
\end{definition}

Now let $\BratD$ be a Bratteli diagram. 
We define $\BratD$ to be a Bratteli diagram of 
an AF certificate $(F_s, \phi_s)_{s \in \N}$ if it is a Bratteli diagram of 
the inductive sequence 
$(\bigoplus_{n \in F_k} M_n(\C), \phi_{k+1}^{-1} \circ \phi_k)_{k \in \N}$.
Then we define $\BratD$ to be a \emph{Bratteli diagram of $\boldA$} if 
it is a Bratteli diagram of an AF certificate of $\boldA$.

\subsubsection{Bratteli diagrams of inductive sequences of simplicial groups}

In the proof of Main Theorem \ref{mthm:CompIso}, it will be useful to have a 
notion of a labeled Bratteli diagram of an inductive sequence of unital simplicial groups. 
Given such a sequence $((\Z^{n_s}, u_s), \phi_s)_{s \in \N}$, we define its 
\emph{standard Bratteli diagram} as follows:
\begin{enumerate}
\item The vertex set is 
$\{(s, e^{n_s}_j)\ :\ s \in \N\ \text{ and }\ j \in \{1, \ldots, n_s\} \}.$
\item We define the level function by setting $L(s,g) = s$. 
\item We define the edge function by setting 
$E((s, g), (s + 1, g')) = \pi_{g'}(\phi_s(g))$.
\item Finally, we define the labeling function by setting $\Lambda(s, e^{n_s}_j) = u_s(j)$.
\end{enumerate}
We then say that a Bratteli diagram $\BratD$ is a Bratteli diagram of 
$((\Z^{n_s}, u_s), \phi_s)_{s \in \N}$ if it is isomorphic to the standard 
Bratteli diagram of this inductive sequence. 

This completes our picture of the classical material to be considered in this paper.
We now discuss its computability theory.

\section{Background from effective (computable) mathematics}\label{sec:BackEff}

Our goal in this section is to summarize relevant prior work from the computability theory
of groups and \cstar algebras.
We assume knowledge of the fundamentals of computability theory as expounded in \cite{Cooper.2004}.

\subsection{Algorithmic group theory}\label{subsec:CompGrp}

The objective of this section is to review the concepts of group presentation and associated
computable maps, c.e. sets, and computable sets.  
We follow the framework that we used in \cite{UHFPaper}.

Fix a countably infinite set $X = \{x_0, x_1, \ldots\}$ of indeterminates
and let $F_\omega$ denote the free group generated by $X$. 
We implicitly view $F_\omega$ as equipped with a computable bijection with $\N$, 
so that it makes sense to speak of, for example, c.e.\ and computable subsets of $F_\omega$.

Fix a group $G$, and assume $\nu$ is an epimorphism from 
$F_\omega$ onto $G$.  The pair $(G, \nu)$ is called a \emph{presentation} of $G$. 
Set $G^\# = (G, \nu)$.
If $w \in F_{\omega}$ and $\nu(w) = a$, then we call $w$ an \emph{$G^\#$ label} of $a$.  The \emph{kernel} of $G^\#$ is the kernel of $\nu$.

Some groups admit ``standard presentations''.  For example, we always view the group $\mathbb{Z}^n$ as equipped with its standard presentation induced by the labeling $\nu:F_\omega\to \mathbb{Z}^n$ given by sending the first $n$ elements of $X$ to the standard generators of $\mathbb{Z}^n$ and mapping the remaining elements of $X$ to the identity in $\mathbb{Z}^n$.  In the sequel, we identify groups which admit standard presentations with said standard presentation.

Suppose $G^\#$ is a presentation of a group.
We say that $G^\#$ is \emph{c.e.}\ (resp. \emph{computable}) if its kernel is a c.e.\ (resp. computable) subset of $F_{\omega}$.
We remark that the class of computably presentable groups is closed under isomorphism as is 
the class of groups with c.e. presentations.

Suppose $S \subseteq G$.  
We say that $S$ is a \emph{computable (resp. c.e.)} set of $G^\#$
if the set of all $G^\#$ labels of elements of $S$ is computable (resp. c.e.).
If $S$ is a c.e. set of $G^\#$, then we say $e \in \N$ is a \emph{$G^\#$ index}
of $S$ if $e$ is a code of a Turing machine that enumerates all $G^\#$ labels of elements 
of $S$.

Suppose $H^\#$ is a group presentation.  A map $\phi : G \rightarrow H$ 
is a \emph{computable map from $G^\#$ to $H^\#$} if 
there is an algorithm that, given a $G^\#$ label of $g \in G$, computes an 
$H^\#$ label of $\phi(g)$.
A \emph{$(G^\#, H^\#)$ index} of $\phi$ is an index of such an algorithm. 

The following simple observation will streamline our discussion of the computability theory 
of dimension groups.

\begin{proposition}\label{prop:CompMapZn}
If $G^\#$ is a group presentation, then every homomorphism from 
the abelian group $\Z^n$ to $G$ is a computable map from $\Z^n$ to $G^\#$.
\end{proposition}

\begin{proof}[Proof sketch]
Such a map is completely described by $G^\#$ labels of its values on the standard
basis for $\Z^n$.
\end{proof}

\subsection{\cstar-algebras}\label{subsec:BackCompAlg}

\subsubsection{Basic background}
We recall the setting for studying \cstar-algebras from the perspective of computability theory.  Much of this basic material is taken from our earlier work \cite{UHFPaper}.  The reader seeking more details on this background material may also consult \cite{fox2022computable} and \cite{EMST}.
\begin{definition}
Let $\boldA$ be a \cstar-algebra.  A \emph{presentation} of $\boldA$ is a sequence $(a_0, a_1, \ldots)$ of points in $\boldA$ such that the *-subalgebra of $\boldA$ generated by $\{a_i : i \in \mathbb{N}\}$ is dense in $\boldA$.  The points $a_0, a_1, \ldots$ are called the \emph{special points} of the presentation, while the points of the form $p(a_0,\ldots,a_n)$, where $p(x_0,\ldots,x_n)$ is a $*$-polynomial with coefficients from $\mathbb{Q}(i)$ (as $n$ varies), are called the \emph{generated points} of the presentation.  We typically denote a presentation of $\boldA$ by $\boldA^\#$ or $\boldA^{\dagger}$.
\end{definition}

In earlier work, such as \cite{burton2024computable} and \cite{UHFPaper}, the generated points of a presentation are referred to as \emph{rational points}; we now prefer to use the term generated points, following the convention established in \cite{EMST}.

By standard techniques (see, for example, \cite[Section 2]{burton2024computable}), we can obtain an effective (but typically non-injective) list of the generated points of a presentation of a \cstar-algebra.  When we say, for example, that an algorithm takes a generated point $q$ of $\boldA^\#$ as input, we really mean that we have fixed an effective list $(q_i)_{i \in \mathbb{N}}$ of generated points of $\boldA^\#$ and the algorithm takes input an index $i$ for which $q = q_i$.

By the Church-Turing thesis, every algorithm can be represented as a Turing machine.  An \emph{index} (or \emph{code}) of an 
algorithm is a natural number that codes a representation of the algorithm as a Turing machine.

There is a natural notion of a (finite) product of presentations.  When we have a presentation $\boldA^\#$ of $\boldA$, if we need a presentation on $\boldA^n$ we use this product presentation, though the specific details of its construction will not be crucial for us.  See \cite[Subsection 1.1]{UHFPaper} for details.

\begin{definition}\label{def:pres.comp.etc}
Fix a presentation $\boldA^\#$.
\begin{enumerate}

    \item $\boldA^\#$ is \emph{computable} if there is an algorithm that, given a generated point $q$ and a $k \in \N$, outputs a rational number $r$ such that $\left\vert\norm{q} - r\right\vert < 2^{-k}$; an index of such an algorithm is an \emph{index} of $\boldA^\#$.
    

    \item $\boldA^\#$ is \emph{(right) c.e.}\ if there is an algorithm that, given a generated point $q$, enumerates a decreasing sequence $(r_n)_{n \in \mathbb{N}}$ of rational numbers such that $\lim_{n\to\infty}r_n = \norm{q}$; an index of such an algorithm is called a 
    \emph{(right) c.e.\ index} of $\boldA^\#$.
\end{enumerate}
\end{definition}

There is a corresponding notion of left c.e. presentation, but we will not make use of it here.


With each presentation of $\boldA$, there is an associated class of computable points.  These are defined as follows.

\begin{definition}
An element $a\in \boldA$ is a \emph{computable point of $\boldA^\#$} if there is an algorithm such that, given $k \in \N$,
returns a generated point $b$ of $\boldA^\#$ such that $\|a-b\|< 2^{-k}$; an index of such an algorithm is a 
\emph{$\boldA^\#$-index for $a$}. 
\end{definition}

In other words, a computable point of $\boldA^\#$ is one that can be effectively approximated by generated points of 
$\boldA^\#$ with arbitrarily good precision.

We remark that the class of computably presentable \cstar-algebras is closed under $\star$-isomorphism, that is, if $\boldA$ is a computably presentable \cstar-algebra and 
$\boldB$ is $\star$-isomorphic to $\boldA$, then $\boldB$ is computably presentable. 
The same principle holds for the class of \cstar-algebras that have c.e. presentations.

A \emph{rational open ball} of a presentation $\boldA^\#$ is an open ball in $\boldA$ that is centred at a generated point of $\boldA^\#$ and has rational radius.  Each rational open ball of $\boldA^\#$ has a \emph{code} consisting of the radius and a code for the centre of the ball.
If $C \subseteq \boldA$ is closed, we say that $C$ is a \emph{c.e.-closed subset of $\boldA^\#$} if 
the set of all (codes for) open rational balls of $\boldA^\#$ that intersect $C$ is c.e.
If $U \subseteq \boldA$ is open, we say that $U$ is a \emph{c.e.-open set of $\boldA^\#$} if 
there is a c.e.\ set $\mathcal{S}$ of rational open balls so that 
$U = \bigcup \mathcal{S}$.

By an index of a c.e.-open set we mean an index for the c.e.\ set of open balls whose union comprises the set; the index of a c.e.-closed set is defined analogously.

The following is \cite[Lemma 1.10]{UHFPaper}:

\begin{lemma}\label{intersection}
Suppose that $\boldA^\#$ is a c.e.\ presentation.  Further suppose that $U$ is a c.e.-open subset of $\boldA^\#$ and $C$ is a c.e.-closed 
subset of $\boldA^\#$ such that $U\cap C\neq\emptyset$.  Then $U\cap C$ contains a computable point.  Moreover, an index of this computable
point can be computed from indexes of $U$ and $C$.
\end{lemma}

The following is folklore and is a simple consequence of the pertinent definitions. 

\begin{proposition}\label{prop:PreimCeOpen}
Suppose $\boldA^\#$ and $\boldB^\#$ are c.e. presentations of \cstar-algebras, and let 
$U$ be a c.e. open set of $\boldB^\#$.  If $f$ is a computable map from 
$\boldA^\#$ to $\boldB^\#$, then $f^{-1}[U]$ is a c.e. open set of $\boldA^\#$. 
Moreover, it is possible to compute an $\boldA^\#$ index of $f^{-1}[U]$ from a $\boldB^\#$ index 
of $U$ and an $(\boldA^\#, \boldB^\#)$ index of $f$.
\end{proposition}

Suppose $\boldA^\#$ is a c.e. presentation of a stably finite unital \cstar-algebra. 
In \cite{UHFPaper}, it is shown that $\unit_\boldA$ is a computable point of 
$\boldA^\#$.  
We say that $e \in \N$ is a \emph{unital index of $\boldA^\#$} if
$e$ is a code of a pair $(e_0, e_1)$ that consists of an index of $\boldA^\#$ and 
an $\boldA^\#$ index of $\unit_\boldA$.  Our motivation for introducing such indexes
is that even for the class of stably finite unital \cstar-algebras, an $\boldA^\#$-index of $\unit_\boldA$ may not be computable from an index of 
$\boldA^\#$; see \cite{mcnicholl2024evaluative}.

We now discuss prior developments regarding the computability of $K$-theory.

\subsubsection{$K$-theory}\label{subsubsec:BackCompKThy}

In \cite{UHFPaper}, it is shown that there is a \emph{computable} functor from the 
category of c.e. presentations of unital \cstar-algebras (where the morphisms are 
computable unital $\star$-morphisms) to the category of c.e. presentations of abelian groups.
This functor is denoted $K_0$. 
The functor is computable in the sense that it is possible to compute an index of $K_0(\boldA^\#)$ from 
an index of $\boldA^\#$.  and from an $(\boldA^\#, \boldB^\#)$ index of a computable $\star$-homomorphism $\phi$, it is possible to compute a $(K_0(\boldA^\#), K_0(\boldB^\#))$-index of $K_0(\phi)$. 

The application of computable $K$-theory to UHF algebras relied heavily on the 
concept of computable weak stability, which will also support our proof of the Effective Glimm Lemma.
Hence, it is our next topic for review.

\subsubsection{Computable weak stability}

The concept of computable weak stability was introduced by Fox, Goldbring, and Hart \cite
{FoxGoldbringHart.2024+} based on the corresponding notion of weak stability in continuous 
model theory, which was introduced in \cite{FarahEtAll.2021}.  
In \cite{UHFPaper}, we presented computable weak stability in a way that illuminated 
its connection with c.e. closedness.  
For the sake of referencing some of its details, we recall this definition now. 

\begin{definition}\label{def:CompWklyStbl}
Suppose that $\vec{x} = (x_1, \ldots, x_N)$ is a tuple of variables, $C_1, \ldots, C_N \in \N$, 
and $p_1(\vec{x}), \ldots, p_M(\vec{x})$ are rational $*$-polynomials.  Consider the set of relations
$$\mathcal{R} = \{p_i(\vec{x})=0 \ : \ i=1,\ldots,M\}\cup\{\norm{x_j} \leq C_j \ : \ j=1,\ldots,N\}$$.
\begin{enumerate}
    \item For a \cstar-algebra $\boldB$ and $w_1, \ldots, w_N \in \boldB$, 
    let $\mathcal{R}^{\boldB}$ denote the quantity 
    $$
    \max(\{\norm{p_j(\vec{w})}\ :\ j \in \{1, \ldots, M\}\} \cup \{\norm{w_j}- C_j\ : j \in \{1, \ldots, N\}\}$$
    and write $\boldB \models \mathcal{R}(\vec{w})$ if 
    $p_j(\vec{w}) = 0$ for all $j \in \{1, \ldots, M\}$ and $\norm{w_j} \leq C_j$ for all $j \in \{1, \ldots, N\}$.

    \item A function $g : \N \rightarrow \N$ is a \emph{modulus of weak stability} for 
    $\mathcal{R}$ provided that, for every $k \in \N$, every \cstar-algebra $\boldB$, 
    and every $w_1, \ldots, w_N \in \boldB$, if $\mathcal{R}^\boldB(\vec{w}) < 2^{-g(k)}$, 
    then there exists $z_1, \ldots, z_N \in \boldB$ so that 
    $\max_j \norm{w_j - z_j} < 2^{-k}$ and $\mathbf{B} \models \mathcal{R}(\vec{z})$.

    \item We say that $\mathcal{R}$ is \emph{weakly stable} if it has a modulus of weak stability and \emph{computably weakly stable} if it has a computable modulus of weak stability.
\end{enumerate}
\end{definition}

In \cite[Theorem 1.12]{UHFPaper} we showed that computably weakly stable relations define c.e. closed sets in c.e. presentations of \cstar-algebras.

We have now completed our description of relevant prior developments in 
classical and computable mathematics.  We now attend to some purely prelminary matters 
in these realms which will support our main efforts later.

\section{Preliminaries from classical mathematics}\label{sec:PrlmClscl}

\subsection{Ordered abelian groups}

\subsubsection{Inductive limits}

Our goal here is to arrive at a more concrete understanding of inductive limits 
of ordered abelian groups which will be valuable when developing their computability 
theory. We begin with abelian groups.

\begin{lemma}\label{lm:IndLimAbl}
Suppose $(G_s, \phi_s)_{s \in \N}$ is an inductive sequence of abelian groups and let 
$(G, (\nu_s)_{s \in \N})$ be an inductive limit of $(G_s, \phi_s)_{s \in \N}$.  Then:
\begin{enumerate}
    \item $G = \bigcup_{s \in \N} \ran(\nu_s)$.\label{lm:IndLimAbl:ran}

    \item For each $s \in \N$, $\ker(\nu_s) = \bigcup_{k \in \N} \ker(\phi_{s, s+k})$.\label{lm:IndLimAbl:ker}
\end{enumerate}
\end{lemma}

\begin{proof}[Proof sketch]
We first review a standard construction of 
an inductive limit of $(G_s, \phi_s)_{s \in \N}$. For each $s \in \N$, let $\rho_s : G_s \rightarrow \prod_n G_n$ be the homomorphism defined by 
setting $\rho_s(a) = (\phi_{s,n}(a))_{n \in \N}$.
When $f,g \in \prod_n G_n$, write $f \sim g$ if there exists $k \in \N$
so that $f(k+n) = g(k + n)$ for all $n \in \N$.  Thus, $\sim$ is a congruence relation on 
$\prod_n G_n$.  Let $\pi : \prod_n G_n \rightarrow \prod_n G_n/ \sim$ denote the canonical epimorphism, 
and set $\mu_s = \pi \circ \rho_s$ for all $s \in \N$.

Define $H$ to be $\bigcup_{s \in \N} \ran(\mu_s)$.  By construction, 
$\ran(\mu_s) \subseteq \ran(\mu_{s+1})$.  Hence, $H$ is an abelian group.
It is well-known, and easy to verify, that $(H, (\mu_n)_{n \in \N})$ is
an inductive limit of $(G_s, \phi_s)_{s \in \N}$.  Furthermore, 
for each $s \in \N$, $\ker(\mu_s) = \bigcup_{k \in \N} \ker(\phi_{s,s+k})$. 

Let $\lambda$ be the reduction of $(G, (\nu_s)_{s \in \N})$ to 
$(H, (\mu_s)_{s \in \N})$.  
By Proposition \ref{prop:IndLimUni}, $\lambda$ is an isomorphism. 
Since $(H, (\mu_s)_{s \in \N})$ satisfies (\ref{lm:IndLimAbl:ran}) and 
(\ref{lm:IndLimAbl:ker}), it follows that $(G, (\nu_s)_{s \in \N})$ does as well.
\end{proof}

We now proceed to describe inductive limits of ordered abelian groups.

\begin{lemma}\label{lm:IndLimOrdAbl}
Suppose $(\calG_s, \phi_s)_{s \in \N}$ is an inductive sequence of ordered abelian groups.
Assume $(\calG, (\nu_s)_{s \in \N})$ is an inductive upper limit of $(\calG_s, \phi_s)_{s \in \N}$. 
Then $(\calG, (\nu_s)_{s \in \N})$ is an inductive limit of $(\calG_s, \phi_s)_{s \in \N}$
if and only if it satisfies the following:
\begin{enumerate}
    \item $\calG^+ = \bigcup_{s \in \N} \nu_s[\calG_s^+]$.\label{lm:IndLimOrdGrp:ran}

    \item For each $s \in \N$, $\ker(\nu_s) = \bigcup_{k \in \N} \ker(\phi_{s,s+k})$. \label{lm:IndLimOrdGrp:ker}
\end{enumerate}
\end{lemma}

\begin{proof}  
Assume $\calG = (G, P)$ and $\calG_s = (G_s, P_s)$ for all $s \in \N$. 
Let $(H, (\mu_s)_{s \in \N})$ be an inductive limit of 
$(G_s, \phi_s)_{s \in \N}$, and let $\lambda$ be the reduction of 
$(H, (\mu_s)_{s \in \N})$ to $(G, (\nu_s))_{s \in \N}$.  
Let $P' = \bigcup_{s \in \N} \mu_s[P_s]$ and set $\calH = (H, P')$.  
By \cite[Proposition 6.2.6]{Rordam.Larsen.Laustsen.2000}, 
$(\calH, (\mu_s)_{s \in \N})$ is an inductive limit of $(\calG_s, \phi_s)_{s \in \N}$.
It also follows that $\lambda$ is positive and so 
$\lambda$ is the reduction of $(\calH, (\mu_s)_{s \in \N})$ to $(\calG, (\nu_s)_{s \in \N})$.
Furthermore, by construction, $\calH^+ = \bigcup_s \mu_s[\calG^+_s]$, and by 
Lemma \ref{lm:IndLimAbl}, $\ker(\mu_s) = \bigcup_{k \in \N} \ker(\phi_{s,s+k})$ for all $s \in \N$.

Suppose $(\calG, (\nu_s)_{s \in \N})$ is an inductive limit of 
$(\calG_s, \phi_s)_{s \in \N}$.  
Then $\lambda$ is an isomorphism.  Since $(\calH, (\mu_s)_{s \in \N})$ satisfies 
(\ref{lm:IndLimOrdGrp:ran}) and (\ref{lm:IndLimOrdGrp:ker}), 
it follows that $(\calG, (\nu_s)_{s \in \N})$ does as well.

Finally, suppose $(\calG, (\nu_s)_{s \in \N})$ satisfies (\ref{lm:IndLimOrdGrp:ran}) and (\ref{lm:IndLimOrdGrp:ker}).  It follows from (\ref{lm:IndLimOrdGrp:ran}) that $\lambda$
is surjective.  Let $h \in \ker(\lambda)$.  Then, there exists $s$ and $g \in \calG_s$ so 
that $\mu_s(g) = h$.  Hence, $g \in \ker(\nu_s)$, and so there exists $k \in \N$
so that $g \in \ker(\phi_{s,s+k})$.  
Thus, $h = \mu_{s+k}(\phi_{s,s+k}(g)) = \zerovec_G$, and so $\lambda$ is an isomorphism.
Hence, $(\calG, (\nu_s)_{s \in \N})$ is an inductive limit of $(\calG_s, \phi_s)_{s \in \N}$.
\end{proof}

\subsubsection{Unital dimension groups}

Unital dimension groups are central to the $K$-theory of unital AF algebras, and 
accordingly we elaborate on our definition of certificate of dimensionality 
as follows.

\begin{definition}\label{def:UntlCertDim}
Suppose $(\calG, u)$ is a unital dimension group.  We say that $(n_s, u_s \nu_s, \phi_s)_{s \in \N}$
is a \emph{unital certificate of dimensionality} of $(\calG, u)$ if $u_s$ is an order unit of 
$\Z^{n_s}$ and if 
$((\calG, u), (\nu_s)_{s \in \N})$ is an inductive limit of 
$((\Z^{n_s}, u_s), \phi_s)_{s \in \N}$ in the category of unital ordered abelian groups. 
\end{definition}

These certificates will form the backbone of our proof of Main Theorem \ref{mthm:ClsCePresDimGrp::Untl}.

\subsection{\cstar-algebras}

\subsubsection{Inductive limits}

As in the previous section, we give a concrete description of 
inductive limits of \cstar-algebras.

\begin{lemma}\label{lm:IndLimAlg}
Suppose $(\boldA_s, \phi_s)_{s \in \N}$ is an inductive sequence of \cstar-algebras, 
and let $(\boldA, (\nu_s)_{s \in \N})$ be an inductive upper limit of 
$(\boldA_s, \phi_s)_{s \in \N}$.  Then, $(\boldA, (\nu_s)_{s \in \N})$ is an inductive limit of 
$(\boldA_s, \phi_s)_{s \in \N}$ if and only if it satisfies the following:
\begin{enumerate}
    \item $\boldA = \overline{\bigcup_{s \in \N} \ran(\nu_s)}$.\label{lm:IndLimAlg::ran}

    \item For every $s \in \boldA_s$, $\norm{\nu_s(a)} = \lim_k \norm{\phi_{s, s + k}(a)}$.\label{lm:IndLimAlg::norm}

    \item For every $s \in \N$, $\ker(\nu_s) = \{a \in A_s\ : \lim_k \norm{\phi_{s,s+k}(a)} = 0\}$.\label{lm:IndLimAlg::ker}
\end{enumerate}
\end{lemma}

\begin{proof}
The forward direction is \cite[Proposition 6.2.4]{Rordam.Larsen.Laustsen.2000}.

Suppose (\ref{lm:IndLimAlg::ran}), (\ref{lm:IndLimAlg::norm}), and (\ref{lm:IndLimAlg::ker}) are satisfied.
Let $(\boldB, (\mu_s)_{s \in \N})$ be an inductive upper limit of 
$(\boldA_s, \phi_s)_{s \in \N}$.  

We first demonstrate that $\ker(\nu_s) \subseteq \ker(\mu_s)$.  
Let $a \in \ker(\nu_s)$.  
Then, by (\ref{lm:IndLimAlg::norm}), $\lim_k \norm{\phi_{s, s+k}(a)} = 0$. 
Since every $\star$-homomorphism is $1$-Lipschitz, 
we have $\lim_k \norm{\mu_{s + k}(\phi_{s, s+k}(a))} = 0$.  However, since 
$(\boldB, (\mu_s)_{s \in \N})$ is an inductive upper limit of $(\boldA_s, \phi_s)_{s \in \N}$,
it follows that $\mu_{s+k}(\phi_{s,s+k}(a)) = \mu_s(a)$ and so $a \in \ker(\mu_s)$.

It now follows that if $\nu_s(a_0) = \nu_{s + k}(a_1)$, then 
$\mu_s(a_0) = \mu_{s+k}(a_1)$.

We now construct a $\star$-homomorphism $\lambda'$ from $\bigcup_{s \in \N} \ran(\nu_s)$ to $\boldB$.  
To begin, let $a \in \ran(\nu_s)$, and choose any $a' \in \boldA_s$ so that $\nu_s(a') = a$. 
Define $\lambda'(a)$ to be $\mu_s(a')$.  It follows from what has just been shown that 
$\lambda'$ is well defined. Since each $\nu_s$ and each $\mu_s$ is a $\star$-homomorphism,
it also follows that $\lambda'$ is a $\star$-homomorphism. 

We now claim that $\lambda'$ is $1$-Lipschtiz.  Let $a \in \ran(\nu_s)$, and suppose 
$\nu_s(a') = a$.  By (\ref{lm:IndLimAlg::norm}), $\norm{\nu_s(a')} = \lim_k \norm{\phi_{s, s+k}(a')}$.
Since each $\mu_{s+k}$ is $1$-Lipschitz, it follows that 
$\norm{\nu_s(a')} \geq \lim_k \norm{\mu_{s + k}(\phi_{s, s+k}(a'))}$. But, 
$\mu_{s + k}(\phi_{s, s + k}(a')) = \mu_s(a') = \lambda'(a)$.

It now follows that $\lambda'$ has a unique continuous extension $\lambda$ to 
$\boldA$.  Furthermore, $\lambda$ is a $\star$-homomorphism and 
by construction $\lambda \circ \nu_s = \mu_s$ for all $s \in \N$.

Finally, it directly follows from the definitions that any two reductions of 
$(\boldA, (\nu_s)_{s \in \N})$ to $(\boldB, (\mu_s)_{s \in \N})$ agree on 
$\bigcup_{s \in \N} \ran(\nu_s)$ and hence must be identical.
\end{proof}

An immediate consequence of Lemma \ref{lm:IndLimAlg} is the following connection
 between AF certificates and inductive limits.

\begin{corollary}\label{cor:IndLimAFCert}
If $(F_j,\psi_j)_{j \in \N}$ is an AF certificate of $\boldA$, 
then $(\boldA, (\psi_j)_{j \in \N})$ is an inductive limit of 
$(\bigoplus_{n \in F_j} M_n(\C), (\psi_{j + 1}^{-1} \circ \psi_j)_{j \in \N})$.
\end{corollary} 

\subsubsection{Matricial systems}

Our only item here is to present some terminology which will support 
our proof of the Effective Glimm Lemma.

\begin{definition}\label{def:CompMat}
Suppose $\boldA$ is a \cstar-algebra and let $n_1, \ldots, n_m$ be positive integers.   
An array $(e^s_{i,j})_{s,i,j}$ is a \emph{type-$(n_1, \ldots, n_m)$ matricial system of $\boldA$}
if it satisfies the following: 
\begin{enumerate}
    \item For each $s \in \{1, \ldots, m\}$, 
    $(e^s_{i,j})_{i,j}$ is an $n_s \times n_s$ system of matrix units of $\boldA$. 

    \item When $s,s' \in \{1, \ldots, m\}$ are distinct, then $e^s_{i,j} e^{s'}_{i',j'} = 0$
    for all $i,j \in \{1, \ldots, n_s\}$ and all $i',j' \in \{1, \ldots, n_{s'}\}$.
\end{enumerate}
If, in addition, $\boldA$ is unital and $\sum_{s,i} e^s_{i,i} = \unit_\boldA$, then we say that the system $(e^s_{i,j})_{s,i,j}$ is 
\emph{unital}.
\end{definition}

If $\boldF$ is a \cstar-algebra that is generated by a matricial system $(e^s_{i,j})_{s,i,j}$, 
then
we call $(e^s_{i,j})_{s,i,j}$ a \emph{matricial generating system for $\boldF$}.
Thus a \cstar-algebra is finite-dimensional if and only if it has a matricial generating system.

\section{Preliminaries from effective mathematics}\label{sec:PrlmEff}

Herein, we present our frameworks for the computability of inductive limits 
in a computably indexed category and for the computability of ordered 
abelian groups.

\subsection{Inductive limits in computably indexed categories}

We now present a computable picture of the material discussed in \ref{subsec:BackIndLimCat}.
Thus, we again assume $\calC$ is a category with a zero object.  
However, we assume we have fixed an \emph{indexing} of $\calC$.  This consists of 
an assignment of natural numbers to the objects of $\calC$.
The number assigned to an object is referred to as one of its indexes. 
Each object has at least one index.  A number may not index more the one object.
Each morphism is indexed by at least one triple of natural numbers.
If $(e_0, e_1,e)$ indexes a morphism $\gamma$, then $e_0$ must index its 
domain and $e_1$ must index its co-domain.

We also assume this indexing is \emph{computable} in the sense that there is 
an algorithm that given an index $(e_0,e_1,e)$ of $\phi : A \rightarrow B$ and
an index $(e_1, e_2,e')$ of $\psi : B \rightarrow C$
computes an index of $\psi \circ \phi$ of the form $(e_0, e_2, e'')$.
Furthermore, we require that from an index $e$ of an object $A$ it is possible to 
compute an index of its identity morphism.

Now that the computable indexing is fixed, our remaining definitions are 
straightforward modifications of those in Section \ref{subsec:BackIndLimCat}.
Fix an inductive sequence $(A_n, \phi_n)_{n \in \N}$ of $\calC$.
We say that $(A_n, \phi_n)_{n \in \N}$ is \emph{computable} if there is an algorithm
that given $n \in \N$ computes indexes of $A_n$ and $\phi_n$.
Now suppose 
$(A, (\nu_n)_{n \in \N})$ is an inductive upper limit of $(A_n, \phi_n)_{n \in \N}$, 
and assume it is possible to compute an index of $\nu_n$ from $n$.
We say that an upper inductive limit $(A, (\nu_n)_{n \in \N})$ 
of $(A_n, \phi_n)_{n \in \N}$ is \emph{computable} with index $e \in \N$ if $e$ is the code of a pair $(e_0, e_1)$ so that 
$e_0$ is an index of $A$ and $e_1$ is an index of an algorithm that computes an index of 
$\nu_n$ from $n$.
We say that $(A, (\nu_n)_{n \in \N})$ is \emph{a computable inductive limit} if 
it is a computable inductive upper limit and if it is possible to compute from an index of a computable inductive upper limit 
$(B, (\mu_n)_{n \in \N})$ of $(A_n, \phi_n)_{n \in \N}$ an index of the reduction from $(A, (\nu_n)_{n \in \N})$
to $(B, (\mu_n)_{n \in \N})$.

\subsection{Abelian groups}

\begin{lemma}\label{lm:IndLimPresAbl}
Suppose $(G, (\nu_s)_{s \in \N})$ is an inductive limit of 
$(G_s, \phi_s)_{s \in \N}$ in the category of 
abelian groups. 
\begin{enumerate}
    \item If $(G_s^\#, \phi_s)_{s\in \N}$ is an inductive sequence in the category 
    of c.e. presentations of abelian groups, and if 
    $(G^\#, (\nu_s)_{s \in \N})$ is a computable inductive upper limit 
    of this sequence, then $(G^\#, (\nu_s)_{s \in \N})$ is 
    a computable inductive limit of $(G^\#, (\nu_s)_{s \in \N})$ in the category 
    of c.e. presentations of abelian groups.

    \item If $(G_s^\#, \phi_s)_{s\in \N}$ is an inductive sequence in the category 
    of computable presentations of abelian groups, and if 
    $(G^\#, (\nu_s)_{s \in \N})$ is a computable inductive upper limit 
    of this sequence, then $(G^\#, (\nu_s)_{s \in \N})$ is 
    a computable inductive limit of $(G^\#, (\nu_s)_{s \in \N})$ in the category 
    of computable presentations of abelian groups.
\end{enumerate}
\end{lemma}

\begin{proof}
Suppose $(G_s^\#, \phi_s)_{s\in \N}$ is an inductive sequence in the category 
    of c.e. presentations of abelian groups, and assume 
    $(G^\#, (\nu_s)_{s \in \N})$ is a computable inductive upper limit 
    of this sequence.  
    Let $(H^\#, (\mu_s)_{s \in \N}$ be a computable inductive upper limit of 
    $(G_s^\#, \phi_s)_{s \in \N}$ in the category of c.e. presentations of abelian groups.
    Thus, $(H, (\mu_s)_{s \in \N})$ is an inductive upper limit of 
    $(G_s, \phi_s)_{s \in \N}$ in the category of abelian groups. 
    Hence, there is a unique reduction $\lambda$ of  
$(G, (\nu_s)_{s \in \N})$ to $(H, (\mu_s)_{s \in \N})$.  

It only remains to show that $\lambda$ is a computable map from $G^\#$ to $H^\#$.  
Suppose $w$ is a $G^\#$ label of $g \in G$.  Then, there exists $s_0 \in \N$ so that 
$g \in \ran(\nu_{s_0})$ and so there is a $\ran(\nu_{s_0})^\#$ label $w'$ of $g$. 
Since the inclusion map is a computable map from $\ran(\nu_s)^\#$ to 
$G^\#$ uniformly in $s$, $s_0$ and $w'$ can be found by a search procedure. 
Then, by a search procedure, we can compute a $G_{s_0}^\#$ label $w''$ of a 
$g' \in G_{s_0}$ so that $\nu_{s_0}(g') = g$ (since $\nu_s$ is a computable map 
from $G_s^\#$ to $G^\#$ uniformly in $s$).  Since $\lambda(g) = \mu_{s_0}(g')$, 
it follows that we can now compute an $H^\#$ label of $\lambda(g)$ from $w$.

The proof of the second part is almost identical.
\end{proof}

\begin{proposition}\label{prop:CompIndLimPresAbl}
\begin{enumerate}
    \item Every computable inductive sequence of c.e. presentations of abelian 
    groups has a computable inductive limit.

    \item Every computable inductive sequence of computable presentations of abelian 
    groups has a computable inductive limit.
\end{enumerate}
\end{proposition}

\begin{proof}
Suppose $(G_s^\#, \phi_s)_{s \in \N}$ is an inductive sequence of c.e. presentations of 
abelian groups.  Let $(G, (\nu_s)_{s \in \N})$ be an inductive limit of 
$(G_s, \phi_s)_{s \in \N}$ in the category of abelian groups.  
We then construct, for each $s \in \N$, a c.e. presentation 
$\ran(\nu_s)^\#$ so that $\nu_s$ is a computable map 
from $\calG_s^\#$ to $\ran(\nu_s)^\#$ and so that 
the inclusion map is a computable map from $\ran(\nu_s)^\#$
to $\ran(\nu_{s+1})^\#$; namely, we declare 
$w$ to be a $\ran(\nu_s)^\#$ label of $g$ if it is a 
label of a $\nu_s$ preimage of $g$.  
Since $G = \bigcup_{s \in \N} \ran(\nu_s)$, by standard techniques
it is possible to construct a c.e. presentation $G^\#$ 
so that the inclusion map from $\ran(\nu_s)^\#$ to $G^\#$ is computable
uniformly in $s$.

By construction, $(G^\#, (\nu_s)_{s \in \N})$ is a computable inductive upper limit 
of $(G_s^\#, \phi_s)_{s \in \N}$ in the category of c.e. presentations of abelian groups. 
It now follows from Lemma \ref{lm:IndLimPresAbl} that 
$(G^\#, (\nu_s)_{s \in \N})$ is a computable inductive limit 
of $(G_s^\#, \phi_s)_{s \in \N}$.

The proof of the second part is nearly identical.
\end{proof}

\begin{lemma}\label{lm:IndLimPresAbl2}
Suppose $(G^\#, (\nu_s)_{s \in \N})$ is a computable 
inductive limit of $(G_s^\#, \phi_s)_{s \in \N}$ in either the category 
of c.e. presentations of abelian groups in or in the category of 
computable presentations of abelian groups. 
Then, $(G, (\nu_s)_{s \in \N})$ is an inductive limit of 
$(G_s, \phi_s)_{s \in \N}$ in the category of abelian groups. 
\end{lemma}

\begin{proof}
For the moment, assume $(G^\#, (\nu_s)_{s \in \N})$ is a computable 
inductive limit of $(G_s^\#, \phi_s)_{s \in \N}$ in the category 
of c.e. presentations of abelian groups.
By definition, $(G, (\nu_s)_{s \in \N})$ is an inductive upper limit
of $(G_s, \phi_s)_{s \in \N}$.  Suppose $(H, (\mu_s)_{s \in \N})$ is 
an inductive limit of this sequence, and let $\lambda$ be the reduction 
of $(H, (\mu_s)_{s \in \N})$ to $(G, (\nu_s)_{s \in \N})$.  
As in the proof of Proposition \ref{prop:CompIndLimPresAbl}, 
it is possible to define a c.e. presentation $H^\#$ of $H$ so that 
$(H^\#, (\mu_s)_{s \in \N})$ is an inductive limit of 
$(G_s^\#, \phi_s)_{s \in \N}$ in the category of c.e. presentations of
abelian groups.  Hence, $\lambda$ is an isomorphism. 
It follows that $(G, (\nu_s)_{s \in \N})$ is an inductive limit of 
$(G_s, \phi_s)_{s \in \N}$.

The case for the category of computable presentations of abelian groups
is identical.
\end{proof}

\subsection{Ordered abelian groups}

We now discuss c.e. presentations of ordered abelian groups 
and related concepts and results.

\subsubsection{Presentations}

\begin{definition}\label{def:CePresOrdAbl}
Suppose $\calG = (G, P)$ is an ordered abelian group.
\begin{enumerate}
    \item We say that $(G^\#, P)$ is  a 
    \emph{presentation} of $\calG$ if $G^\#$ is a presentation of $G$.

    \item We say that a presentation $(G^\#, P)$ of $\calG$ is \emph{c.e.}
    if $G^\#$ is c.e. and $P$ is a c.e. set of $G^\#$.
\end{enumerate}
\end{definition}

If $\calG^\# = (G^\#, P)$ is a c.e. presentation of an ordered abelian group, 
we say that $e \in \N$ is an \emph{index} of $\calG^\#$ if 
it is the code of a pair $(e_0, e_1)$ that consists of an index of $G^\#$
and a $G^\#$ index of $P$.  If $g \in G$ and $w \in F_\omega$, we say that 
$w$ is a \emph{$\calG^\#$ label} of $g$ if it is a $G^\#$ label of $g$.

\begin{definition}\label{def:CePresUntlOrdAbl}
Suppose $\calG$ is an ordered abelian group and let $u$ be a unit of $\calG$.
\begin{enumerate}
    \item We say that $(\calG^\#, u)$ is  a 
    \emph{presentation} of $(\calG,u)$ if $\calG^\#$ is a presentation of $\calG$.

    \item We say that a presentation $(\calG^\#, u)$ of $(\calG,u)$ is \emph{c.e.}
    if $\calG^\#$ is c.e.
\end{enumerate}
\end{definition}

If $(\calG^\#, u)$ is a c.e. presentation of $(\calG, u)$, then we say that 
$e \in \N$ is an \emph{index} of $(\calG^\#, u)$ if it codes a pair that 
consists of an index of $\calG^\#$ and a code of a $\calG^\#$ label of $u$.
When $\boldA$ is stably finite and unital, we let $\uK{\boldA^\#}= ((K_0(\boldA^\#), K_0(\boldA)^+), K_0(\unit_\boldA))$.

The following is a linchpin for our proof of Main Theorem \ref{mthm:ClsCePresDimGrp}.\ref{mthm:ClsCePresDimGrp::Untl}.

\begin{lemma}\label{lm:CompCnvxGen}
If $\calH = (H, H\cap \Z^n_{\geq 0})$ is the convex subgroup of $\Z^n$ generated by an element $x$ of $\Z^n$, then $H$ is a computable set of $\Z^n$.
\end{lemma}

\begin{proof}
We first show that for every $g \in \Z^n$, 
$g \in H$ if and only if there is a positive integer $n$ so that 
$-nx \leq g \leq n x$.  Let $g \in \Z^n$.  
On the one hand, if such a positive integer exists, then the convexity of $\calH$
ensures that $g \in H$.  Conversely, since by Lemma \ref{lm:UntlCnvx} $x$ is an order 
unit of $\calH$, if $g \in H$, then such a positive integer $n$ must exist.

Given $g \in \Z^n$, by inspection of components, it is possible to determine if 
there is a positive integer $n$ so that $-n x \leq g \leq n x$.  
Thus, by what has just been shown, $H$ is a computable set of $\Z^n$.
\end{proof}

We note that the proof of Lemma \ref{lm:CompCnvxGen} is uniform.

\subsubsection{Inductive limits}\label{subsubsec:PrlmCompIndLimOrdAbl}

We now apply our framework for computable inductive limits to the category of 
ordered abelian groups.  We have already discussed how we computably index this 
category.  The following not only tell us that computable inductive sequences 
in this category have inductive limits, but also how to find them.

The following is an immediate consequence of Lemma \ref{lm:IndLimPresAbl}.

\begin{corollary}\label{cor:CompIndLimOrdAbl1}
Suppose $(\calG_s^\#, \phi_s)_{s \in \N}$ is a computable inductive sequence of 
c.e. presentations of ordered abelian groups.  Assume $(\calG, (\nu_s)_{s\in \N})$ is an 
inductive limit of $(\calG_s, \phi_s)_{s \in \N}$, and assume also that 
$\calG^\#$ is a c.e. presentation of $\calG$ so that 
$(\calG^\#, (\nu_s)_{s \in \N})$ is a computable upper limit of 
$(\calG_s^\#, \phi_s)_{s \in \N}$.  Then $(\calG^\#, (\nu_s)_{s \in \N})$ is 
a computable inductive limit of $(\calG_s^\#, \phi_s)_{s \in \N}$.
\end{corollary}


\begin{corollary}\label{cor:CompIndLimOrdAbl2}
Every computable inductive sequence of c.e. presentations of ordered abelian groups 
has a computable inductive limit.
\end{corollary}

\begin{proof}[Proof sketch]
Suppose $(\calG_s^\#, \phi_s)_{s \in \N}$ is a computable inductive sequence of 
c.e. presentations of ordered abelian groups.  
Let $\calG_s^\# = (G_s^\#, P_s)$.
By Proposition \ref{prop:CompIndLimPresAbl}, 
$(G_s^\#, \phi_s)_{s \in \N}$ has a computable inductive limit $(G^\#, (\nu_s)_{s \in \N}$ 
in the 
category of c.e. presentations of abelian groups. 

Set $P = \bigcup_{s \in \N} \nu_s[P_s]$.   
Let $\calG = (G, P)$, and let $\calG^\# = (G^\#, P)$.  
By Lemma \ref{lm:IndLimPresAbl2}, $(G, \nu_s)_{s \in \N}$ is 
an inductive limit of $(G_s, \phi_s)_{s \in \N}$. 
It then follows from Lemma \ref{lm:IndLimOrdAbl} that $(\calG, (\nu_s)_{s \in \N})$
is an inductive limit of $(\calG_s, \phi_s)_{s \in \N}$.

We claim that $\calG^\#$ is c.e.  It is only required to prove that $P$ is a c.e.
set of $G^\#$.  For each $s$, $P_s$ is a c.e. set of $G_s^\#$ uniformly in $s$. 
At the same time, $\nu_s$ is a computable map from $G_s^\#$ to $G^\#$ 
uniformly in $s$.  Hence, $\nu_s[P_s]$ is a c.e. set of $G^\#$ uniformly in $s$, and so 
$P$ is a c.e. set of $G^\#$.

Finally, since $\nu_s$ is a computable map from $G_s^\#$ to $G^\#$ uniformly in $s$, 
it follows that $(\calG^\#, (\nu_s)_{s \in \N})$ is a computable inductive upper 
limit of $(\calG_s^\#, \phi_s)_{s \in \N}$.  Hence, by Corollary 
\ref{cor:CompIndLimOrdAbl1}, $(\calG^\#, (\nu_s)_{s \in \N})$ is 
a computable inductive limit of $(\calG_s^\#, \phi_s)_{s \in \N}$.
\end{proof}

\subsubsection{Dimension groups}

The material in this section will be used in our proof of 
Main Theorem \ref{mthm:ClsCePresDimGrp}.

\begin{definition}\label{def:EffShen}
Suppose $\calG^\#$ is a presentation of an ordered abelian group. 
 We say that $\calG^\#$ has the \emph{effective Shen property} if 
     from a $(\Z^n, \calG^\#)$ index of a homomorphism $\theta$ from $\Z^n$ to $\calG^\#$
    and an $\alpha \in \ker(\theta)$, it is possible to compute a positive integer $p$, an index of 
   an ordered group homomorphism $\varphi:\Z^n\to \Z^p$, and a $(\Z^p, \calG^\#)$ index of
    a homomorphism $\theta':\Z^p\to \calG$
    so that $(\phi, \theta')$ factors $\theta$ at $\alpha$. 
\end{definition}

\begin{lemma}\label{lm:EffShen}
Every c.e. presentation of a dimension group has the effective Shen property.
\end{lemma}

\begin{proof}[Proof sketch]
Suppose $\calG^\#$ is a c.e. presentation of a dimension group.  
By Proposition \ref{prop:ShenPrp}, $\calG$ has the Shen property.  
Given a homomorphism $\theta : \mathbb{Z}^n \to \calG$ and $\alpha \in \ker \theta$, the ordered  homomorphisms required by the effective Shen property do in fact exist since $\calG$ has the (ordinary) Shen property.  As these maps are all completely determined 
by their values on the standard  bases, they can now be found by a search procedure.
\end{proof}

\subsection{\cstar-algebras}

We now attend to some preliminary developments from the computability theory 
of \cstar-algebras.  We first add to existing results on computable weak stability 
so as to compute matricial systems within c.e. presentations of AF algebras. 
These results will be crucial in our proof of the Effective Glimm Lemma. 
We then present some material on computing inductive limits of 
\cstar-algebras, and finally our formal definition of computable AF certificate.

\subsubsection{Computable weak stability results}

\begin{lemma}\label{lm:UnionEWS}
Suppose $\mathcal{R}(x_1, \ldots, x_n)$ and $\mathcal{S}(x_1, \ldots, x_n)$ are finite sets of relations 
as in Definition \ref{def:CompWklyStbl}. 
Assume further that $\mathcal{R}$ is computably weakly stable and that there is a computable $\Delta : \N \rightarrow \N$ 
so that for every \cstar-algebra $\boldA$ and all $a_1, \ldots, a_n \in \boldA$, if $\boldA \models\mathcal{R}(a_1, \ldots, a_n)$, 
and if $(\mathcal{R} \cup \mathcal{S})^\boldA(a_1, \ldots, a_n) \leq 2^{-\Delta(k)}$, then there exist 
$a_1', \ldots, a_n' \in \boldA$ so that $\boldA \models (\mathcal{R} \cup \mathcal{S})(a_1', \ldots, a_n')$ and so that 
$\max_j \norm{a_j - a_j'} < 2^{-k}$.  Then, $\mathcal{R} \cup \mathcal{S}$ is effectively weakly stable.
\end{lemma}

\begin{proof}
Let $\Delta_1$ be a computable modulus of weak stability for $\mathcal{R}$.  Let $M$ denote the maximum of the constants that 
appear in $\mathcal{S}$ (that is, the $C_j$'s).  
It follows that there is a computable $g : \N \rightarrow \N$
so that for every \cstar-algebra $\boldA$, $g$ is a modulus of uniform continuity for $\mathcal{S}^\boldA$ on 
$B(\zerovec; M)^n$.
Set $\Delta_2(k) = \Delta_1(\max\{k + 1, g(\Delta(k + 1))\})$.  

Suppose $c_1, \ldots, c_n \in \boldA$ and $(\mathcal{R} \cup \mathcal{S})^\boldA(c_1, \ldots, c_n) < 2^{-\Delta_2(k)}$.  
Let 
\[
a = \max\{1, g(\Delta(k + 1)) - k\}.
\]
Since $\Delta_2(k) \geq \Delta_1(k + a)$, there exist 
$a_1, \ldots, a_n \in A$ so that $\boldA \models \mathcal{R}(a_1, \ldots, a_n)$ and so that 
$\max_j \norm{a_j - c_j} < 2^{-(k + a)}$.  
Since $\boldA \models \mathcal{R}(a_1, \ldots, a_n)$, 
$(\mathcal{R} \cup \mathcal{S})^\boldA(a_1, \ldots, a_n) = \mathcal{S}^\boldA(a_1, \ldots, a_n)$.  
As $k + a \geq g(\Delta(k + 1))$, it follows that 
$\mathcal{S}^\boldA(a_1, \ldots, a_n) < 2^{-\Delta(k+1)}$.   
Hence, there exist $a_1', \ldots, a_n' \in \boldA$ so that 
$\boldA \models (\mathcal{R} \cup \mathcal{S})(a_1', \ldots, a_n')$ and 
$\max_j \norm{a_j - a_j'} < 2^{-(k + 1)}$. 
Therefore, $\max_j \norm{c_j - a_j} < 2^{-k}$.
\end{proof}

We note that the proof of Lemma \ref{lm:UnionEWS} is uniform in that an index of a modulus of weak stability for $\mathcal{R} \cup \mathcal{S}$ can be computed from an index of $\Delta$ and an index of a modulus of weak stability for $\mathcal{R}$.

\begin{proposition}\label{prop:MutOrthoGen}
The following are defined by computably weakly stable sets of relations. 
\begin{enumerate}
    \item For each $n \in \N$, the property of being an $n$-tuple of mutually orthogonal 
    projections. \label{prop:MutOrthoGen:Proj}

    \item For each $m, n_1, \ldots, n_m \in \N$, the property of being 
    a (unital) matricial system of type $(n_1, \ldots, n_m)$. \label{prop:MutOrthoGen:Gen}
\end{enumerate}
\end{proposition}

\begin{proof}
(\ref{prop:MutOrthoGen:Proj}): The property of being a projection is effectively 
weakly stable (see \cite[Examples 12]{FoxGoldbringHart.2024+}).
Define $\Delta(0, k)$ arbitarily, and for positive $n$ let 
$\Delta(n,k) = \lceil \log_2( \delta_1(2^{-k}, n)) \rceil$.
Apply Lemmas \ref{lm:844Strung} and \ref{lm:UnionEWS}.

(\ref{prop:MutOrthoGen:Gen}): We use Lemma \ref{lm:UnionEWS}. However, some extra steps must be taken 
(compared to (\ref{prop:MutOrthoGen:Proj})) to show that the required function $\Delta$ exists.  

$\Delta(0,k)$ can be defined arbitarily.  
Suppose $n, k \in \N$, and assume $n$ is positive. 
Let $k_0 = 1-\lfloor \log_2(\delta_2(2^{-k}n^{-1}, n)) \rfloor$, and 
set $\Delta(n,k) = 1 - \lfloor \log_2(n^{-2} \delta_1(2^{-k_0}, n)  ) \rfloor$.

Assume $n_1, \ldots, n_m$ are positive integers so that 
$n_1^2 + \ldots + n_m^2 = n$.  
Suppose that for each $s \in \{1, \ldots, m\}$, $(e^s_{i,j})_{i,j}$ is an $n_s \times n_s$
system of matrix units of $\boldA$, and assume also that 
$\max_{i,j,j',j'} \norm{e^s_{i,j} e^{s'}_{i',j'}} \leq 2^{-\Delta(n,k)}$ when $s \neq s'$. 
Set $p_s = \sum_i e^s_{i,i}$.  Then, 
$p_1, \ldots, p_m$ are projections, and $\max_{s \neq s'} \norm{p_s p_{s'}} \leq n^2 2^{-\Delta(n,k)}$. 
By definition, $n^2 2^{-\Delta(n,k)} < \delta_1(2^{-k_0}, n)$.  
Thus, by Lemma \ref{lm:844Strung}, there exist mutually orthogonal projections 
$q_1, \ldots, q_n \in \boldA$ so that $\max_j \norm{p_j - q_j} < 2^{-k_0}$. 

Let $\boldA_s = q_s \boldA q_s$.  
Since $\norm{p_s - q_s} < 2^{-k_0}$, 
\[
\norm{e^s_{i,j} - q_s e^s_{i,j} q_s} = \norm{p_se^s_{i,j}p_s - q_s e^s_{i,j} q_s} < 2^{-k_0}.
\]
Thus,  $\max_{s,i,j} d(e^s_{i,j}, \boldA_s) < 2^{-k_0}$. 
As $2^{-k_0} < \delta_2(2^{-k}, n)$, 
it follows from Lemma \ref{lm:847Strung} that $\boldA_s$ contains an $n_s \times n_s$ system 
$(f^s_{i,j})_{s,i,j}$ of matrix units so that $\norm{e^s_{i,j} - f^s_{i,j}} < n^{-1}2^{-k}$. 
And, since $q_1, \ldots, q_m$ are mutually orthogonal, it follows that 
$(f^s_{i,j})_{s,i,j}$ is a matricial system.  Moreover, if $\sum_{s,i} e^s_{i,i} = \unit_\boldA$, 
then $\norm{ \unit_\boldA - \sum_{s,i} f^s_{i,i}} < 1$.   
Hence, since each $f^s_{i,i}$ is a projection $\sum_{s,i} f^s_{i,i} = \unit_\boldA$.  
\end{proof}

\subsubsection{Inductive limits}

The c.e. presentations of \cstar-algebras form a category in which the morphisms 
are the computable $\star$-homomorphims.  
That is, $\psi$ is a morphism from $\boldA^\#$ to $\boldB^\#$ if it is a 
computable $\star$-homomorphism from $\boldA^\#$ to $\boldB^\#$.
There is a natural indexing of this category, but note that a number $e$ 
will always index several presentations.  However, there is no harm in identifying all presentations 
indexed by a number $e$ even if they are presentations of different algebras as there is 
an obvious computable unital $\star$-isomorphism between any two of them; namely
map the $n$-th special point to the $n$-th special point.

We present two results on computable inductive limits of \cstar-algebras that 
parallel the results in Section \ref{subsubsec:PrlmCompIndLimOrdAbl}.

\begin{corollary}\label{cor:CompIndLimAlg1}
Suppose $(\boldA_s^\#, \phi_s)_{s \in \N}$ is a computable inductive sequence in the category 
of c.e. \cstar-algebra presentations, and let $(\boldA, (\nu_s)_{s \in \N})$
be an inductive limit of $(\boldA_s, \phi_s)_{s \in \N}$. 
If $\boldA^\#$ is a c.e. presentation so that $(\boldA^\#, (\nu_s)_{s \in \N})$
is a computable inductive upper limit of $(\boldA_s^\#, \phi_s)_{s \in \N}$, 
then $(\boldA^\#, (\nu_s)_{s \in \N})$ is a computable inductive limit of 
$(\boldA_s^\#, \phi_s)_{s \in \N}$.  
\end{corollary}

\begin{proof}[Proof sketch]
Suppose $(\boldB^\#, (\mu_s)_{s \in \N})$ is a computable inductive upper limit of 
$(\boldA_s^\#, \phi_s)_{s \in \N}$, and let $\lambda$ be the reduction of 
$(\boldA, (\nu_s)_{s \in \N})$ to $(\boldB, (\mu_s)_{s \in \N})$.
Let $g$ be a generated point of $\boldA^\#$.
Wait for $s_0 \in \N$ and a generated point $g_0$ of $\boldA_{s_0}^\#$ so that 
$\nu_{s_0}(g_0)$ is sufficiently close to $g$.  By Lemma \ref{lm:IndLimAlg}, 
such a number and generating point must exist.  Since $\lambda$ is $1$-Lipschitz,
$\lambda(g)$ is sufficiently close to $\lambda(\nu_{s_0}(g_0)) = \mu_{s_0}(g_0)$ which 
can be computed with arbitrarily good precision.
\end{proof}

The following is due to the second author \cite{Goldbring.2024+}.

\begin{corollary}\label{cor:CompIndLimAlg2}
Suppose $(\boldA_s^\#, \phi_s)_{s \in \N}$ is a computable inductive sequence in the category 
of c.e. \cstar-algebra presentations.
\begin{enumerate}
    \item $(\boldA_s^\#, \phi_s)_{s \in \N}$ has a computable inductive limit. 

    \item If $(\boldA_s^\#, \phi_s)_{s \in \N}$ is a computable inductive sequence 
    in the category of 
computable \cstar-algebra presentations, and if each $\phi_s$ is injective, then 
$(\boldA_s^\#, \phi_s)_{s \in \N}$ has a computable inductive limit in the category of 
computable \cstar-algebra presentations. \label{cor:CompIndLimAlg2::Comp}
\end{enumerate}
\end{corollary}

\begin{proof}[Proof sketch]
We use a technique similar to that used in the proof of Corollary \ref{cor:CompIndLimOrdAbl2}
to construct $\boldA^\#$.  Part (\ref{cor:CompIndLimAlg2::Comp}) follows from 
Lemma \ref{lm:IndLimAlg}.\ref{lm:IndLimAlg::norm}.
\end{proof}

\subsubsection{Computable AF certificates}

We conclude this section by formally defining computable AF certificates.

\begin{definition}\label{def:CompAFCert}
Suppose $(F_j, \psi_j)_{j \in \N}$ is an AF certificate for $\boldA$, and assume $\boldA^\#$ is a 
presentation of $\boldA$.  We say that $(F_j, \psi_j)_{j \in \N}$ is a \emph{computable AF certificate of $\boldA^\#$} if $(F_j)_{j \in \N}$ is computable and if each $\psi_j$ is a computable map
from $\bigoplus_{n \in F_j}M_n(\C)$
    to $\boldA^\#$.
\end{definition}

\subsection{Bratteli diagrams}\label{subsec:PreCompBrat}

We set forth our framework for the computability of Bratteli diagrams.

\begin{definition}\label{def:PresBrat}
Suppose $\BratD$ is a Bratteli diagram.
\begin{enumerate}
    \item A \emph{presentation} of $\BratD$ is a pair $(\BratD, \nu)$ where $\nu$ is a surjection
     of $\N$ onto $V_\BratD$.

    \item The \emph{kernel} of a presentation $(\BratD, \nu)$ is 
    $\{(m,n)\ :\ \nu(m) = \nu(n)\}$.
\end{enumerate}
\end{definition}


\begin{definition}\label{def:CompPresBrat}
Suppose $\BratD^\# = (\BratD, \nu)$ is a presentation of a Bratteli diagram.  We say that $\BratD^\#$ is 
\emph{computable} if it satisfies the following:
\begin{enumerate}
    \item The kernel of $\BratD^\#$ is computable. 

    \item $L_\BratD \circ \nu$ is computable. 

    \item The map $(m,n) \mapsto E_\BratD(\nu(m), \nu(n))$ 
    is computable.

    \item If $\BratD$ is labeled, then $\Lambda_\BratD \circ \nu$ is computable.

    \item The sequence $(\# L_\BratD^{-1}[\{n\}])_{n \in \N}$ is computable.
\end{enumerate}
\end{definition}

If $\BratD$ and $\BratD'$ are Bratteli diagrams, then a 
\emph{computable telescoping} of $\BratD'$ to $\BratD$ is a computable sequence of integers that 
satisfies the conditions of Definition \ref{def:Tel}.

\begin{definition}\label{def:CompEqvBrat}
Suppose $\BratD$ and $\BratD'$ are Bratteli diagrams.   We say that 
$\BratD$ and $\BratD'$ are \emph{computably equivalent} 
if there is a sequence $\BratD_0 = \BratD$, $\ldots$, $\BratD_m = \BratD'$
so that for each $j < m$, either $\BratD_j$ and $\BratD_{j+1}$ are isomorphic 
or one is a computable telescoping of the other.
\end{definition}

This concludes our discussion of preliminary developments.  We now 
attend to formally stating and proving our Effective Glimm Lemma.

\section{An effective Glimm Lemma}\label{sec:EffGlimm}

Throughout this section, $\boldA$ denotes a stably finite unital \cstar-algebra and $\boldA^\#$ denotes a c.e. presentation of $\boldA$. 

Suppose $\boldB$ is a c.e. closed $\star$-subalgebra of $\boldA^\#$.  Then 
from an $\boldA^\#$ index of $\boldB$ and an index of $\boldA^\#$, 
it is possible to compute an $\boldA^\#$ index of 
a sequence $(c_n)_{n \in \N}$ that is dense in $\boldB$.
We then define $\boldB^\#$ to be the presentation of $\boldB$ whose $n$-th special point is $c_n$. 
Since $(c_n)_{n \in \N}$ is a computable sequence of $\boldA^\#$, it follows that 
$\boldB^\#$ is a c.e. 
presentation of $\boldB$.  Furthermore, the inclusion map is a computable map 
from $\boldB^\#$ to $\boldA^\#$.  

\begin{theorem}[Effective Glimm Lemma]\label{thm:EffGlimm}
There is a computable function $\Delta:\mathbb{N}^2\to \mathbb{N}$
so that for every c.e.-closed $\star$-subalgebra $\boldB$ of $\boldA^\#$,
 all $n,k \in \N$, and every c.e.-closed $\star$-subalgebra $\boldF$ of $\boldA^\#$ for which $\unit_\boldA \in \boldF$, 
if $\dim(\boldF) = n$, and if $B(\boldB; 2^{-\Delta(n,k)})$ contains a matricial generating 
system of $\boldF$, then there is a computable unitary $u$ of $\boldA^\#$ so that 
$\norm{u - \unit_\boldA} < 2^{-k}$ and $u^*\boldF u \subseteq \boldB$. 
Furthermore, an $\boldA^\#$ index of $u$ can be computed from an index of $\boldA^\#$ and $\boldA^\#$ indexes of $\boldF$ and $\boldB$. 
\end{theorem}

We approach this proof through a sequence of lemmas.
For the sake of the first lemma, we define a function $\Delta_1 : \N^2 \rightarrow \N$ as 
follows.  
Define $\Delta_1(0,k)$ arbitrarily. 
When $n \geq 1$, define $\Delta_1(n,k)$ to be the least natural number so that 
$2^{-\Delta_1(n,k)} < \delta_0(2^{-(k+1)}, n)$.  
We note that by induction, $\delta_0(2^{-(k+1)}, n)$ is a positive rational number when $n > 0$. 
Thus, $\Delta_1(n,k)$ is defined for all $n,k \in \N$.
Furthermore, the recursive definition of $\delta_0$ ensures that $\Delta_1$ is computable.

\begin{lemma}\label{lm:PertMutOrth}
For all $n,k \in \N$, if 
$p_1, \ldots, p_n \in B(\Proj(\boldB); 2^{-\Delta_1(n,k)})$ are mutually orthogonal computable
projections of $\boldA^\#$, then $\boldB$ contains mutually orthogonal computable projections 
$q_1, \ldots, q_n$ of $\boldA^\#$ so that 
$\max_j \norm{p_j - q_j} < 2^{-k}$.  Furthermore, if 
$\sum_j p_j = \unit_\boldA$, then $\sum_j q_j = \unit_\boldA$.
\end{lemma}

\begin{proof}
Suppose $p_1, \ldots, p_n \in B(\Proj(\boldB); 2^{-\Delta_1(n,k)})$
are mutually orthogonal computable projections of $\boldA^\#$.  It follows from Lemma \ref{lm:842Strung}
that there exist mutually orthogonal projections 
$q_1, \ldots, q_n \in \boldB$ so that $\max_j \norm{q_j - p_j} < 2^{-k}$. 
Let $\Omega$ denote the set of all mutually orthogonal $(a_1, \ldots, a_n) \in \boldB^n$.
By Proposition \ref{prop:MutOrthoGen}, $\Proj(\boldB)^n \cap \Omega$ is a c.e. closed 
set of $(\boldB^\#)^n$. Since the inclusion map from $\boldB^\#$ to $\boldA^\#$ is computable, 
$\Proj(\boldB) \cap \Omega$ is a c.e. closed set of $(\boldA^\#)^n$. Then, since 
$p_1, \ldots, p_n$ are computable points of $\boldA^\#$, $B((p_1, \ldots, p_n); 2^{-k})$ is a 
c.e. open set of $(\boldA^\#)^n$.  Thus, by \cite[Lemma 1.11]{UHFPaper},
$\Proj(B) \cap \Omega$ contains a computable point
$(q_1, \ldots, q_n)$ of $(\boldA^\#)^n$.  
Hence, $q_1, \ldots, q_n$ are computable points of $\boldA^\#$.  

Finally, the definitions of $\delta_0$ and $\Delta_1$ ensure that 
if $\sum_j p_j = \unit_\boldA$, then $\sum_j q_j = \unit_\boldA$.  (See proof of Lemma \ref{lm:842Strung}.)
\end{proof}

Set $\Delta_2(n,k) = n + k + 2$. 

\begin{lemma}\label{lm:PrtlIso}
Suppose $p_1$, $\ldots$, $p_n$, $q_1$, $\ldots$, $q_n$ are computable projections of $\boldA^\#$ so that 
$p_ip_j = q_iq_j = \zerovec$ whenever $i \neq j$. Assume also that $\max_j \norm{p_j - q_j} < 2^{-\Delta_2(n,k)}$. 
Then, \cstar$(\{p_1, \ldots, p_n, q_1, \ldots, q_n, \unit_\boldA\})$ contains a computable partial isometry 
$v$ of $\boldA^\#$ so that for all $j \in \{1, \ldots, n\}$, $v^*p_jv = q_j$ and 
$\max_j \norm{p_j - p_jvq_j}< 2^{-k}$. Furthermore, if $\sum_j p_j = \unit_\boldA$, 
then $v$ is unitary and $\norm{v - \unit_\boldA} < 2^{-k}$.
\end{lemma}

\begin{proof}
For each $j \in \{1, \ldots, n\}$, let $u_j = \unit_\boldA - p_j - q_j - 2q_jp_j$.  
By \cite[Lemma 8.3.6]{Strung.2021}, $u_jq_ju_j^* = p_j$ and 
$\norm{\unit_\boldA - u_j} \leq \sqrt{2} \norm{p_j - q_j}$.  
Set $v = \sum_j p_ju_jq_j$.  It is shown in the proof of \cite[Lemma 8.4.3]{Strung.2021} that 
$v$ is a partial isometry and that $v^*p_jv = q_j$ for all $j \in \{1, \ldots, n\}$.  We also have that 
\begin{eqnarray*}
\norm{p_j - p_j v q_j} & = & \norm{p_j - p_ju_jq_j} \\
& = & \norm{p_j - p_ju_jp_j + p_ju_jp_j - p_ju_jq_j} \\
& \leq & \norm{p_j}\norm{\unit_A - u_j} + \norm{p_j} \norm{u_j} \norm{p_j - q_j} \\
& \leq & \sqrt{2} \norm{p_j - q_j} + \norm{p_j - q_j}\\
& = & (1 + \sqrt{2})\norm{p_j - q_j} \\
& \leq & (1 + \sqrt{2})2^{-(n + k+2)} < 2^{-n}2^{-k} < n^{-1}2^{-k} \leq 2^{-k}.
\end{eqnarray*}

Finally, suppose $\sum_j p_j = \unit_\boldA$.  
Then, 
\[
\norm{\unit_\boldA - \sum_j q_j} = \norm{\sum_j (q_j - p_j)} \leq n\cdot2^{-n} < 1.
\]
Hence, $\sum_j q_j = \unit_\boldA$.  Moreover, since $v^*p_jv = q_j$ for all $j \in \{1, \ldots,n\}$, 
it also follows that 
$v$ is unitary.  In addition,
\begin{eqnarray*}
\norm{\unit_\boldA - v} & = & \norm{\sum_{j = 1}^n (p_j - p_ju_jq_j)}\\
& \leq & n \max_j \norm{p_j - p_ju_jq_j} \\
& < & 2^{-k}.
\end{eqnarray*}
\end{proof}

When $n,k \in \N$ and $n > 0$, define $\Delta_3(n,k)$ to be 
the smallest natural number so that $2^{-\Delta_3(n,k)} < \delta_2(2^{-(k+1)}, n)$;
define $\Delta_3(0,k)$ arbitrarily.

\begin{lemma}\label{lm:PertMtrxU}
Suppose $(e_{i,j})_{i,j}$ is a computable $n \times n$ system of matrix units of $\boldA^\#$, and assume 
$\max_{i,j} d(e_{i,j}, \boldB) \leq 2^{-\Delta_3(n,k)}$.  
Then, $\boldB$ contains a computable $n \times n$ system $(f_{i,j})_{i,j}$ of matrix units so that 
$\max_{i,j} \norm{e_{i,j} - f_{i,j} } < 2^{-k}$.  
\end{lemma}

\begin{proof}
By definition, $2^{-\Delta_3(n,k)} < \delta_2(2^{-(k+1)}, n)$.  It then follows from 
Lemma \ref{lm:847Strung}
that $\boldB$ contains an $n \times n$ system $(g_{i,j})_{i,j}$ of matrix units so that 
$\max_{i,j} \norm{e_{i,j} - g_{i,j} } < 2^{-k}$.  
However, the property of being an $n \times n$ system of matrix units is defined by a computably weakly stable 
set of relations (see \cite[Examples 12]{FoxGoldbringHart.2024+}).  
Since $e_{i,j}$ is a computable point of $\boldA^\#$, 
$B( (e_{i,j})_{i,j}; 2^{-k}) \cap \boldB^{n^2}$ is a c.e. open set of 
$(\boldB^\#)^{n^2}$.  Therefore, by \cite[Lemma 1.11]{UHFPaper}, there is a computable $n \times n$ system 
$(f_{i,j})_{i,j}$ of matrix units of $\boldB^\#$ so that $\max_{i,j} \norm{e_{i,j} - f_{i,j} } < 2^{-k}$. 
Again, since the inclusion map from $\boldB^\#$ to $\boldA^\#$ is computable, 
$(f_{i,j})_{i,j}$ is also a computable $n \times n$ system of matrix units of $\boldA^\#$.
\end{proof}

For every $n,k \in \N$, let 
\[
\Delta_4(n,k) = \max\{ \max_{m \leq k} \Delta_1(m, \max_{n' \leq n} \Delta_3(n',k) + 2), 1 + \max_{n' \leq n} \Delta_3(n', k)\}.
\]

\begin{lemma}\label{lm:PertMtrxGen}
Suppose $(e_{i,j}^s)_{s,i,j}$ is a type-$(n_1, \ldots, n_m)$ computable matricial system of $\boldA^\#$ so that 
$\max_{i,j,s} d(e^s_{i,j}, \boldB) \leq 2^{-\Delta_3(\sum_j n_j^2,k)}$.  Then, $\boldB$ contains 
a computable type-$(n_1, \ldots, n_m)$ matricial system $(f^s_{i,j})_{s,i,j}$ of $\boldA^\#$ 
so that 
$\max_{s,i,j} \norm{e^s_{i,j} - f^s_{i,s}} < 2^{-k}$.  
\end{lemma}

\begin{proof}
For each $s \in \{1, \ldots, m\}$, let $p_s = \sum_i e^s_{i,i}$.  
Set $n = n_1^2 + \ldots, n_m^2$, and set $k_0 = \max_{n'\leq n} \Delta_3(n',k) + 2$.  
By definition, $\Delta_4(n,k) \geq \Delta_1(m,k_0)$.  
Hence, by Lemma \ref{lm:PertMutOrth}, $\boldB$ contains computable and mutually orthogonal projections 
$q_1, \ldots, q_m$ of $\boldA^\#$ so that $\max_i \norm{q_i - p_i} < 2^{-k_0}$.  

For each $s \in \{1, \ldots, m\}$, let $\boldB_s = q_s \boldB q_s$.  Fix $b \in \boldB$.  Then:
\begin{eqnarray*}
\norm{e^s_{i,j} - q_s b q_s} & = & \norm{p_s e^s_{i,j} p_s - q_s b q_s} \\
& \leq & \norm{p_s e^s_{i,i} p_s - p_s e^s_{i,j} q_s} + \norm{p_s e^s_{i,j} q_s - q_s b q_s} \\
& \leq & \norm{p_s - q_s} + \norm{p_s e^s_{i,j} - q_s b } \\
& \leq & \norm{p_s - q_s} + \norm{p_s e^s_{i,j} - q_s e^s_{i,j} } + \norm{q_s e^s_{i,j} - q_s b } \\
& \leq & 2\norm{p_s - q_s} + \norm{e^s_{i,j} - b}.
\end{eqnarray*}
Hence, 
\[
d(e^s_{i,j}, \boldB_s) \leq 2^{-k_0 + 1} + 2^{-\Delta_4(n,k)}.
\]
However, by the definitions of $k_0$ and $\Delta_4$, 
$2^{-k_0 + 1} + 2^{-\Delta_4(n,k)} < 2^{-\Delta_3(n_s, k)}$. 
Hence, by Lemma \ref{lm:PertMtrxU}, $\boldB_s$ contains a computable $n_s \times n_s$ system 
$(f^s_{i,j})_{i,j}$ of matrix units of $\boldA^\#$ so that 
$\max_{i,j} \norm{e^s_{i,j} - f^s_{i,j}} < 2^{-k}$ and so that $\sum_i f^s_{i,i} = q_s$.  
Since $q_1, \ldots, q_m$ are mutually orthogonal, it follows 
that $(f^s_{i,j})_{s,i,j}$ is a matricial system.
\end{proof}

We are now ready to prove the Effective Glimm Lemma.

\begin{proof}[Proof of Theorem \ref{thm:EffGlimm}]
Let $k \in \N$.  
We may define $\Delta(0,k)$ arbitrarily.  So, assume $n$ is a positive integer.  
Let $P$ denote the set of all square partitions of $n$; that is, 
$P$ is the set of all sequences $(n_1, \ldots, n_m)$ of positive integers 
so that $\sum_j n_j^2 = n$.  
We note that $P \neq \emptyset$ (set each $n_j = 1$).
Set:
\begin{eqnarray*}
k_0 & = & \lceil \log_2(n 2^{k+1}) \rceil\\
\Delta(n,k) & = & \max\{\Delta_5(n, \Delta_2(\sum_j n_j, k_0))\ :\ (n_1, \ldots, n_m) \in P\}. 
\end{eqnarray*}
We remark that $\Delta_2(\sum_j n_j, k_0) \geq k_0$.

Suppose $\boldF$ is a c.e. closed $\star$-subalgebra of $\boldA^\#$ that contains $\unit_\boldA$, and assume $\dim(\boldF) = n$.  
Further, assume $B(\boldB; 2^{-\Delta(n,k)})$ contains a matricial generating 
system of $\boldF$.  

We first show that there is a unital matricial generating system $(g^s_{i,j})_{s,i,j}$ of $\boldF$
so that $g^s_{i,s}$ is a computable point of $\boldA^\#$ uniformly in $s,i,j$ and so that 
$\max_{s,i,j} d(g^s_{i,j}, \boldB) < 2^{-\Delta(n,k)}$. 
In order to demonstrate this, when $(n_1, \ldots, n_m) \in P$, let 
$\mathcal{G}_{n_1, \ldots, n_m}$ denote the set of all type-$(n_1, \ldots, n_m)$ unital 
matricial systems $(e^s_{i,j})_{s,i,j}$ of $\boldF$.  By Proposition \ref{prop:MutOrthoGen}, $\mathcal{G}_{n_1, \ldots, n_m}$ is a 
c.e. closed set of $(\boldF^\#)^n$ and hence of $(\boldA^\#)^n$.  
Since $\boldB$ is a c.e. closed set of $\boldA^\#$, $B(\boldB; 2^{-\Delta(n,k)})$ is a c.e. open set of 
$\boldA^\#$.  It follows that the set of all $(n_1, \ldots, n_m) \in P$ so that 
$\mathcal{G}_{n_1, \ldots, n_m} \cap B(\boldB; 2^{-\Delta(n,k)}) \neq \emptyset$ is c.e. uniformly in $n,k$.
Furthermore, there is a unique $(n_1, \ldots, n_m) \in P$ so 
that $\mathcal{G}_{n_1, \ldots, n_m} \cap B(\boldB; 2^{-\Delta(n,k)}) \neq \emptyset$.
By \cite[Lemma 1.11]{UHFPaper}, this intersection contains a computable point of $(\boldF^\#)^n$ which is 
also a computable point of $(\boldA^\#)^n$.  

We now construct $u$.  Firstly, since $\Delta(n,k) \geq \Delta_4(n,  \Delta_2(\sum_j n_j; k_0))$, 
by Lemma \ref{lm:PertMtrxGen}, $\boldB$ contains a computable type-$(n_1, \ldots, n_m)$ matricial 
generating system $(h^s_{i,j})_{s,i,j}$ of $\boldA^\#$ so that 
$\max_{s,i,j} \norm{g^s_{i,j} - h^s_{i,j}} < \min\{2^{-k_0}, 2^{-\Delta_2(\sum_j n_j , k_0)}\}$. 
Hence, by Lemma \ref{lm:PrtlIso}, there is a computable unitary $v$ of $\boldA^\#$ so that 
$v^*g^s_{i,i}v = h^s_{i,j}$ for all $i,s$ and so that $\norm{v - \unit_\boldA} < 2^{-k_0}$. 
We then set $u = \sum_{i,s} g^s_{i,1} v h^s_{1,i}$.  

By direct computation, $u$ is unitary, and $u^* g^s_{i,j} u = h^s_{i,j}$.  
It only remains to show that $\norm{u - \unit_\boldA} < 2^{-k}$.  To begin, we observe that 
\[
\norm{g^s_{i,j} - g^s_{i,1} v h^s_{i,1}} \leq \norm{g^s_{i,1} - g^s_{i,1} v g^s_{1,i} } + \norm{g^s_{i,1} v g^s_{1,i} - g^s_{i,1} v h^s_{1,i} }.
\]
Next, we obtain an upper bound on the first term in the above sum as follows. 
\begin{eqnarray*}
\norm{g^s_{i,i} - g^s_{i,1} v g^s_{1,i}} & \leq & \norm{g^s_{1,i} - g^s_{i,1} \unit_\boldA g^s_{1,i}} + 
\norm{g^s_{i,1} \unit_\boldA g^s_{1,i} - g^s_{i,1} v g^s_{1,i} }\\
& \leq & \norm{\unit_\boldA - v} \\
& < & 2^{-k_0} .
\end{eqnarray*}
We then note that 
\[
\norm{g^s_{i,1} v g^s_{1,i} - g^s_{i,1} v h^s_{1,i}} \leq \norm{g^s_{1,i} - h^s_{1,i}} < 2^{-k_0}.
\]
Hence, since $\sum_{s,i} g^s_{i,i} = \unit_\boldA$, 
\begin{eqnarray*}
\norm{\unit_\boldA - u} & < & \sum_{s,i} 2^{-k_0 + 1} \\
& = & n 2^{-k_0 + 1} \\
& \leq & 2^{-k}.
\end{eqnarray*}
\end{proof}

We can now advance to prove the first of our main theorems.

\section{Proof of Main Theorem \ref{mthm:CompAFCert}}\label{sec:CompAFCert}

We begin with two lemmas concerning finite-dimensional algebras.  
The first was already proved by A. Fox \cite{fox2022computable}; 
we give an alternate proof based on our results on computable weak stability.

\begin{lemma}\label{lm:FinDim}
If $\boldA^\#$ is a c.e. presentation of a finite-dimensional and unital \cstar-algebra, 
then there is a finite multiset $F$ of positive integers and a computable unital $\star$-isomorphism $\psi$
from $\bigoplus_{n \in F} M_n(\C)$ to $\boldA^\#$.  Moreover, $F$ and an 
$(\bigoplus_{n \in F} M_n(\C), \boldA^\#)$ index of $\psi$ 
can be computed from the dimension of $\boldA$ and an index of $\boldA^\#$.
\end{lemma}

\begin{proof}
By following the proof of Theorem \ref{thm:EffGlimm}, it is possible to compute 
positive integers $m,n_1, \ldots, n_m$ and a computable type-$(n_1, \ldots, n_m)$ 
unital matricial generating system $(g^s_{r,t})_{s,r,t}$ of $\boldA^\#$.
Let $F$ denote the multiset whose elements are $n_1, \ldots, n_m$ 
and each element is repeated according to the number of times it appears 
in $(n_1, \ldots, n_m)$.   
Let $(E^s_{r,t})_{s,r,t}$ denote the standard matricial generating system 
for $\bigoplus_{n \in F} M_n(\C)$.  Let $\psi$ be the unique $\star$-isomorphism 
of $\bigoplus_{n \in F} M_n(\C)$ onto $\boldA$ so that $\psi(E^s_{r,t}) = g^s_{r,t}$. 
It follows that $\psi$ is a computable map from $\bigoplus_{n \in F} M_n(\C)$ to 
$\boldA^\#$.
\end{proof}

\begin{lemma}\label{lm:FindF}
Assume $\boldA^\#$ is a c.e. presentation of an AF algebra.  Suppose 
$p_0, \ldots, p_n$ are computable points of $\boldA^\#$, and let $k \in \N$. 
Then $\boldA^\#$ has a c.e. closed and finite-dimensional 
$\star$-subalgebra $\boldF$ so that $\unit_\boldA \in \boldF$ and 
$\max_j d(p_j, \boldF) < 2^{-k}$.  Furthermore, 
an $\boldA^\#$ index of $\boldF$ can be computed from $k$ and 
$\boldA^\#$ indexes of $p_0$, $\ldots$, $p_n$.  
\end{lemma}

\begin{proof}
Fix a sequence $(n_1, \ldots, n_m)$ of positive integers and let $\ell = \sum_j n_j^2$. 
Define $\mathcal{M}_{n_1, \ldots, n_m}$ to be the set of all 
type-$(n_1, \ldots, n_m)$ matricial systems of $\boldA$. 
For each $\vec{a} = (a_1, \ldots, a_\ell) \in \boldA^\ell$ and every $\vec{q} = (q_1, \ldots, q_n) \in \Q(i)^\ell$, 
let $\langle \vec{q}, \vec{a} \rangle = \sum_j q_j a_j$.  
Then let $U_{n_1, \ldots, n_m}$ denote the set of all $\vec{a} \in \boldA^\ell$ for which 
there exists $\vec{q} \in \Q(i)^\ell$ so that 
$\max_j \norm{p_j - \langle \vec{a}, \vec{q} \rangle} < 2^{-k}$.

It follows from Proposition \ref{prop:MutOrthoGen} and \cite[Theorem 1.14]{UHFPaper} that 
$\mathcal{M}_{n_1, \ldots, n_m}$ is a c.e. closed set of $(\boldA^\#)^\ell$ 
uniformly in $n_1, \ldots, n_m$.  
We claim that $U_{n_1, \ldots, n_m}$ is a c.e. open set of $(\boldA^\#)^\ell$ 
uniformly in $n_1, \ldots, n_m$.  
To see this, for each $\vec{q} \in \Q(i)^\ell$, let $f_{j,\vec{q}} : \boldA^\ell \rightarrow \C$
be defined by setting $f_{j,\vec{q}}(\vec{a}) = p_j - \langle \vec{q}, \vec{a} \rangle $. 
It follows that $f_{j, \vec{q}}$ is a computable map from $(\boldA^\#)^\ell$ to $\boldA^\#$ uniformly in 
$j, \vec{q}$.  Thus, by Proposition \ref{prop:PreimCeOpen}, 
$f_{j,\vec{q}}^{-1}[B(\zerovec; 2^{-k})]$ is a c.e. open set of $(\boldA^\#)^\ell$ uniformly in $j, \vec{q}$.  
Hence, $\bigcap_j f_{j,\vec{q}}^{-1}[B(\zerovec; 2^{-k})]$ is a c.e. open set of $(\boldA^\#)^\ell$ uniformly in $\vec{q}$.
However, $U_{n_1, \ldots, n_m} = \bigcup_{\vec{q}} \bigcap_j f_{j,\vec{q}}^{-1}[B(\zerovec; 2^{-k})]$.
It follows that $U_{n_1, \ldots, n_m}$ is c.e. open uniformly in $n_1, \ldots, n_m$.  

It now follows that the set of all $(n_1, \ldots, n_m)$ so that 
$U_{n_1, \ldots, n_m} \cap \mathcal{M}_{n_1, \ldots, n_m} \neq \emptyset$ is 
c.e.  At the same time, since $\boldA$ is an AF algebra, there must exist 
$n_1, \ldots, n_m$ so that $U_{n_1, \ldots, n_m} \cap \mathcal{M}_{n_1, \ldots, n_m} \neq \emptyset$.
Therefore, it is possible to compute positive integers $m, n_1, \ldots, n_m$ so that 
$U_{n_1, \ldots, n_m} \cap \mathcal{M}_{n_1, \ldots, n_m} \neq \emptyset$ via a search procedure. 
It then follows from \cite[Lemma 1.12]{UHFPaper} that it is possible to compute 
a type-$(n_1, \ldots, n_m)$ compound generating matricial system $(f^s_{r,t})_{s,r,t}$ 
of $\boldA^\#$ that belongs to $U_{n_1, \ldots, n_m}$.  
Taking $\boldF$ to be the $\star$-algebra generated by this system yields the desired result.
\end{proof}

\begin{proof}[Proof of Main Theorem \ref{mthm:CompAFCert}]
Let $\Delta : \N^2 \rightarrow \N$ denote the function whose existence is guaranteed by the 
Effective Glimm Lemma (Theorem \ref{thm:EffGlimm}).  Without loss of generality, we assume 
$\Delta$ is increasing in each variable.

Suppose $\boldA^\#$ is a c.e. presentation of an AF algebra.
Fix an effective enumeration $(g_n)_{n \in \N}$ of the generated points of $\boldA^\#$.  

We construct a monotonically non-decreasing sequence $(\boldA_k)_{k \in \N}$ of finite-dimensional and 
c.e. closed  $\star$-subalgebras of $\boldA^\#$ so that $d(g_k, \boldA_k) < 2^{-k}$ for all $k \in \N$ and 
so that $\unit_\boldA \in \boldA_0$.  Furthermore, we ensure $(\dim(\boldA_k))_{k \in \N}$
is computable.
The existence of a computable AF certificate for $\boldA^\#$ then follows from Lemma \ref{lm:FinDim}. 

To begin, set $\boldA_0 = \{\unit_\boldA\}$.  Let $k \in \N$, and suppose $\boldA_k$ has been 
defined.  Assume also that $n := \dim(\boldA_k)$ has been computed.  
It follows from Lemma \ref{lm:FinDim} that it is possible to compute a 
matricial generating system $(e^s_{r,t})_{s,r,t}$ for $\boldA_k$;   
let $(n_1, \ldots, n_m)$ denote the type of this system. 
Then, by Lemma \ref{lm:FindF}, it is possible to compute an $\boldA^\#$ index of a 
c.e. closed and finite-dimensional unital $\star$-subalgebra $\boldF$ of $\boldA^\#$
so that $d(a, \boldF) < 2^{-\Delta(n, k + 2)}$ whenever 
\[
a \in \{g_{k + 1}\}\ \cup\ \{e^s_{r,t}\ :\ s \in \{1, \ldots, m\}\ \wedge\ r,t \in \{1, \ldots, n_s\} \}.
\]
By the Effective Glimm Theorem, it is now possible to compute an $\boldA^\#$ index of a unitary 
$u$ so that $u^* \boldA_k u \subseteq \boldF$ and so that 
$\norm{u - \unit_\boldA} < 2^{-(k+2)}$.  
We then set $\boldA_{k + 1} = u \boldF u^*$.  We now observe that 
\begin{eqnarray*}
d(g_{k + 1}, \boldA_{k + 1}) & < & \norm{g_{k + 1} - u^*g_{k + 1} u} + 2^{-\Delta(n,k+2)} \\
& \leq & \norm{g_{k + 1} - u^* g_{k + 1} } + \norm{u^* g_{k + 1} -u^*g_{k + 1} u}  + 2^{-\Delta(n,k+2)} \\
& < & 2^{-(k + 1)} + 2^{-\Delta(n, k + 2)}.
\end{eqnarray*}
As noted above, $\Delta$ is increasing in each variable.  Thus, in particular, 
$\Delta(n, k + 2) \geq k + 2$, and so $d(g_{k + 1}, \boldA_{k + 1}) < 2^{-(k + 1)}$.  
\end{proof}

We note that our proof of Main Theorem \ref{mthm:CompAFCert} is uniform.
The following is an immediate, and somewhat surprising, consequence.

\begin{corollary}\label{cor:CompPresAF}
Every c.e. presentation of a unital AF algebra is computable.
\end{corollary}

\begin{proof}
Apply Main Theorem \ref{mthm:CompAFCert} and Corollary \ref{cor:CompIndLimAlg2}.
\end{proof}

We now proceed to prove our second main theorem.

\section{Proof of Main Theorem \ref{mthm:ClsCePresDimGrp}}\label{sec:ClsCePresDimGrp}

We begin with two lemmas concerning computable certificates of dimensionality.

\begin{lemma}\label{lm:CompCertDim}
If $\calG^\#$ is a c.e. presentation of a dimension group, then 
$\calG^\#$ has a computable certificate of dimensionality.
\end{lemma}

\begin{proof}
Suppose $\calG^\#$ is a c.e. presentation of a dimension group and let $(g_n)_{n \in \N}$ be an effective enumeration of the positive cone of $\calG$.
By Lemma \ref{lm:EffShen}, $\calG^\#$ has the Effective Shen Property.

For each $s \in \N$, we construct the following: 
\begin{itemize}
    \item A positive integer $n_s$. 

    \item An ordered group homomorphism $\phi_s : \Z^{n_s} \rightarrow \Z^{n_{s + 1}}$.

    \item A homomorphism $\theta_s$ from $\Z^{n_s}$ to $\calG^\#$.
\end{itemize}

We ensure the following requirements.
{
\renewcommand{\theenumi}{R-\arabic{enumi}}
\begin{enumerate}
    \item $(\calG, (\theta_s)_{s \in \N})$ is an inductive upper limit of 
    $(\Z^{n_s}, \phi_s)_{s \in \N}$. 

    \item $\calG^+ = \bigcup_{s \in \N} \theta_s[\Z^{n_s}_{\geq 0}]$. 

    \item For every $s \in \N$, $\ker(\theta_s) = \bigcup_{k \in \N} \ker(\phi_{s,s+k})$.
\end{enumerate}
}
It follows from Lemma \ref{lm:IndLimOrdAbl} that if these requirements 
are satisfied, then $(\calG^\#, (\theta_s)_{s \in \N})$ is 
an inductive limit of $(\Z^{n_s}, \phi_s)_{s \in \N}$.  

We divide the construction into stages.  At stage $s$, we define $n_s$ and $\theta_s$.
At stage $s + 1$, we define $\phi_s$.

We let $\ker_s(\theta_k)$ denote the set of 
all vectors that have been enumerated into $\ker(\theta_k)$ after $s$ steps of computation.
We control the enumeration of these kernels so as to satisfy the following for every $t \in \N$. 
{
\renewcommand{\theenumi}{C-\arabic{enumi}}
\begin{enumerate}
    \item If $t = 0$, then $\ker_t(\theta_k) = \emptyset$ for all $k \in \N$.

    \item There is 
at most one $k \in \N$ so that $\ker_{t+1}(\theta_k) \setminus \ker_t(\theta_k) \neq \emptyset$.

    \item For all $k \in \N$, if $\ker_{t+1}(\theta_k) \setminus \ker_t(\theta_k) \neq \emptyset$, 
    then $k \leq t$ and $\#(\ker_{t+1}(\theta_k) \setminus \ker_t(\theta_k)) = 1$. 
\end{enumerate}
}

In addition, we assume these enumerations are produced so that from $k,t$ it is possible to compute
an index of $\ker_t(\theta_k)$.  We note that $\#\ker_t(\theta_k) \leq t$ for all $k,t$ (since it takes at least $t$ steps of computation to produce an output of size at least $t$).

During the construction, we ensure the following invariants. 
{
\renewcommand{\theenumi}{I-\arabic{enumi}}
\begin{enumerate}
    \item $\theta_{s+1} \circ \phi_s = \theta_s$. \label{inv:comp}

    \item If $\alpha \in \ker_{t+1}(\theta_s)\setminus\ker_t(\theta_s)$
    then, $\alpha \in \ker(\phi_{s,t})$.\label{inv:ker}

    \item $g_s \in \theta_k[(\Z^{n_s})^+]$. \label{inv:sur}
\end{enumerate}
}
It follows that if these invariants are ensured, then the above requirements are satisfied.\\

\noindent\bf Stage 0:\rm\ Let $n_0  =  1$ and let $\theta_0(n)  =  ng_0$.\\

\noindent\bf Stage $s+1$:\rm\ Assume $n_0, \ldots, n_s$ and $\theta_0$, $\ldots$, $\theta_s$ 
have been defined.  If $s > 0$, then assume that $\phi_0, \ldots, \phi_{s-1}$ have been defined.
Assume all invariants have been maintained at previous stages.

Let $\iota_s$ denote the natural injection of $\Z^{n_s}$ into $\Z^{n_s} \bigoplus \Z$.
That is, $\iota_s(\vec{n}) = (\vec{n}, 0)$.
Define $\zeta_s : \Z^{n_s} \bigoplus \Z \rightarrow G$ by setting 
$\zeta_s(\vec{n}, m) = \theta_s(\vec{n}) + m g_{s+1}$.\\

\noindent\it Case 1:\rm\ $s = 0$ or  $\ker_s(\theta_k) \subseteq \ker_{s-1}(\theta_k)$ for all $k < s$ (that is, for $k < s$ nothing new was added to our approximation of the kernel of $\theta_k$ at stage $s$). \\

We set:
\begin{eqnarray*}
n_{s+1} & = & n_s + 1\\
\phi_s & = & \iota_s\\
\theta_{s+1} & = & \zeta_s.
\end{eqnarray*}
\ \\
\noindent\it Case 2:\rm\ $s > 0$ and there exists $k < s$ so that $\ker_s(\theta_k) \setminus \ker_{s-1}(\theta_k) \neq \emptyset$. \\

Because of our conditions on the enumeration of the kernels,
the number $k$ is unique, and we set $k_{s+1} = k$.  
We also define $\alpha_{s+1}$ to be the unique element in $\ker_s(\theta_k) \setminus \ker_{s-1}(\theta_k)$.
We seek to define $\phi_s$ so that $\phi_{s-1} \circ \cdots \circ \phi_k(\alpha_{s+1}) \in \ker(\phi_s)$. 
To this end, set $\beta_{s+1} = \phi_{s-1} \circ \cdots \circ \phi_k(\alpha_{s+1})$.
We note that $\beta_{s+1} \in \ker(\theta_s)$ (since all invariants have been maintained at prior stages).
Thus, by definition, $\beta_{s+1} \in \ker(\zeta_s)$.  
We apply the Effective Shen Property to $n_s$, $\zeta_s$, and $\beta_{s+1}$ to obtain a positive 
integer 
$p$, a homomorphism $\phi_s': \Z^{n_s} \bigoplus \Z \rightarrow \Z^p$, and an ordered group 
homomorphism $\theta_s' : \Z^p \rightarrow \calG$ so that 
$\beta_{s+1} \in \ker(\phi_s')$ and so that $\zeta_s = \theta_s' \circ \phi_s'$.  
We then set:
\begin{eqnarray*}
n_{s+1} & = & p \\
\theta_{s+1} & = & \theta_s' \\
\phi_s & = & \phi_s' \circ \iota_s.
\end{eqnarray*}
This completes the construction. 

We now demonstrate that all invariants are maintained at every stage of the construction.
(\ref{inv:ker}) and (\ref{inv:sur}) follow directly from the construction. 
We proceed to verify (\ref{inv:comp}).  We first note that by 
construction, $\zeta_s \circ \iota_s = \theta_s$ for all $s \in \N$. Let $s \in \N$.  

Suppose $s = 0$.  By construction, $\phi_0 = \iota_0$ and $\theta_1 = \zeta_0$.  
By definition, for each $n \in \N$, $\zeta_0(\iota_0(n)) = \zeta_0(n,0) = \theta_0(n)$. 
Hence, $\theta_1 \circ \phi_0 = \theta_0$. 

Now, suppose $s > 0$.
We first consider the case where 
$\ker_s(\theta_k) \subseteq \ker_{s-1}(\theta_k)$ for all $k < s$ (that is, Case 1 holds at stage $s+1$).  
By construction, $n_{s+1} = n_s + 1$, $\phi_s = \iota_s$ and $\theta_{s+1} = \zeta_s$.  
Thus, for all $\vec{n} \in \Z^{n_{s+1}}$, $\zeta_s(\iota_s(\vec{n})) = \zeta_s(\vec{n}, 0) = \theta_s(\vec{n})$; that is, $\theta_{s+1} \circ \phi_s = \theta_s$. 

Now, we consider the case in which there exists $k < s$ so that $\ker_s(\theta_k) \setminus \ker_{s-1}(\theta_k) \neq \emptyset$ (that is, Case 2 holds at stage $s+1$).
By construction, $\phi_s = \phi_s' \circ \iota_s$ and $\theta_{s+1} = \theta_s'$.  
Furthermore, $\zeta_s = \theta_s' \circ \phi_s'$.  
Hence, $\theta_s = \zeta_s \circ \iota_s = \theta_{s+1} \circ \phi_s$. 
\end{proof}

The proof given above is an adaptation of the classical proof as given in, for example, \cite{Durand.Perrin.2022}.  The main difference between our proof and the classical proof is that in the latter, one can simply consider a basis for $\ker(\theta_k)$ and apply the Shen property a finite number of times, once to each element of the basis.  In the effective setting, one cannot, in general, know when one has found a basis for a finitely generated free abelian group; for this reason, we have to ``revisit'' each kernel infinitely often.

\begin{lemma}\label{lm:CompUntlCertDim}
Every c.e. presentation of a unital dimension group has a unital certificate of 
dimensionality.
\end{lemma}

\begin{proof}
Suppose $(\calG, u)^\#$ is a c.e. presentation of a unital dimension group.
By Lemma \ref{lm:CompCertDim}, $\calG^\#$ has a computable certificate of 
dimensionality $(n_s, \nu_s, \phi_s)_{s \in \N}$.
By the proof of Lemma \ref{lm:CompCertDim}, we may assume 
$n_0 = 1$ and $\nu_0(n) = n u$.
Set $u_s = \phi_{0,s}(1)$.

We now construct an inductive sequence $((\Z^{m_s}, v_s), \gamma_s)_{s \in \N}$ of 
unital ordered abelian groups.  
To begin, for each $s$, let $\calH_s$ denote the convex ordered 
subgroup of $\Z^{n_s}$ generated by $u_s$, and assume $\calH_s = (H_s, H_s \cap \Z^{n_s}_{\geq 0})$.
By the uniformity of Lemma \ref{lm:CompCnvxGen}, $H_s$ is computable 
uniformly in $s$.  Define $B_s$ to be the intersection of the standard basis of $\Z^{n_s}$
and $H_s$, and set $m_s = $ the cardinality of $B_s$. 
By the proof of \cite[Proposition 3.8]{Goodearl.1986}, $B_s$ is a basis for $\calH_s$.  

For each $s \in \N$, there is a unique isomorphism $\tau_s$ from $\Z^{m_s}$ 
to $\calH_s$ so that for each $j \in \{1, \ldots, m_s\}$, 
$\tau_s(e^{m_s}_j) = e^{n_s}_k$ where $k$ is the $j$-th element of 
$\{k'\ :\ e^{n_s}_{k'} \in B_s\}$. 
Let $v_s = \tau_s^{-1}(u_s)$.  Thus, $v_s$ is an order unit of $\Z^{m_s}$.  
Let $\psi_s = \phi_s|_{H_s}$.  
Since $\phi_s(u_s) = u_{s+1}$, it follows that $\phi_s[H_s] \subseteq H_{s+1}$.
  We now set $\gamma_s = \tau_{s+1}^{-1}\psi_s\tau_s$.

Let $\iota_s$ denote the inclusion map of $\calH_s$ into $\Z^{n_s}$, 
and define $\mu_s$ to be $\nu_s\iota_s\tau_s$.  It follows that 
$((\calG,u)^\#, (\mu_s)_{s \in \N})$ is a computable upper inductive limit of 
$((\Z^{m_s}, v_s), \gamma_s)_{s \in \N}$;   
we claim it is in fact a computable inductive limit.
It suffices to show that $(\calG, (\mu_s)_{s \in \N})$ is 
an inductive limit of $(\Z^{m_s}, \gamma_s)_{s \in \N}$.  
To this end, suppose $x \in \calG^+$. By Lemma \ref{lm:IndLimOrdAbl}, 
there exists $s \in \N$ and $y \in \Z^{n_s}_{\geq 0}$ so that 
$\nu_s(y) = x$.  Since $u$ is an order unit of $\calG$, there exists a positive integer
$t$ so that $-tu \leq x \leq tu$.  Thus, $-t\nu_s(u_s) \leq \nu_s(y) \leq t\nu_s(u_s)$. 
Since $\ker(\nu_s) = \bigcup_{k \in \N} \ker(\phi_{s,s+k})$, it follows that 
there exists $k \in \N$ so that $-t\phi_{s,s+k}(u_s) \leq \phi_{s,s+k}(y) \leq t \phi_{s,s+k}(u_s)$.
However, $\phi_{s,s+k}(u_s) = u_{s+k}$ which belongs to $\calH_{s + k}$.  
Since $\calH_{s + k}$ is convex, $\phi_{s,s+k}(y) \in \calH_{s+k}$.  
Moreover, $\phi_{s,s+k}(y) \in \calH_{s + k}^+$, and so $x = \nu_{s+k}(\phi_{s,s+k}(y)) \in \nu_{s+k}[\calH_{s+k}^+]$.  Hence, $\calG^+ = \bigcup_{s \in \N} \mu_s[\calH_s^+]$.
Since $\tau_s$ and $\iota_s$ are injective, it now follows that $\ker(\mu_s) = \bigcup_{k \in \N} \ker(\gamma_{s,s+k})$.  Thus, by Lemma \ref{lm:IndLimOrdAbl}, $(\calG, (\mu_s)_{s \in \N})$ is
an inductive limit of $(\Z^{m_s}, \gamma_s)_{s \in \N}$.

It now follows that $(m_s, v_s, \mu_s, \gamma_s)_{s \in \N}$ is a computable
unital certificate of dimensionality of $(\calG, u)^\#$.
\end{proof}

We note that the proofs of the above lemmas are uniform.

\begin{proof}[Proof of Main Theorem \ref{mthm:ClsCePresDimGrp}] \ \\ \ \\
(\ref{mthm:ClsCePresDimGrp::NotUntl}): Suppose $\calG^\#$ is a c.e. presentation of a dimension group. 
By Lemma \ref{lm:CompCertDim}, $\calG^\#$ has a computable certificate of 
dimensionality 
$(n_s, \nu_s, \phi_s)_{s \in \N}$.  

Define $F_0$ to be the multiset whose only element is $1$ repeated $n_0$ times. 
Set $u_0 = v_{F_0}$ (which is $(1, \ldots, 1) \in \Z^{n_0}$), and let $\boldA_0 = \bigoplus_{n \in F_0} M_n(\C)$.   

Assume $F_s$ has been defined, $u_s = v_{F_s}$, and $\boldA_s = \bigoplus_{n \in F_s} M_n(\C)$. 
We compute an order unit $u_{s+1} \in \Z^{n_{s+1}}$ so that 
$u_{s+1} \geq \phi_0(u_s)$. We let $F_{s+1}$ be the multiset that consists of 
the integers that appear in $u_s$ and the number of times an integer is repeated in 
$F_{s+1}$ is the number of times it is repeated in $u_s$.  Set $u_{s+1} = v_{F_{s+1}}$, and 
set $\boldA_{s+1} = \bigoplus_{n \in F_{s+1}} M_n(\C)$.  
Since $\phi_s(u_s) \leq u_{s+1}$, it follows from the discussion in 
Section \ref{subsubsec:BackFinDimAlg} that there exists a $\star$-homomorphism 
$\psi_s : \boldA_s \rightarrow \boldA_{s+1}$ so that $\eta_{F_s, F_{s+1}}(\psi_s)) = \phi_s$.
Furthermore, $\psi_s$ can be computed from the information in $\phi_s$. 

By Corollary \ref{cor:CompIndLimAlg2}, 
$(\boldA_s, \psi_s)_{s \in \N}$ sequence has a computable inductive limit $(\boldA^\#, (\mu_s)_{s \in \N})$ in the category of c.e. presentations of \cstar-algebras. 
As noted in Section \ref{subsubsec:BackAF}, $\boldA$ is an AF algebra.

We now show that $(K_0(\boldA^\#), K_0(\boldA)^+)$ is computably isomorphic to $\calG^\#$.   
 It follows from \cite[Theorem 6.3.2]{Rordam.Larsen.Laustsen.2000} and 
Lemma \ref{lm:IndLimOrdAbl} that $((K_0(\boldA^\#), K_0(\boldA)^+), (K_0(\mu_s))_{s \in \N})$
is an inductive limit of $((K_0(\boldA_s), K_0(\boldA_s)^+), K_0(\psi_s))_{s \in \N}$ in the category of c.e. presentations of ordered abelian groups.
It then follows that 
$((K_0(\boldA^\#), K_0(\boldA)^+), (K_0(\mu_s)\circ \zeta_{F_s}^{-1})_{s \in \N}$ is 
an inductive limit of $(\Z^{n_s}, \phi_s)_{s \in \N}$ in the category of c.e. presentations of ordered abelian groups.
Thus, the reduction of $(\calG, (\nu_s)_{s \in \N})$ to $((K_0(\boldA^\#), K_0(\boldA)^+), (K_0(\mu_s)\circ \zeta_{F_s}^{-1})_{s \in \N}$
is an isomorphism.

(\ref{mthm:ClsCePresDimGrp::Untl}):\ Suppose $u$ is an order unit of $\calG$.
By Lemma \ref{lm:CompUntlCertDim}, 
$(\calG, u)^\#$ has a computable unital certificate of dimensionality 
$(n_s, u_s, \nu_s, \phi_s)_{s \in \N}$.
 Thus, by Definition \ref{def:UntlCertDim}
$u_s$ is an order unit for $\Z^{n_s}$ and 
$\nu_s(u_s) = u$.  
We modify the above construction so that $\eta_{F_s}(K_0(\unit_{\boldA_s})) = u_s$.
\end{proof}

\begin{corollary}
Suppose $\calG^\# = (G^\#, P)$ is a c.e. presentation of a dimension group. 
Then, $G^\#$ is computable. 
\end{corollary}

\begin{proof}
Let $(n_s, \nu_s, \phi_s)_{s \in \N}$ be a computable certificate of dimensionality 
for $\calG^\#$.  Thus, 
$(G^\#, (\nu_s)_{s \in \N})$ is a computable inductive limit of 
$(\Z^{n_s}, \phi_s)_{s \in \N}$ in the category of c.e. presentations of 
abelian groups.  However, by Proposition \ref{prop:CompIndLimPresAbl}, 
$(\Z^{n_s}, \phi_s)_{s \in \N}$ has a computable inductive limit 
$(H^\#, (\mu_s)_{s \in \N})$ in the category of computable presentations of 
abelian groups.  Since every computable presentation is also c.e., 
$(H^\#, (\mu_s)_{s \in \N})$ is also a computable inductive upper limit of 
$(\Z^{n_s}, \phi_s)_{s \in \N}$ in the category of c.e. presentations of abelian groups.
Let $\lambda$ be the reduction of $(G^\#, (\nu_s)_{s \in \N})$ to 
$(H^\#, (\mu_s)_{s \in \N}$.  
Thus, by definition of computable inductive limit, 
$\lambda$ is a computable map from $G^\#$ to $H^\#$.  

On the other hand, by Lemma \ref{lm:IndLimPresAbl2} $(H, (\mu_s)_{s \in \N})$ is an inductive limit 
of $(\Z^{n_s}, \phi_s)_{s \in \N}$ and $(G, (\nu_s)_{s \in \N})$
is as well.  Thus, $\lambda$ is the reduction of $(G, (\nu_s)_{s \in \N})$ to 
$(H, (\mu_s)_{s \in \N})$ (since reductions are unique).  
However, it then follows that $\lambda$ is an isomorphism.

Thus, $\lambda$ is a computable isomorphism from $G^\#$ to $H^\#$. 
Since $H^\#$ is computable, it follows that $G^\#$ is computable.
\end{proof}

\begin{corollary}\label{cor:CompK0AF}
If $\boldA^\#$ is a c.e. presentation of a unital AF algebra, 
then $K_0(\boldA^\#)$ is computable.
\end{corollary}

We now turn to the proof of our third main theorem.

\section{Proof of Main Theorem \ref{mthm:ClsCompPrsntAF}}\label{sec:ProofMthm3}

(\ref{mthm:ClsCompPrsntAF::CePres}) $\Rightarrow$ (\ref{mthm:ClsCompPrsntAF::CompIndSeqFinDimAlg}): This is a straightforward application of Main Theorem \ref{mthm:CompAFCert} 
and Corollary \ref{cor:IndLimAFCert}.\\

(\ref{mthm:ClsCompPrsntAF::CompIndSeqFinDimAlg}) $\Rightarrow$ (\ref{mthm:ClsCompPrsntAF::CompPres}): Suppose $(\bigoplus_{n \in F_s} M_n(\C), \psi_s)_{s \in \N}$ is a computable inductive 
sequence whose inductive limit is $\star$-isomorphic to $\boldA$. Furthermore,
assume each $\psi_s$ is unital. 
By Corollary \ref{cor:CompIndLimAlg2}, $(\bigoplus_{n \in F_s} M_n(\C), \psi_s)_{s \in \N}$ has a 
computable inductive limit $(\boldB^\#, (\nu_s)_{s \in \N})$ in the category of 
computable presentations of \cstar-algebras.  Thus, $\boldB$ is $\star$-isomorphic to 
$\boldA$.  As remarked in Section \ref{subsec:BackCompAlg}, the class of computably presentable 
\cstar-algebras is closed under isomorphism.  Hence, $\boldA$ is computably presentable.\\

(\ref{mthm:ClsCompPrsntAF::CompPres}) $\Rightarrow$ (\ref{mthm:ClsCompPrsntAF::CompPresUntlOrdAbl}):
This follows from Corollary \ref{cor:CompK0AF}.\\

(\ref{mthm:ClsCompPrsntAF::CompPresUntlOrdAbl}) $\Rightarrow$ (\ref{mthm:ClsCompPrsntAF::CompBrat}):\
Let $(\calG,u) = \uK{\boldA}$, and suppose $(\calG,u)^\#$ is a c.e. presentation.
By Main Theorem \ref{mthm:ClsCePresDimGrp}, $(\calG,u)^\#$ has a computable unital certificate
of dimensionality $(n_s, u_s, \nu_s, \gamma_s)_{s \in \N}$. 

We first show that $\boldA$ has a c.e. presentation. 

Let $F_s$ denote the multiset that consists of the integers that appear in $u_s$ where
the number of times an integer is repeated in $F_s$ is the number of times it is 
repeated in $u_s$.  Set $\boldB_s = \bigoplus_{n \in F_s} M_n(\C)$.
For each $s$, it is fairly straightforward to compute a unital $\star$-homomorphism 
$\psi_s : \boldB_s \rightarrow \boldB_{s+1}$ 
so that $\eta_{F_s, F_{s+1}}( K_0(\psi_s)) = \phi_s$. 
By Corollary \ref{cor:CompIndLimAlg2}, $(B_s, \psi_s)_{s \in \N}$ has a computable
inductive limit $(\B^\#, (\nu_s)_{s \in \N})$ in the category of 
\cstar-algebras with c.e. presentations.  

It follows that $(\uK{\boldA}, (\nu_s \circ \zeta_{F_s}^{-1})_{s \in \N})$ is an inductive limit of $(\uK{\boldB_s}, K_0(\psi_s))_{s \in \N}$ as is $(\uK{\boldB}, (K_0(\mu_s))_{s \in \N})$.
Thus, $\uK{\boldA}$ and 
$\uK{\boldB}$ are isomorphic.  Hence, by Elliott's Theorem,
$\boldB$ is $\star$-isomorphic to $\boldA$. Therefore, $\boldA$ has a c.e. presentation 
$\boldA^\#$.  

By Main Theorem \ref{mthm:CompAFCert}, $\boldA^\#$ has a computable AF certificate
$(G_s, \gamma_s)_{s \in \N}$.  It is straightforward to construct a computably presented labeled 
Bratteli diagram of this certificate.  \\

(\ref{mthm:ClsCompPrsntAF::CompBrat}) $\Rightarrow$ (\ref{mthm:ClsCompPrsntAF::CePres}):  Suppose $\BratD^\#$ is a computable presentation of labeled Bratteli diagram of 
an AF certificate $(F_s, \phi_s)_{s \in \N}$ of $\boldA$.  
By means of the labeling, it is possible to compute $F_s$ from $s$. 
Set $\boldA_s = \bigoplus_{n \in F_s} M_n(\C)$.

By means of the information provided by the edge function of $\BratD$, 
it is possible to compute from $s \in \N$ a unital $\star$-monomorphism 
$\tau_s : \boldA_s \rightarrow \boldA_{s+1}$ so that 
$K_0(\phi_{s+1}^{-1} \circ \phi_s) = K_0(\tau_s)$.  
By Corollary \ref{cor:CompIndLimAlg2}, $(\boldA_s, \tau_s)_{s \in \N}$ 
has a computable inductive limit $(\boldB^\#, (\nu_s)_{s \in \N})$ in the 
category of computably presentable \cstar-algebras. However, by construction, 
$\uK{\boldA} = \uK{\boldB}$.  
By Elliott's Theorem, $\boldA$ and $\boldB$ are $\star$-isomorphic. 
Thus, $\boldA$ is computably presentable.

\section{Proof of Main Theorem \ref{mthm:CompIso}}

Our proof relies on two very technical lemmas: one for AF algebras and one for groups.

\begin{lemma}\label{lm:CertIntr}
Suppose $(F_s, \phi_s)_{s \in \N}$ is an AF certificate for $\boldA$, and let 
$\boldG = \bigoplus_{n \in G} M_n(\C)$ where $G$ is a finite multiset of positive integers. 
Furthermore, assume $\omega : \uK{\bigoplus_{s \in F_0} M_n(\C)} \rightarrow \uK{\boldG}$
and $\chi : \uK{\boldG} \rightarrow \uK{\boldA}$ are homomorphisms 
so that $\chi \circ \omega = K_0(\phi_0)$.  
Then there exists $s_0 \in \N$ and a homomorphism 
$h : \uK{\boldG} \rightarrow \uK{\bigoplus_{n \in F_{s_0}} M_n(\C)}$ so that 
$h \circ \omega = K_0(\phi_{s_0}^{-1} \circ \phi_0)$ and $K_0(\phi_{s_0}) \circ h = \chi$.
\end{lemma}

\begin{proof}
Set $\boldA_s = \bigoplus_{n \in F_s} M_n(\C)$, and let $\psi_s = \phi_{s + 1}^{-1} \circ \phi_s$.
By Corollary \ref{cor:IndLimAFCert}, $(\boldA, (\phi_s)_{s \in \N})$ is 
an inductive limit of $(\boldA_s, \phi_s)_{s \in \N}$.  We now apply \cite[Lemma 7.3.3]{Rordam.Larsen.Laustsen.2000}.
\end{proof}

\begin{lemma}\label{lm:CompIntr}
Suppose $(n_s, \nu_s, \phi_s)_{s \in \N}$ is a computable certificate of dimensionality 
of $\calG^\#$, and assume $(m_s, \mu_s, \psi_s)_{s \in \N}$ is a computable certificate of dimensionality of $\calG^\#$.  Suppose $s,r,m_s, m_r \in \N$ are such that 
$s > r$, $m_s > n_r$, and let $\gamma : \Z^{n_r} \rightarrow \Z^{m_s}$
be an ordered group homomorphism so that $\nu_r = \mu_s \gamma$. 
Then there exists $t > s$ so that $n_t > m_s$ and so that there exists 
an ordered group homomorphism $\delta : \Z^{m_s} \rightarrow \Z_{n_t}$ so that 
$\delta\gamma = \phi_{r,t}$ and $\mu_s = \nu_t\delta$.
\end{lemma}

\begin{proof}
Compute $n_{t'} > m_s$ and a positive group 
homomorphism $\delta' : \Z^{m_s} \rightarrow \Z^{n_{t'}}$ so that 
$\nu_{t'} \delta' =  \mu_s$.  

Fix a $g$ in the standard basis for $\Z^{n_r}$.  Then, 
\[
\nu_r(g) = \nu_r(g) = \mu_s \gamma(g) = \nu_{t'} \delta'\gamma(g).
\]
Thus, $\nu_{t'}\phi_{r,t'}(g) = \nu_{t'}\delta'\gamma(g)$.  
Thus, $\phi_{r,t'}(g) - \delta'\gamma(g) \in \ker(\nu_{t'})$. 
Therefore, there exists $k_g \in \N$ so that 
$\phi_{r,t'}(g) - \delta'\gamma(g) \in \ker(\phi_{t', t' + k_g})$. 
Therefore, $\phi_{r,t' + k_g}(g) = \phi_{t', t' + k_g}\delta'\gamma(g)$. 
Set $t = t' + \max_g k_g$ and $\delta = \phi_{t',t}\delta'$.
Thus, $\phi_{r,t} = \delta\gamma$, and $\nu_t \delta = \nu_t \phi_{t',t} \delta'= \nu_{t'}\delta' = \mu_s$. 
\end{proof}

\begin{proof}[Proof of Main Theorem \ref{mthm:CompIso}]
We first show that (\ref{mthm:CompIso::Pres}) and (\ref{mthm:CompIso::Grp}) 
are equivalent.

To begin, suppose $\boldA^\#$ and $\boldB^\#$ are c.e. presentations of AF algebras, and let 
$\psi$ be a computable unital $\star$-isomorphism from $\boldA^\#$ to $\boldB^\#$.
Thus, as discussed in Section \ref{subsubsec:BackCompKThy} $K_0(\psi)$ is a computable isomorphism from 
$\uK{\boldA^\#}$ to $\uK{\boldB^\#}$.  

We note that if the AF certificate in Lemma \ref{lm:CertIntr} is a computable 
AF certificate of a c.e. presentation $\boldA^\#$, then $s_0$, $h$ can be computed from 
$G$, $\omega$, $\chi$ as $h$ is determined by a finite amount of data (see the discussion
in Section \ref{subsubsec:BackFinDimAlg}).  

Suppose $\alpha$ is a computable isomorphism from 
$\uK{\boldA^\#}$ to 
$\uK{\boldB^\#}$.  By Main Theorem \ref{mthm:CompAFCert}, $\boldA^\#$ has a computable 
AF certificate $(F_s, \phi_s)_{s \in \N}$, and $\boldB^\#$ has a computable 
AF certificate $(G_s, \psi_s)_{s \in \N}$.  Set $\boldA_s = \bigoplus_{n \in F_s} M_n(\C)$, 
and set $\boldB_s = \bigoplus_{n \in G_s} M_n(\C)$.  Without loss of generality, we 
assume $\boldA_0 = \boldB_0 = \C$.  
Set $f_s = \phi_{s+1}^{-1} \circ \phi_s$, and set $g_s = \psi_{s+1}^{-1}\circ \psi_s$. 
By Corollary \ref{cor:IndLimAFCert}, $(\boldA^\#, (\phi_s)_{s \in \N})$ is an inductive limit
of $(\boldA_s, f_s)_{s \in \N}$ and $(\boldB^\#, (\psi_s)_{s \in \N})$ is an 
inductive limit of $(\boldB_s, g_s)_{s \in \N}$.

By iterating Lemma \ref{lm:CertIntr}, we compute 
increasing sequences $(n_s)_{s \in \N}$ and $(m_s)_{s \in \N}$ of natural numbers   
along with sequences $(\alpha_s)_{s \in \N}$ and $(\beta_s)_{s \in \N}$ 
that satisfy the following.
\begin{enumerate}
    \item $n_0 = 0$. 

    \item $\alpha_s$ is a computable homomorphism from 
    $\uK{\boldA_{n_s}^\#}$ to $\uK{\boldB_{m_s}^\#}$ uniformly in $s$.

    \item $\beta_s$ is a computable homomorphism from 
    $\uK{\boldB_{m_s}^\#}$ to $\uK{\boldA_{n_{s + 1}}^\#}$.

    \item $\beta_s \circ \alpha_s = f_{n_s, n_{s+1}}$ and 
    $\alpha_{s + 1} \circ \beta_s = g_{m_s, m_{s + 1}}$.
\end{enumerate}

It is now possible to compute from $s$, a unital $\star$-homomorphism 
$\sigma_s' : \boldA_{n_s} \rightarrow \boldB_{m_s}$ 
so that $K_0(\sigma_s') = \alpha_s$ (see the material in Section \ref{subsubsec:BackFinDimAlg}), 
and a unital $\star$-homomorphism $\tau_s' : \boldB_{m_s} \rightarrow \boldA_{n_{s+1}}$ 
so that $K_0(\tau_s') = \beta_s$. 

It follows from the discussion of multiplicities in Section \ref{subsubsec:BackMtrx}, 
that there exist, and we can compute, 
sequences $(u_s)_{s \in \N}$ and $(v_s)_{s \in \N}$ that satisfy the following. 
\begin{enumerate}
    \item $v_s$ is a unitary of $\boldA_{n_{s+1}}$ and $u_s$ is a unitary of $\boldB_{m_s}$. 

    \item $f_{n_s,n_{s+1}} = \Ad_{v_s} \circ \tau'_s \circ \Ad_{u_s} \circ \sigma'_s$
    and $g_{m_s, m_{s+1}} = \Ad_{u_{s+1}} \circ \sigma_{s+1}' \circ \Ad_{v_s} \circ \tau_s'$.
\end{enumerate}
Set $\sigma_s = \Ad_{u_s} \circ \sigma'_s$ and set 
$\tau_s = \Ad_{v_s} \circ \tau_s'$.  
Thus, for each $s \in \N$, $K_0(\sigma_s) = \alpha_s$ and $K_0(\tau_s) = \beta_s$. 
Furthermore, $f_{n_s, n_{s+1}} = \tau_s \circ \sigma_s$ and 
$g_{m_s, m_{s+1}} = \sigma_{s+1} \circ \tau_s$.  

It now follows that $(\boldA^\#, (\phi_{n_{s+1}}\circ\tau_s)_{s \in \N})$ is an inductive 
upper limit of $(\boldB_{m_s}, g_{m_s, m_{s+1}})_{s \in \N}$ in the category of c.e. presentations
of unital \cstar-algebras.  It follows from Lemma \ref{lm:IndLimAlg} that 
$(\boldB^\#, (\psi_{m_s})_{s \in \N})$ is an inductive limit of 
$(\boldB_{m_s}, g_{m_s, m_{s+1}})_{s \in \N}$ in this category.  So, let $\lambda$ be the reduction of 
$(\boldB^\#, (\psi_{m_s})_{s \in \N})$ to $(\boldA^\#, (\phi_{n_{s+1}} \circ \tau_s)_{s \in \N})$. 
Hence, $\lambda \circ \psi_{m_s} = \phi_{n_{s+1}} \circ \tau_s$. 

At the same time, $(\boldB^\#, (\psi_{m_s} \circ \sigma_s)_{s \in \N})$ is an upper inductive 
limit of $(\boldA_s, f_{n_s, n_{s+1}})_{s \in \N}$, and $(\boldA^\#, (\phi_{n_s})_{s\in \N})$
is an inductive limit of this inductive sequence. 
So, let $\lambda'$ be the reduction from $(\boldA^\#, (\phi_{n_s})_{s \in \N})$ to 
$(\boldB^\#, (\psi_{m_s} \circ \sigma_s)_{s\in \N})$.
Hence, $\lambda' \circ \phi_{n_s} = \psi_{m_s} \circ \sigma_s$. 
Therefore, 
\[
\lambda\lambda'\phi_{n_s} = \lambda\psi_{m_s}\sigma_s = \phi_{n_{s+1}}\tau_s \sigma_s = \phi_{n_{s+1}} f_{n_s, n_{s + 1}} = \phi_{n_s}.
\]
Thus, $\lambda\lambda'$ is the identity on $\bigcup_{s \in \N} \ran(\phi_{n_s})$; hence it is the 
identity on $\boldA$ as well. 
Similarly, $\lambda'\lambda$ is the identity on $\boldB$, and $\lambda$ is an isomorphism. 
By the continuity of $K_0$ (see, e.g. \cite[Section 6.3]{Rordam.Larsen.Laustsen.2000}), $\lambda = K_0(\alpha)$.

Now we show that 
(\ref{mthm:CompIso::Pres}) and (\ref{mthm:CompIso::Brat}) are equivalent.
For one direction, it suffices to show that if $\calC_0 = (F_s,\phi_s)_{s \in \N}$
and $\calC_1 = (G_s, \psi_s)_{s in \N}$ are computable AF certificates for $\boldA^\#$, 
then every Bratteli diagram of $\calC_0$ is computably equivalent to every 
Bratteli diagram of $\calC_1$. 
Set $\boldA_s = \bigoplus_{n \in F_s} M_n(\C)$, and set $\boldB_s = \bigoplus_{n \in F_s} M_n(\C)$. 
In addition, set $\sigma_s = \phi_{s+1}^{-1} \circ \phi_s$, and set 
$\tau_s = \psi_{s+1}^{-1} \circ \psi_s$. 
It suffices to show that the standard Bratteli diagram of 
$(\boldA_s, \sigma_s)_{s \in \N}$ is computably equivalent to the 
standard Bratteli diagram of $(\boldB_s, \tau_s)$; denote these 
Bratteli diagrams by $\BratD_0$ and $\BratD_1$ respectively.

Let $(\Z^{m_s}, u_s)$ denote the co-domain of $\zeta_{F_s}$, and let 
$(\Z^{n_s}, v_s)$ denote the co-domain of $\zeta_{G_s}$. 
Set $\gamma_s = \eta_{F_s, F_{s+1}}(K_0(\sigma_s))$, and set 
$\tau_s = \eta_{G_s, G_{s+1}}(\tau_s))$. 
Let $\mu_s = K_0(\phi_s) \circ \zeta_{F_s}^{-1}$, and let 
$\nu_s =  K_0(\psi_s) \circ \zeta_{G_s}^{-1}$. 
Hence, $(m_s, u_s, \mu_s, \gamma_s)_{s \in \N}$ is a computable 
unital certificate of dimensionality of $\uK{\boldA^\#}$ as is 
$(n_s, v_s, \nu_s, \kappa_s)_{S \in \ N}$.  

We now construct computable increasing sequences $(k_s)_{s \in \N}$ and $(\ell_s)_{s \in \N}$
of natural numbers and an inductive sequence $((\Z^{x_s}, w_s), \rho_s)_{s \in \N}$
that satisfy the following. 
\begin{enumerate}
    \item $x_{2s} = m_{k_s}$ and $x_{2s + 1} = n_{\ell_s}$. 

    \item $w_{2s} = u_{k_s}$ and $w_{2s + 1} = v_{\ell_s}$.

    \item $\sigma_{m_{k_s}, m_{k_{s + 1}}} = \rho_{2s, 2(s + 1)}$. 

    \item $\tau_{n_{\ell_s}, n_{\ell_{s+1}}} = \rho_{2s+1, 2s + 3}$.
\end{enumerate}
Without loss of generality, we assume $m_0 = n_0 = u_0 = v_0 = 1$.
Thus, $u_s = \sigma_{0,s}(1)$ and $v_s = \tau_{0,s}(1)$.
Set $k_0 = 0$. Let $\ell_0 = 1$ and set $x_0 = 1$.
Let $x_1 = n_1$.  
Let $\delta_0(n) = n v_1$.
We now define the remaining terms of the desired sequences by iterating Lemma \ref{lm:CompIntr}.
Since $\sigma_s$ and $\psi_s$ are unit-preserving, the constructed maps
$\rho_0, \rho_1, \ldots$ are as well.

It follows fairly straightforwardly that $\BratD_0$ is a Bratteli diagram of 
$((\Z^{m_s}, u_s), \gamma_s)_{s \in \N}$ and that $\BratD_1$ is a Bratteli diagram 
of $((\Z^{n_s}, v_s), \kappa_s)_{s \in \N}$.
The sequences we have constructed witness that the standard Bratteli diagrams of 
$((\Z^{m_s}, u_s), \gamma_s)_{s \in \N}$, $((\Z^{x_s}, w_s), \rho_s)_{s \in \N}$, 
and $((\Z^{n_s}, v_s), \kappa_s)_{s \in \N}$ are computably equivalent. 
Hence, $\BratD_0$ and $\BratD_1$ are computably equivalent.

Conversely, suppose $\BratD_\boldA$ is a labeled Bratteli diagram for 
a computable AF certificate $(F_s, \phi_s)_{s \in \N}$ 
of $\boldA^\#$, and assume $\BratD_\boldB$ is a labeled Bratteli diagram for 
a computable AF certificate $(G_s, \psi_s)_{s \in \N}$ of $\boldB^\#$.
Furthermore, assume $\BratD_\boldA$ and $\BratD_\boldB$ are computably
equivalent.  

By Definition \ref{def:CompEqvBrat}, there is a finite sequence 
$\BratD_0 = \BratD_\boldA$, $\ldots$, $\BratD_m = \BratD_\boldB$ of labeled Bratteli 
diagrams so that for each $j < m$, one of $\BratD_j$ and $\BratD_{j+1}$ 
is a computable reduction of the other or they are isomorphic. 
Starting with $\BratD_0$, each of these reductions yields another computable AF certificate 
for $\boldA^\#$, and isomorphisms yield the same AF certificate; let $\calC_0, \ldots, \calC_m$ be the resulting sequence of computable 
AF certificates for $\boldA^\#$ with $\calC_0 = \calC_A$.
Since $\BratD_m = \BratD_\boldB$, it follows that $\calC_m$ has the form 
$(G_s, \gamma_s)_{s \in \N}$.  Set $C_s = \bigoplus_{n \in G_s} M_n(\C)$.
Since $\BratD_m = \BratD_\boldB$, it then follows that 
$K_0(\gamma_{s+1}^{-1} \circ \gamma_s) = K_0(\psi_{s+1}^{-1} \circ \psi_s)$. 
Hence, it follows that $\uK{\boldA^\#}$ is 
computably isomorphic to $\uK{\boldB^\#}$. 
Thus, by what has just been shown, $\boldA^\#$ is computably $\star$-isomorphic 
to $\boldB^\#$.
\end{proof}

We note that all parts of the above proof are uniform which we can exploit 
as follows.  
Recall that a structure is \emph{computably categorical} if all of its computable presentations 
are computably isomorphic. 
The following stands in contrast to our result in \cite{UHFPaper} that all 
UHF algebras are computably categorical.

\begin{corollary}\label{cor:NotCompCatAF}

\

\begin{enumerate}
    \item The AF algebra $C(\omega +1 )$ is not computably categorical.

    \item There is a unital dimension group that has two c.e. presentations that are 
    not computably isomorphic.

    \item There exist two computably presentable Bratteli diagrams that are
    equivalent but not computably equivalent.
\end{enumerate}
\end{corollary}

\begin{proof}[Proof sketch]
Only the first item requires proof.
It is well-known that the linear order $(\N, <)$ is not computably categorical.
By elaborating on the proof of this result, we obtain that 
there are computably compact presentations of $\omega + 1$ that are not 
computably homeomorphic.  It then follows from the results in \cite{mcnicholl2024evaluative}
that $C(\omega + 1)$ is not computably categorical.  As $\omega + 1$ is totally disconnected, 
$C(\omega + 1)$ is an AF algebra.
\end{proof}

\section{Proof of Main Theorem \ref{mthm:CompEqvCat}}\label{sec:CompEqvCat}

\subsection{A preliminary lemma}

We first prove the following which is made possible by Main Theorem \ref{mthm:CompAFCert}.

\begin{lemma}\label{lm:CompHm}
Suppose $\boldA^\#$ and $\boldB^\#$ are c.e. presentations of 
unital AF algebras, and let $\rho$ be a 
computable homomorphism from $\uK{\boldA^\#}$ to $\uK{\boldB^\#}$. 
Then there is a computable unital $\star$-homomorphism 
$\tau$ from $\boldA^\#$ to $\boldB^\#$ 
so that $K_0(\tau) = \rho$.
\end{lemma}

\begin{proof}
Let $(F_s, \phi_s)_{s \in \N}$ be a computable AF certificate for $\boldA^\#$, 
and let $(G_s, \sigma_s)_{s \in \N}$ be a computable AF certificate for $\boldB^\#$.
Set $\boldA_s = \bigoplus_{n \in F_s} M_n(\C)$, and set 
$\boldB_s = \bigoplus_{n \in G_s} M_n(\C)$. 
Let $\psi_s = \phi_{s+1}^{-1} \circ \phi_s$.

Fix $s \in \N$.  By Lemma \ref{lm:IndLimOrdAbl}, there exists $s_0 \in \N$ so that 
$\ran(\rho \circ K_0(\phi_s)) \subseteq \ran(K_0(\sigma_{s_0}))$.
It is possible to compute $s_0$ from $s$.
Set $\boldC = \ran(\sigma_{s_0})$, and let $\kappa = \rho \circ K_0(\phi_s)$. 
Thus, $\kappa$ is a homomorphism from $K_0(\boldA_s)$ to $K_0(\boldC)$.
Since $\sigma_{s_0}$ is injective, by Lemma \ref{lm:FinDim}, it is possible
to compute a matricial generating system for $\boldC$.
It then follows from the material in Section \ref{subsubsec:BackMtrx}
that we can now compute a unital $\star$-homomorphism $\tau_s$ so that 
$K_0(\tau_s) = \kappa$. 

It now follows that $(\boldB^\#, (\tau_s)_{s \in \N})$ is an inductive upper limit 
of $(\boldA_s, \psi_s)_{s \in \N}$.  Hence, we can compute the reduction 
$\tau$ of $(\boldA^\#, (\phi_s)_{s \in \N})$ to $(\boldB^\#, (\tau_s)_{s \in \N})$. 
Thus, $\tau \circ \phi_s = \tau_s$.  
Since $\rho \circ K_0(\phi_s) = K_0(\tau_s)$, it now follows that 
$K_0(\tau) \circ K_0(\phi_s) = \rho \circ K_0(\phi_s)$, and so 
$K_0(\tau) = \rho$.
\end{proof}

\subsection{Computable functors and equivalences}

In order to prove our final main theorem, we lay out what we mean by 
a computable equivalence of categories.  
Suppose $\Cat_0$ and $\Cat_1$ are categories for which we have fixed computable indexings.

\begin{definition}\label{def:CompFnctr}
Let $F$ be a multi-valued functor from $\Cat_0$ to $\Cat_1$.  We say that $F$ is 
\emph{computable} if there exist computable functions $\alpha : \N \rightarrow \N$ and 
$\alpha' : \N^3 \rightarrow \N$ that satisfy the following:
\begin{enumerate}
    \item If $e \in \N$ indexes an object $A$ of $\Cat_0$, then 
    $\alpha(e)$ indexes a value of $F(A)$. 

    \item If $(e_0, e_1, e)$ indexes a morphism $f : A \rightarrow B$ of $\Cat_0$, 
    then $(\alpha(e_0), \alpha(e_1), \alpha'(e_0,e_1,e))$ indexes 
    a value of $F(f)$.
\end{enumerate}
We say that $(\alpha, \alpha')$ \emph{witnesses} $F$ is computable.
\end{definition}

We note that every computable multi-valued functor is effectively essentially 
single-valued in the sense that from indexes $e,e'$ of an object $A$ it is possible 
to compute an index of a computable isomorphism between the objects indexed by 
$\alpha(e)$ and $\alpha(e')$.  Hence, we occasionally treat them as single-valued.

\begin{definition}\label{def:CompEqvCat}
A \emph{computable equivalence of $\Cat_0$ and $\Cat_1$} is a 
quadruple $(F, G, \calE_0, \calE_1)$ for which there exist  
$\alpha$, $\alpha'$, $\beta$, $\beta'$, $\gamma_0$, $\gamma_1$ that satisfy the following:
\begin{enumerate}
    \item $(\alpha, \alpha')$ witnesses that $F$ is a computable functor from $\Cat_0$ to $\Cat_1$.

    \item $(\beta, \beta')$ witnesses that $G$ is a a computable functor from $\Cat_1$ to $\Cat_0$.

    \item $\calE_0$ is a natural multi-valued isomorphism of $GF$ and $\Id_{\Cat_0}$, 
    and for each $e \in \N$, if $e$ indexes an object $A$ of $\Cat_0$, then 
    $(e, \beta(\alpha(e)), \gamma_0(e))$ indexes a value of $\calE_0(A)$.

    \item $\calE_1$ is a natural multi-valued isomorphism of $FG$ and $\Id_{\Cat_1}$, 
    and for each $e \in \N$, if $e$ indexes an object $B$ of $\Cat_1$, then 
    $(e, \alpha(\beta(e)), \gamma_1(e))$ indexes a value of $\calE_1(B)$.
\end{enumerate}
If these conditions are satisfied, then we say that $(\alpha, \alpha', \beta, \beta', \gamma_0, \gamma_1)$ \emph{witnesses} that $(F, G, \calE_0, \calE_1)$ is a computable 
equivalence of categories.
\end{definition}

Again, it follows that these computable natural isomorphisms are effectively essentially
single-valued.  

\subsection{The categories}

We now set forth formal definitions of the categories considered in Main Theorem \ref{mthm:CompEqvCat}
and their indexings.

\subsubsection{The category of c.e. presentations of unital AF algebras}

The objects in this category are c.e. presentations of unital AF algebras; we have alredy defined the 
indexing of these objects.

Suppose $\boldA_0^\#$ and $\boldA_1^\#$ are c.e. presentations of unital AF algebras.
A \emph{morphism} from $\boldA_0^\#$ to $\boldA_1^\#$ in this category is an approximately unitarily 
equivalence class of a computable unital $\star$-homomorphism from $\boldA_0^\#$ to 
$\boldA_1^\#$. 
We say that $(e_0, e_1, e)$ indexes a morphism $\Psi$ from $\boldA_0^\#$ to 
$\boldA_1^\#$ if $e_j$ indexes $\boldA_j^\#$ and $e$ is an 
$(\boldA_0^\#, \boldA_1^\#)$ index of an element of $\Psi$.

\subsubsection{The category of c.e. presentations of unital dimension groups}

The objects in this category are c.e. presentations of unital dimension groups. 
The morphisms are the computable homomorphisms.  We have already described 
the indexing of the objects in this category.

\subsection{Proof of computable equivalence}

Throughout the rest of this section, $\Cat_0$ denotes the category of c.e. presentations of unital
dimension groups, and $\Cat_1$ denotes the category of c.e. presentations of unital AF algebras.

We set $G = \uKnoarg$.  There exists $(\beta, \beta')$ that witnesses $G$ is a computable functor.

We construct $F$, $\calE_0$, $\calE_1$ so that $(F, G, \calE_0, \calE_1)$ is a 
computable equivalence of 
$\Cat_0$ and $\Cat_1$.
We simultaneously construct $\alpha$, $\alpha'$, 
$\gamma_0$, $\gamma_1$ so that 
$(\alpha, \alpha',  \beta, \beta', \gamma_0, \gamma_1)$
witnesses $(F, G, \calE_0, \calE_1)$ is a computable equivalence of categories.

\subsubsection{Construction of $F$, $\alpha$, $\alpha'$}\label{subsubsec:CnstrF}

Suppose $A = (\calG, u)^\#$ is an object of $\Cat_0$, and let $x$ be an index of $A$. 
By the uniformity of Lemma \ref{lm:CompUntlCertDim}, it is possible to compute an index $\alpha_0(x)$
of a computable unital certificate of dimensionality $(m_s, u_s, \epsilon_s, \phi_s)_{s \in \N}$
of $A$.  

For each $s \in \N$, let $H_s$ be the multiset of positive integers so that the number of 
times an integer appears in $H_s$ is the number of times it appears in $u_s$. 
Set $\boldA_s = \bigoplus_{n \in H_s} M_n (\C)$.  It is now 
 possible to compute for each $s$ a unital $*$-homomorphism 
$\sigma_s : \boldA_s \rightarrow \boldA_{s+1}$ so that 
$\eta_{H_s, H_{s+1}}(K_0(\sigma_s)) = \phi_s$.  
From $x$, it is possible to compute an index $\alpha_1(x)$ of 
$(\boldA_s, \sigma_s)_{s \in \N}$.
By the uniformity of Corollary \ref{cor:CompIndLimAlg2}, it is now possible to compute an index 
$\alpha_2(x)$ of a computable inductive limit $(\boldA^\dagger, (\psi_s)_{s \in \N})$
of $(\boldA_s, \sigma_s)_{s \in \N}$.  
From $\alpha_2(x)$, we can compute an index $\alpha(x)$ of 
$\boldA^\dagger$.  We declare $\boldA^\dagger$ to be a value of $F(A)$.

Now, for each $j \in \{0,1\}$, suppose $A_j = (\calG_j, u_j)^\#$ is an object 
of $\Cat_0$, and let $x_j$ be an index of $A_j$.  Assume 
$(x_0, x_1, y)$ is an index of a morphism $\delta$ from $A_0$ to $A_1$. 
Thus, $x_j$ is an index of $A_j$, and $y$ is a $(A_0, A_1)$ index of $\delta$.
Assume that for $j \in \{0,1\}$, $\alpha_0(x_j)$ indexes 
$(m_{j,s}, u_{j,s}, \epsilon_{j,s}, \phi_{j,s})_{s \in \N}$.
Thus, $((\calG_j, u_j)^\#, (\epsilon_{j,s})_{s \in \N})$ is a computable 
inductive limit of $((\Z^{m_{j,s}}, u_{j,s}), \phi_{j,s})_{s \in \N}$. 
Assume $\alpha_1(x_j)$ indexes $(\boldA_{j,s}, \sigma_{j,s})_{s \in \ N}$, 
and assume $\alpha_2(x_j)$ indexes $(\boldA_j^\dagger, (\psi_{j,s})_{s \in \N})$.
Thus, $\boldA_j^\dagger$ is indexed by $\alpha(x_j)$ for $j \in \{0,1\}$. 

For the moment, fix $j \in \{0,1\}$.  We observe that 
$((\calG_j, u_j)^\#, (\epsilon_{j,s} \circ \zeta_s)_{s \in \N})$
is a computable inductive limit of $(\uK{\boldA_{j,s}}, K_0(\sigma_{j,s}))_{s \in \N}$
as is $(\uK{\boldA^\dagger_j}, (K_0(\psi_{j,s}))_{s \in \N})$.
Hence, it is now possible to compute an index $\gamma_0(x_j)$ of 
the reduction $\lambda_j$ from $(\uK{\boldA^\dagger_j}, (K_0(\psi_{j,s}))_{s \in \N}$
to $((\calG_j, u_j)^\#, (\epsilon_{j,s} \circ \zeta_s)_{s \in \N})$ 
which is an isomorphism from $\uK{\boldA^\dagger}$ to $(\calG_j, u_j)^\#$. 
We can now compute a $(\uK{\boldA_0^\dagger}, \uK{\boldA_1^\dagger})$ 
index $\alpha_0'(x_0, x_1, y)$ of $\rho:=\lambda^{-2}\delta\lambda_1$.
By Lemma \ref{lm:CompHm}, it is now possible to compute an
$(\boldA_0^\dagger, \boldA_1^\dagger)$ 
index $\alpha'(x_0, x_1, y)$ of a unital $\star$-homomorphism 
$\tau$ so that $K_0(\tau) = \rho$.
We declare $\tau$ to be a value of $F(\delta)$.

\subsubsection{Construction of $\calE_j$ and $\gamma_j$}

In the course of the above process, we constructed $\gamma_0$ and $\calE_0$
so that $\calE_0$ is a natural isomorphism from $GF$ to $\Id_{\Cat_0}$; 
namely declare $\lambda_j$ to be a value of $\calE_0(A_j)$.

We now construct $\gamma_1$ and $\calE_1$.  Assume the notation of Section \ref{subsubsec:CnstrF}.
To begin, suppose $z$ is a unital index of a c.e. presentation $\boldA^\#$ of a unital AF algebra.
Let $B = \boldA^\#$.  Set $x = \beta(z)$, and set 
$A = (\calG, u)^\# = G(B) = \uK{\boldA^\#}$.  
We now consider $\boldA^\dagger = F(A)$ which is indexed by $\alpha(x) = \alpha(\beta(z))$. 
We make two observations.  First, by construction 
$(A, (\epsilon_s)_{s \in \N})$ is a computable inductive limit of 
$((\Z^{m_s}, u_s), \phi_s)_{s \in \N}$.  
Second, also by construction, $(\uK{\boldA^\dagger}, (K_0(\psi_s) \circ \zeta_{F_s}^{-1})_{s \in \N})$
is also a computable inductive limit of this inductive sequence. 
Thus, we can compute the reduction $\lambda$ of 
$(\uK{\boldA^\dagger}, (K_0(\psi_s) \circ \zeta_{F_s}^{-1})_{s \in \N})$ to 
$(A, (\epsilon_s)_{s \in \N})$ which is also an isomorphism 
from $\uK{\boldA^\dagger}$ to $\uK{\boldA^\#}$.  By the uniformity of 
Main Theorem \ref{mthm:CompIso}, 
it is now possible to compute an index $\gamma_1(z)$ of a unital $\star$-isomorphism 
$\tau$ from $\boldA^\dagger$ to $\boldA^\#$, and we declare $\tau$ to be a value 
of $\calE_1(B)$.

As the morphisms of $\Cat_1$ are approximate unitary equivalence classes, 
it follows that $\calE_1$ is a natural isomorphism of $FG$ and $\Id_{\Cat_1}$ (see, for example, \cite[Exercise 7.6]{Rordam.Larsen.Laustsen.2000}).

\section{Applications to index set and isomorphism problems}\label{sec:IndxIso}

We now use our main theorems to resolve some questions concerning index set 
and isomorphism problems for AF and UHF algebras.  We first discuss the meaning of 
these kinds of problems and how they relate to classification problems in mathematics. 

Suppose $\calC$ is some class of mathematical structures.  The 
\emph{index set problem} for $\calC$ is the complexity of determining if 
a natural number indexes a computable (or c.e.) presentation of 
a structure in $\calC$.  This models how hard it is to discern if a 
given structure belongs to $\calC$. 
Here, complexity of a set is measured by its position in 
the arithmetical hierarchy, which measures the number 
of alternating quantifiers required to define the set. 
For example, a $\Pi^0_2$ set is one that can be defined by an expression of the 
form $\forall y \exists z R(x,y,z)$ where $R$ is a computable predicate. 
$\Pi_2^0$ complete sets are $\Pi_2^0$ sets that are `as hard as possible'.
We refer the reader to standard sources such as \cite{Cooper.2004} 
for a more precise definition (which can also be gleaned from the proofs below). 

We let $(\phi_e)_{e \in \N}$ be an effective enumeration of all computable partial 
functions from $\N$ into $\N$.  The set $\Tot$ is defined to be 
$\{e \in \N\ :\ \dom(\phi_e) = \N\}$.  It is well-known that $\Tot$ is 
$\Pi^0_2$ complete.  Let $\phi_{e,s}(x) = \phi_e(x)$ if the computation of
$\phi_e$ on $x$ halts in at most $s$ steps; otherwise $\phi_{e,s}(x)$ is undefined.

We begin with the index problem for all \cstar-algebras. 
The proof of the following is very similar to the proof of \cite[Theorem 2.6]{Brown.McNicholl.Melnikov.2020}; we leave the details to the reader.

\begin{theorem}\label{thm:IndxCstar}
The set of all indexes for c.e. presentation of \cstar-algebras 
is $\Pi^0_2$ complete as is the set of all indexes of 
computable presentations of \cstar-algebras. 
\end{theorem}

\begin{theorem}\label{thm:IndxAF}
The index set problem for AF algebras is $\Pi^0_2$ complete.
\end{theorem}

\begin{proof}
We say that $(a_{i,j}^s)_{s,i,j}$ is a \emph{type-$(n_1, \ldots, n_m)$ system}
if for each $s \in \{1, \ldots, m\}$, $(a_{i,j}^s)_{i,j}$ is an 
$n_s \times n_s$ array.

Suppose $e$ is an index of a c.e. presentation $\boldA^\#$ of a \cstar-algebra.  Let $m$, $n_1$, $\ldots$, $n_m$
be positive integers.  
When $m, n_1, \ldots, n_m$ are positive integers, let $\calS_{(k,n_1, \ldots, n_m)}[\boldA^\#]$
denote the set of all type-$(n_1, \ldots, n_m)$ arrays $(g^s_{i,j})_{s,i,j}$ of 
generated points of $\boldA^\#$ for which there exists a type-$(n_1, \ldots, n_m)$ 
matricial generating system of $\boldA$ so that $\norm{g^s_{i,j} - f^s_{i,j}}< 2^{-k}$
for all $s,i,j$.  It follows from Proposition \ref{prop:MutOrthoGen} 
that $\calS_{(k,n_1, \ldots, n_m)}[\boldA^\#]$
is c.e. uniformly in $e, k, m, n_1, \ldots, n_m$. 

Now, when $M \in \N$, let $\calS_{k,M, n_1, \ldots, n_m}'[\boldA^\#]$ denote the set
of all finite sets $F$ of generated points of $\boldA^\#$ so that for each $a \in F$, 
there is a type-$(n_1, \ldots, n_m)$ system $(\alpha^s_{i,j})_{s,i,j}$ 
of rational scalars so that $\sum_{s,i,j} |\alpha^s_{i,j}| < 2^{-k}$ and so that 
$\norm{a - \sum_{s,i,j} \alpha^s_{i,j} g^s_{i,j}} < 2^{-k}$.  
It follows that $\calS_{k,M, n_1, \ldots, n_m}'[\boldA^\#]$ is c.e. uniformly in 
$e,k,M, n_1, \ldots, n_m$.

It now follows from \cite[Proposition 7.2.2]{Rordam.Larsen.Laustsen.2000} that $\boldA^\#$ is AF if and only if 
for every finite set $F$ of generated points of $\boldA^\#$ and every $k \in \N$, 
there exists $M, k_1 \in \N$ so that $F \in \calS_{k_1,M, n_1, \ldots, n_m}'[\boldA^\#]$
and so that $2^{-k_1}(1 + M) < 2^{-k}$.  Thus, the index set problem for AF algebras is a $\Pi^0_2$ set.

For the lower bound, we create a computable reduction from $\Tot$ to the index set problem for AF algebras.  Let $f$ be a partial computable function.  For each $n \in \mathbb{N}$, let $I_n$ be the union of the intervals removed in step $n$ of the middle-thirds construction of the Cantor set.  We now carry out the following procedure.  Begin with $A_0 = [0,1]$.  At step $n$, begin the computation of $f(n)$ and for each $m < n$ advance the computation of $f(m)$ by one step.  If computations $m_1, \ldots, m_k$ terminate at step $n$, then set $A_n = A_{n-1} \setminus \bigcup_{j=1}^kI_{m_j}$ (if no computations terminate at step $N$, set $A_n = A_{n-1}$).  Let $A = \bigcap_{n=1}^{\infty}A_n$.  Then $A$ is a computably compact space (see \cite[Definition 2.5]{Downey.Melnikov.2023}), so the abelian \cstar-algebra $C(A)$ has a computable presentation \cite{fox2022computable}.  It is fairly easy to see (on the basis of, say, \cite[Proposition 7.2.2]{Rordam.Larsen.Laustsen.2000}) that $C(A)$ is an AF algebra if and only if $A$ is totally disconnected. By construction, $A$ is totally disconnected if and only if $f$ is a total function.
\end{proof}

The \emph{CAR algebra} is the \cstar-algebra obtained via the inductive limit of 
$(M_{2^n}(\C), \phi_n)_{n \in \N}$, where $\phi_n$ is the diagonal embedding of 
$M_{2^n}(\C)$ into $M_{2^{n + 1}}(\C)$.
A \cstar-algebra is \emph{simple} if it has no non-trivial proper closed two-sided ideals.

\begin{lemma}\label{lm:IndxAF}
From $e \in \N$, it is possible to compute an index of a c.e. presentation 
$\boldA^\#$ of an AF algebra so that the following are equivalent: 
\begin{enumerate}
    \item{$e \in \Tot$.}
 
    \item{$\boldA$ is isomorphic to the CAR algebra}.

    \item{$\boldA$ is UHF.}

    \item{$\boldA$ is simple.}
\end{enumerate}
\end{lemma}

\begin{proof}
Let $e \in \N$. 
For every $s \in \N$, let 
$m_{e,s} = \max\{y \in \N\ :\ \forall x < y\ x \in \dom(\phi_{e,s})\}$.
Thus, $m_{e,0} = 0$.  
Let:
\begin{eqnarray*}
\mathfrak{p}_0 & = & \{0\}\\
\mathfrak{p}_{s+1} & = & \left\{ 
\begin{array}{cc}
\{0\} & m_{e,s} < m_{e,s+1} \\
\{0,1\} & m_{e,s} = m_{e,s+1} \\
\end{array}
\right.
\end{eqnarray*}
We define a labeled Bratteli diagram $\BratD_e$ as follows. 
Set $V_{\BratD_e} = \bigcup_{s \in \N} \{s\} \times \mathfrak{p}_s$, and let 
$L_{\BratD_e}(s,t) = s$.
If $(s,1) \not \in V_{\BratD_e}$, then set $E_{\BratD_e}((s,0), v) = 1$ 
for all $v \in L^{-1}[\{s + 1\}]$.  
Suppose $(s,1) \in V_{\BratD_e}$.  Set 
$E_{\BratD_e}((s,0), (s+1, 0)) = E_{\BratD_e}((s,1), (s + 1, j)  =  1$ 
where $j$ is the largest number so that $(s+1,j) \in V_{\BratD_e}$.
For all other pairs $(v,v')$ of vertices we set $E_{\BratD_e}(v,v') = 0$.
We set $\Lambda_{\BratD_e}(0,0) = 1$.  For all other vertices $v$, 
we set $\Lambda_{\BratD_e}(v) = \sum_{v'} \Lambda_{\BratD_e}(v')$ where 
$v'$ ranges over all vertices so that $E_{\BratD_e}(v',v) > 0$.

It is now straightforward to construct a computable presentation $\BratD_e^\#$
of $\BratD_e$.  Furthermore, it is straightforward to construct an AF algebra $\boldA$
whose labeled Bratteli diagram is $\BratD_e$.  By the uniformity of Main Theorem \ref{mthm:ClsCompPrsntAF}, it is now possible to compute an index of a c.e. 
presentation $\boldA^\#$ of $\boldA$.  

Suppose $\phi_e$ is total.  
In this case, by appropriately telescoping the diagram $\BratD_e$ we obtain a diagram where every level has exactly one vertex and between each pair of adjacent levels there are exactly two edges.  This Bratteli diagram corresponds to the CAR algebra (see \cite[Example III.2.4]{Davidson.1996}), which is UHF and, therefore, simple (\cite[Corollary III.5.3]{Davidson.1996}).

On the other hand, suppose $\phi_e$ is not total, and let $s_0$ be the least number 
so that $m_{e,s_0} = m_{e,s}$ for all $s \geq s_0$. 
This entails that for all $s > s_0$, $P_{\BratD_e}((s_0 + 1, 0), (s, 1)) = 0$.
This implies that the associated AF algebra is not simple (see \cite[Corollary III.4.3]{Davidson.1996}) and therefore is not a UHF algebra (and hence is not $\star$-isomorphic to the CAR
algebra).
\end{proof}

\begin{theorem}\label{thm:IndxUHF}
The index set problem for UHF algebras is $\Pi^0_2$ complete.
\end{theorem}

\begin{proof}
To see that the index set problem for UHF algebras is a $\Pi_2^0$-set, argue as in the proof of Theorem \ref{thm:IndxAF}, using that a separable unital \cstar-algebra is UHF if and only if it is locally matricial (see \cite[Theorem 1.13]{Glimm.1960}).

For the lower bound, Lemma \ref{lm:IndxAF} gives a many-one reduction from $\Tot$ to the index set problem for UHF algebras.
\end{proof}

\begin{theorem}
The index set problem for simple AF algebras is $\Pi^0_2$ complete.
\end{theorem}

\begin{proof}
For the upper bound, by Theorem \ref{thm:IndxAF}, expressing that an algebra is AF is $\Pi^0_2$.  
Suppose $e$ is an index of a c.e. presentation $\boldA^\#$ of a \cstar-algebra. 
By Main Theorem \ref{mthm:ClsCompPrsntAF}, from $e$ it is possible to compute a computable
presentation $\BratD^\#$ of a labelled Bratteli diagram of $\boldA^\#$.
By \cite[Corollary III.4.3]{Davidson.1996} $\boldA$ is simple
if and only if for every $n \in \N$ and every $v \in L^{-1}[\{n\}]$, there is an $m > n$ such that there is a path from $v$ to every vertex in $L^{-1}[\{m\}$.  This is clearly a $\Pi^0_2$ assertion.

For the lower bound, Lemma \ref{lm:IndxAF} gives a computable reduction from $\Tot$ to the index set problem for simple AF algebras.
\end{proof}

We now turn to isomorphism problems.  The isomorphism problem for a class $\calC$ of structures
is the problem of determining if two natural numbers index presentations of isomorphic 
structures in the class.  Isomorphism problems thus correspond to sets of pairs of natural
numbers, and their complexity is still measured by their position in the arithmetical hierarchy.
These problems model the complexity of the classification problem for $\calC$.  
We begin with the isomorphism problem for UHF algebras.

\begin{theorem}\label{thm:IsoUHF}
The isomorphism problem for UHF algebras is $\Pi^0_2$ complete.
\end{theorem}

\begin{proof}
Suppose $\boldA^\#$ and $\boldB^\#$ are computable presentations of UHF algebras. 
For the upper bound, first compute c.e. presentations of their $K_0$ groups.  Then compute
certificates of dimensionality for their $K_0$ groups.  Then lower semicompute 
their supernatural numbers from these certificates.  Testing equality of the supernatural
numbers is easily seen to be $\Pi^0_2$.  

For the lower bound, we show that there is a many-one reduction from $\Tot$ to the isomorphism problem for UHF algebras.  Let $(p_n)_{n \in \N}$ be the increasing enumeration of the prime numbers. 
For each $n \in \dom(\phi_e)$, let $\calE_0(n) = \calE_1(n) =$ the least $s \in \N$ so that 
$n \in \dom(\phi_{e,s})$. When $n \not \in \dom(\phi_e)$, let 
$\calE_0(n) = \infty$ and let $\calE_1(n) = 0$.  
It follows that $\calE_0$, $\calE_1$ are lower semi-computable uniformly in $e$.
By the results in \cite{UHFPaper}, it is now possible to compute for each $j \in \{0,1\}$
a c.e. presentation $\boldA_j^\#$ of a UHF algebra whose supernatural number is $\calE_j$.
We now note that $\boldA_0$ is $\star$-isomorphic to $\boldA_1$ if and only if 
$\calE_0 = \calE_1$ (see, for example, \cite[Section 7.4]{Rordam.Larsen.Laustsen.2000}). However, by construction, $\calE_0 = \calE_1$ if and only 
if $\phi_e$ is total.
\end{proof}

While the complexity of the index set problem for AF algebras is fairly low, 
the complexity of the isomorphism problem for this class is as high as possible, 
namely, it is $\Sigma^1_1$ complete.  The $\Sigma^1_1$ sets are those that 
can be defined by expressions of the form 
$\exists f \forall y R(x,f,y)$ where $f$ ranges over all functions from $\N$ into 
$\N$ (that is, over Baire space) and $R$ is a computable predicate.  
Every isomorphism problem is at worst $\Sigma^1_1$ (search through all possible
functions and check if any are isomorphisms).  If an isomorphism problem is 
$\Sigma^1_1$ complete, this is interpretated as saying that no effective 
classification is possible.

\begin{theorem}\label{thm:IsoAF}
The isomorphism problem for commutative AF algebras is $\Sigma_1^1$ complete.
\end{theorem}

\begin{proof}
For each Boolean algebra $B$, let $X_B$ denote the Stone space of $B$.
From an index of a computable presentation $B^\#$ of a Boolean algebra, 
it is possible to compute an index of a computably compact presentation 
$X_B$ of the Stone space of $B$ (see \cite{Bazhenov.2023}).
From this index, it is possible to compute an index of a presentation of 
$C(X_B)$ (see \cite{burton2024computable}, \cite{fox2022computable}).
By Gelfand Duality, $C(X_{B_0})$ is $\star$-isomorphic to 
$C(X_{B_1})$ if and only if $X_{B_0}$ is homeomorphic to $X_{B_1}$.  
By the Stone Duality Theorem, $X_{B_0}$ is homeomorphic to $X_{B_1}$ if and only if 
$B_0$ is isomorphic to $B_1$.

Thus, we have defined a many-one reduction from the isomorphism problem for Boolean algebras to the isomorphism problem for commutative AF algebras; 
since the former problem is known to be $\Sigma_1^1$ complete (see \cite[Theorem 4.4(d)]{GoncharovKnight.2002}), so is the latter.
\end{proof}

\begin{corollary}\label{cor:IsoUntlDimGrp}
The isomorphism problem for unital dimension groups is $\Sigma_1^1$ complete. 
\end{corollary}

\begin{corollary}\label{cor:EqvLblBrat}
The equivalence
problem for labeled Bratteli diagrams is $\Sigma_1^1$ complete.  
\end{corollary}

\begin{proof}
Suppose $e$ is an index of a c.e. presentation $(\calG, u)^\#$ of a 
unital dimension group. 
By the uniformity of Lemma \ref{lm:CompUntlCertDim}, it is possible to compute from $e$ an index 
of a unital certificate of dimensionality $(n_s, u_s, \mu_s, \phi_s)_{s \in \N}$
of $(\calG, u)^\#$.  Define a Bratteli diagram $\BratD_e$ by setting:
\begin{eqnarray*}
V_{\BratD_e} & = & \{(s+1, j)\ :\ s \in \N\ \wedge\ j \in \{1, \ldots, n_s\} \} \cup \{(0, 1)\}\\
L_{\BratD_e}(s,j) & = & s\\
E_{\BratD_e}((s+1,j), (s+2, j')) & = & \phi_s(e^{n_s}_j)(j')\\
E_{\BratD_e}((0,1), (1, j)) & = & u_1(j)
\end{eqnarray*}
It is straightforward to compute an index of a computable presentation of $\BratD_e^\#$ 
from $e$.  

Now, suppose that for each $j \in \{0,1\}$, $e_j$ is an index of a c.e. presentation 
$(\calG_j, u_j)^\#$ of a unital dimension group.  It is well-known 
that $(\calG_0, u_0)$ is isomorphic to $(\calG_1, u_1)$ if and only if 
$\BratD_{e_0}$ is equivalent to $\BratD_{e_1}$ (see \cite{Durand.Perrin.2022}).
The corollary now follows from Corollary \ref{cor:IsoUntlDimGrp}.
\end{proof}

The previous two theorems show that the isomorphism problem for AF algebras is substantially more difficult than the isomorphism problem for UHF algebras, confirming the intuition that the former class is indeed more complicated than the latter.  That being said, since UHF algebras are simple while commutative ones are not, a more satisfying confirmation of this aforementioned intuition would be to show that the isomorphism problem for the class of \emph{simple} AF algebras is also $\Sigma_1^1$ complete.  This raises the following question:

\begin{question}
    What is the complexity of the isomorphism problem for simple AF algebras?
\end{question}

\section*{Acknowledgments} \label{sec:ack} 
Some of the results in this paper were obtained during a SQuaRE at the American Institute of Mathematics.  We thank AIM for providing a supportive and mathematically rich environment for collaboration.  We also thank Russell Miller for several helpful discussions.
\bibliographystyle{amsplain}
\bibliography{paperbib}

@book {Durand.Perrin.2022,
    AUTHOR = {Durand, Fabien and Perrin, Dominique},
     TITLE = {Dimension groups and dynamical systems---substitutions,
              {B}ratteli diagrams and {C}antor systems},
    SERIES = {Cambridge Studies in Advanced Mathematics},
 PUBLISHER = {Cambridge University Press, Cambridge},
      YEAR = {2022},
     PAGES = {vii+584}
}

@book {Goodearl.1986,
    AUTHOR = {Goodearl, K. R.},
     TITLE = {Partially ordered abelian groups with interpolation},
    SERIES = {Mathematical Surveys and Monographs},
    VOLUME = {20},
 PUBLISHER = {American Mathematical Society, Providence, RI},
      YEAR = {1986},
     PAGES = {xxii+336}
}

@unpublished{burton2024computable,
  title={Computable {G}elfand Duality},
  author={Burton, Peter and Eagle, Christopher J and Fox, Alec and Goldbring, Isaac and Harrison-Trainor, Matthew and McNicholl, Timothy H and Melnikov, Alexander and Thewmorakot, Teerawat},
  note={To appear in Proceedings of the American Mathematical Society; 
        arXiv preprint arXiv:2402.16672}
}

@article {Downey.Melnikov.2023,
    AUTHOR = {Downey, Rodney G. and Melnikov, Alexander G.},
     TITLE = {Computably compact metric spaces},
   JOURNAL = {Bull. Symb. Log.},
  FJOURNAL = {The Bulletin of Symbolic Logic},
    VOLUME = {29},
      YEAR = {2023},
    NUMBER = {2},
     PAGES = {170--263},
      ISSN = {1079-8986,1943-5894},
   MRCLASS = {03D45 (03C57 03D78)},
  MRNUMBER = {4616053},
MRREVIEWER = {Nikolay\ Bazhenov},
       DOI = {10.1017/bsl.2023.16},
       URL = {https://doi.org/10.1017/bsl.2023.16},
}

@book{Cooper.2004,
	Author = {Cooper, S. Barry},
	Publisher = {Chapman \& Hall/CRC, Boca Raton, FL},
	Title = {Computability theory},
	Year = {2004}}

@book{Davidson.1996,
    AUTHOR = {Davidson, Kenneth R.},
     TITLE = {{$\rm C^*$}-algebras by example},
    SERIES = {Fields Institute Monographs},
    VOLUME = {6},
 PUBLISHER = {American Mathematical Society, Providence, RI},
      YEAR = {1996},
     PAGES = {xiv+309}
}

@unpublished{EMST,
author={Franklin, Johanna F. and Goldbring, Isaac. and McNicholl, Timothy H.},
title="Effective metric structure theory",
publisher="Springer-Verlag",
note="To appear."
}

@book {Strung.2021,
    AUTHOR = {Strung, Karen R.},
     TITLE = {An introduction to {${\rm C}^*$}-algebras and the
              classification program},
    SERIES = {Advanced Courses in Mathematics. CRM Barcelona},
      NOTE = {Edited and with a foreword by Francesc Perera},
 PUBLISHER = {Birkh\"auser/Springer, Cham},
      YEAR = {[2021] \copyright 2021},
     PAGES = {xiii+322}
}

@article {FarahEtAll.2021,
    AUTHOR = {Farah, Ilijas and Hart, Bradd and Lupini, Martino and Robert,
              Leonel and Tikuisis, Aaron and Vignati, Alessandro and Winter,
              Wilhelm},
     TITLE = {Model theory of {$\rm C^*$}-algebras},
   JOURNAL = {Mem. Amer. Math. Soc.},
  FJOURNAL = {Memoirs of the American Mathematical Society},
    VOLUME = {271},
      YEAR = {2021},
    NUMBER = {1324},
     PAGES = {viii+127}
}

@article{fox2022computable,
    AUTHOR = {Fox, Alec},
     TITLE = {Computable presentations of {${\rm C}^*$}-algebras},
   JOURNAL = {J. Symb. Log.},
  FJOURNAL = {The Journal of Symbolic Logic},
    VOLUME = {89},
      YEAR = {2024},
    NUMBER = {3},
     PAGES = {1313--1338}
}

@article {Brown.McNicholl.Melnikov.2020,
    AUTHOR = {Brown, Tyler A. and McNicholl, Timothy H. and Melnikov,
              Alexander G.},
     TITLE = {On the complexity of classifying {L}ebesgue spaces},
   JOURNAL = {J. Symb. Log.},
  FJOURNAL = {The Journal of Symbolic Logic},
    VOLUME = {85},
      YEAR = {2020},
    NUMBER = {3},
     PAGES = {1254--1288}
     }

@unpublished{FoxGoldbringHart.2024+,
author = {Alec Fox and Isaac Goldbring and Bradd Hart},
title = {Locally universal {$\rm C^*$}-algebras with computable presentations},
note="Preprint available at https://arxiv.org/pdf/2303.02301"
}

@article {Glimm.1960,
    AUTHOR = {Glimm, James G.},
     TITLE = {On a certain class of operator algebras},
   JOURNAL = {Trans. Amer. Math. Soc.},
  FJOURNAL = {Transactions of the American Mathematical Society},
    VOLUME = {95},
      YEAR = {1960},
     PAGES = {318--340}
}

@unpublished{Goldbring.2024+,
  title={Computably strongly self-absorbing {$\rm C^*$}-algebras},
  author={Goldbring, Isaac},
  note={arXiv preprint arXiv:2409.18834},
  year={2024}
}

@article{mcnicholl2024evaluative,
    AUTHOR = {McNicholl, Timothy H.},
     TITLE = {Evaluative presentations},
   JOURNAL = {J. Logic Comput.},
  FJOURNAL = {Journal of Logic and Computation},
    VOLUME = {35},
      YEAR = {2025},
    NUMBER = {5},
     PAGES = {Paper No. exaf036, 11}
}

@book {Rordam.Larsen.Laustsen.2000,
    AUTHOR = {R{\o}rdam, M. and Larsen, F. and Laustsen, N.},
     TITLE = {An introduction to {$K$}-theory for {$\rm C^*$}-algebras},
    SERIES = {London Mathematical Society Student Texts},
    VOLUME = {49},
 PUBLISHER = {Cambridge University Press, Cambridge},
      YEAR = {2000},
     PAGES = {xii+242}
}

@unpublished{UHFPaper,
    Author = {Eagle, C.J. and Goldbring, I. and McNicholl, T.H. and Miller, R.},
    Title = {Computable {$K$}-theory for {C}*-algebras: {UHF} algebras},
Note = {To appear in Transactions of the American Mathematical Society.  Preprint available at https://arxiv.org/abs/2501.08526}
}

@article {Bratteli.1972,
    AUTHOR = {Bratteli, Ola},
     TITLE = {Inductive limits of finite dimensional {$C\sp{\ast}
              $}-algebras},
   JOURNAL = {Trans. Amer. Math. Soc.},
  FJOURNAL = {Transactions of the American Mathematical Society},
    VOLUME = {171},
      YEAR = {1972},
     PAGES = {195--234}
}

@article{Bazhenov.2023,
    AUTHOR = {Bazhenov, Nikolay and Harrison-Trainor, Matthew and Melnikov,
              Alexander},
     TITLE = {Computable {S}tone spaces},
   JOURNAL = {Ann. Pure Appl. Logic},
  FJOURNAL = {Annals of Pure and Applied Logic},
    VOLUME = {174},
      YEAR = {2023},
    NUMBER = {9},
     PAGES = {Paper No. 103304, 25}
}

@article{GoncharovKnight.2002,
title = {Computable structure and non-structure theorems},
journal = {Algebra and Logic},
volume = {41},
pages = {351--373},
year = {2002},
author = {Sergey S. Goncharov and Julia F. Knight}
}

@article {Elliott.1976,
    AUTHOR = {Elliott, George A.},
     TITLE = {On the classification of inductive limits of sequences of
              semisimple finite-dimensional algebras},
   JOURNAL = {J. Algebra},
  FJOURNAL = {Journal of Algebra},
    VOLUME = {38},
      YEAR = {1976},
    NUMBER = {1},
     PAGES = {29--44}
}

\end{document}